\documentclass{amsart}
\usepackage{amsfonts,latexsym,enumerate}
\usepackage{amsmath}
\usepackage{amscd}
\usepackage{float,amsmath,amssymb,mathrsfs,bm,multirow,graphics}
\usepackage[dvips]{graphicx}
\usepackage[percent]{overpic}
\usepackage{amsaddr}
\usepackage[numbers,sort&compress]{natbib}
\usepackage{todonotes}

\newcommand{\R}{{\Bbb R}}

\newcommand{\C}{{\Bbb C}}

\newcommand{\proofbegin}{\noindent{\it Proof.\quad}}
\newcommand{\proofend}{\hfill$\Box$\bigskip}

\newcommand{\tr}{\text{\upshape tr\,}}
\newcommand{\diag}{\text{\upshape diag\,}}
\newcommand{\re}{\text{\upshape Re\,}}

\DeclareMathOperator{\dist}{dist}

\def\XXint#1#2#3{{\setbox0=\hbox{$#1{#2#3}{\int}$}
\vcenter{\hbox{$#2#3$}}\kern-.5\wd0}}

\allowdisplaybreaks

\newtheorem{theorem}{Theorem}[section]
\newtheorem{proposition}[theorem]{Proposition}
\newtheorem{corollary}[theorem]{Corollary}
\newtheorem{lemma}[theorem]{Lemma}
\newtheorem{definition}[theorem]{Definition}
\newtheorem{assumption}[theorem]{Assumption}
\newtheorem{remark}[theorem]{Remark}

\newtheorem{RHproblem}[theorem]{RH problem}
\newtheorem{figuretext}{Figure}

\numberwithin{equation}{section}

\usepackage[colorlinks=true]{hyperref}
\hypersetup{urlcolor=blue, citecolor=red, linkcolor=blue}

\input epsf
\title[The ``good'' Boussinesq equation]
{The ``good'' Boussinesq equation: \\a Riemann-Hilbert approach}

\author{C. Charlier and J. Lenells}
\address{Department of Mathematics, KTH Royal Institute of Technology, \\ 100 44 Stockholm, Sweden.}
\email{cchar@kth.se}
\email{jlenells@kth.se}

\begin{document}

\begin{abstract} 
We develop an inverse scattering transform formalism for the ``good'' Boussinesq equation on the line. Assuming that the solution exists, we show that it can be expressed in terms of the solution of a $3 \times 3$ matrix Riemann-Hilbert problem. The Riemann-Hilbert problem is formulated in terms of two reflection coefficients whose definitions involve only the initial data, and it has a form which makes it suitable for the evaluation of long-time asymptotics via Deift-Zhou steepest descent arguments.
\end{abstract}

\maketitle

\noindent
{\small{\sc AMS Subject Classification (2010)}: 35G25, 35Q15, 37K15.}

\noindent
{\small{\sc Keywords}: Spectral analysis, Boussinesq equation, Riemann-Hilbert problem, inverse scattering transform, initial value problem.}


\section{Introduction}
About 150 years ago, the French mathematician Joseph Boussinesq derived an equation for shallow water waves propagating in a rectangular channel \cite{B1872}. In nondimensional units, this equation---now known as the Boussinesq equation---takes the form
\begin{align}\label{badboussinesq}
  u_{tt} - u_{xx} - (u^2)_{xx} - u_{xxxx} = 0,
\end{align}
where $u(x,t)$ is a real-valued function and subscripts denote partial derivatives, see \cite[Eq. (26)]{B1872}. Equation (\ref{badboussinesq}) also describes ion sound waves in a plasma \cite{S1974} and lattice waves in the continuum approximation of the Fermi-Pasta-Ulam problem \cite{ZK1965, Z1974}.

The sign of the term $(u^2)_{xx}$ in (\ref{badboussinesq}) can be reversed by replacing $u$ by $-u$. Moreover, replacing $u$ by $u - 1$ switches the sign of the $u_{xx}$ term. The sign of the $u_{xxxx}$ term is more fundamental. In fact, since the $u_{tt}$ and $u_{xxxx}$ terms in (\ref{badboussinesq}) have opposite signs, equation (\ref{badboussinesq}) is linearly ill-posed and is therefore sometimes referred to as the ``bad'' Boussinesq equation. This is in contrast to the ``good'' Boussinesq equation
\begin{align}\label{goodboussinesq}
  u_{tt} - u_{xx} + (u^2)_{xx} + u_{xxxx} = 0,
\end{align}
in which the $u_{tt}$ and $u_{xxxx}$ terms have the same sign. Equation (\ref{goodboussinesq}) models the nonlinear dynamics of waves in a weakly dispersive medium and is also known as the ``nonlinear string equation'' \cite{FST1983}. 

Both the good and the bad Boussinesq equations are integrable. In fact, explicit formulas for the multisoliton solutions of (\ref{badboussinesq}) were found by Hirota using the bilinear transformation method \cite{H1973}, a Lax pair was presented in \cite{Z1974}, and an inverse scattering scheme was outlined in \cite{ZS1974}. Hirota's bilinear method was further used in \cite{CD2017} to derive rational solutions of (\ref{badboussinesq}). 
First results on the well-posedness of (\ref{goodboussinesq}) were presented in \cite{BS1988} where it was shown to be locally well-posed for initial data $u(x,0)$ and $u_t(x,0)$ in $H^s(\R) \times H^{s-2}(\R)$ with $s > 5/2$. The initial-boundary value problem for (\ref{goodboussinesq}) on the half-line was studied in \cite{HM2015}.

The $u_{xx}$ terms in (\ref{badboussinesq}) and (\ref{goodboussinesq}) can be removed by replacing $u$ by $u - \frac{1}{2}$ and $u + \frac{1}{2}$, respectively, and the two equations then reduce to
\begin{align}\label{boussinesqnouxx}
  u_{tt}  + \sigma((u^2)_{xx} + u_{xxxx}) = 0, \qquad \sigma = \pm 1,
\end{align}
where $\sigma =1$ and $\sigma = -1$ correspond to the good and bad versions, respectively. An inverse scattering transform formalism for the solution on the line of (\ref{boussinesqnouxx}) with $\sigma = -1$ has been developed by Deift, Tomei, and Trubowitz \cite{DTT1982}. 

In this paper, we develop an inverse scattering transform formalism for the solution of (\ref{boussinesqnouxx}) with $\sigma = 1$. 
For later convenience, we will rescale the coefficients in (\ref{boussinesqnouxx}) slightly and consider the following equation:
\begin{align}\label{boussinesq}
& u_{tt} + \frac{4}{3} (u^2)_{xx} + \frac{1}{3} u_{xxxx} = 0.
\end{align}
Following \cite{Z1974, DTT1982}, equation (\ref{boussinesq}) can be rewritten as the system
\begin{align}\label{boussinesqsystem}
& \begin{cases}
 v_{t} + \frac{1}{3}u_{xxx} + \frac{4}{3}(u^{2})_{x} = 0,
 \\
 u_t = v_x,
\end{cases}
\end{align}
which is equivalent to (\ref{boussinesq}) provided that the initial data $u_1(x) := u_t(x,0)$ satisfy
\begin{align}\label{u1zeromean}
\int_\R u_1(x) dx = 0,
\end{align}
see Lemma \ref{systemlemma} for details. The system (\ref{boussinesqsystem}) admits the Lax pair representation  \cite{Z1974}
\begin{align}\label{operatorLaxpair}
\mathsf{L}_t + [\mathsf{L}, \mathsf{A}] = 0,
\end{align}
where the Lax operators $\mathsf{L}$ and $\mathsf{A}$ are defined by
\begin{align}\label{LAdef}
\mathsf{L} = \partial_x^3 + 2u\partial_x + u_x + v, \qquad
\mathsf{A} = \partial_x^2 + \frac{4u}{3}.
\end{align}

Assuming that the solution of (\ref{boussinesqsystem}) exists, we will show that it can be expressed in terms of the solution of a $3 \times 3$ matrix Riemann-Hilbert (RH) problem whose jump matrix is expressed in terms of two reflection coefficients $r_1(k)$ and $r_2(k)$. The fact that the RH problem involves $3 \times 3$ matrices is related to the fact that the operator $\mathsf{L}$ in (\ref{LAdef}) is third order. 
As usual in the implementation of the inverse scattering transform, the reflection coefficients are defined in terms of the initial data via linear integral equations. For simplicity, we will restrict ourselves to smooth solitonless solutions which have rapid decay as $|x| \to \infty$. Our main results are stated in Theorem \ref{r1r2th} and Theorem \ref{RHth} and can be summarized as follows: 
\begin{enumerate}[$-$]

\item Theorem \ref{r1r2th} studies the map from the initial data $\{u(x,0), v(x,0)\}$ to the scattering data $\{r_1(k), r_2(k)\}$. In particular, it establishes several properties of the functions $r_1(k)$ and $r_2(k)$, such as their behavior as $k \to 0$. 

\item Theorem \ref{RHth} shows that the solution $\{u(x,t), v(x,t)\}$ of (\ref{boussinesqsystem}) can be recovered from the solution $M(x,t,k)$ of a $3 \times 3$ matrix RH problem via the relations
\begin{align}\label{recoveruv}
\begin{cases}
 \displaystyle{u(x,t) = -\frac{3}{2}\frac{\partial}{\partial x}\lim_{k\to \infty}k\big[(M(x,t,k))_{33} - 1\big],}
	\vspace{.1cm}\\
 \displaystyle{v(x,t) = -\frac{3}{2}\frac{\partial}{\partial t}\lim_{k\to \infty}k\big[(M(x,t,k))_{33} - 1\big].}
\end{cases}
\end{align}
The jump contour $\Gamma$ of this RH problem consists of the three lines $\R \cup \omega \R \cup \omega^2 \R$ where $\omega = e^{2\pi i/3}$, see Figure \ref{Gamma.pdf}, and the jump matrix is given explicitly in terms of $r_1(k)$ and $r_2(k)$, see (\ref{vdef}).
\end{enumerate}
The above theorems are formulated for the system (\ref{boussinesqsystem}). As corollaries, we obtain analogous results for equation (\ref{boussinesq}) provided that $u$ satisfies (\ref{u1zeromean}). 
Moreover, if $u$ satisfies (\ref{boussinesq}) then 
\begin{align}\label{fromboussinesqtogoodboussinesq}
  \tilde{u}(x,t) = \frac{4}{\sqrt{3}}u\bigg(\frac{x}{3^{1/4}},t\bigg) + \frac{1}{2}
\end{align}
satisfies (\ref{goodboussinesq}). Hence, as corollaries, we also obtain results for equation (\ref{goodboussinesq}) under the finite background density assumption that the solution approaches $1/2$ as $x \to \pm \infty$. 
Finite density boundary conditions have been frequently studied for integrable equations, see e.g. \cite{FT2007}.

\begin{figure}
\begin{center}
 \begin{overpic}[width=.6\textwidth]{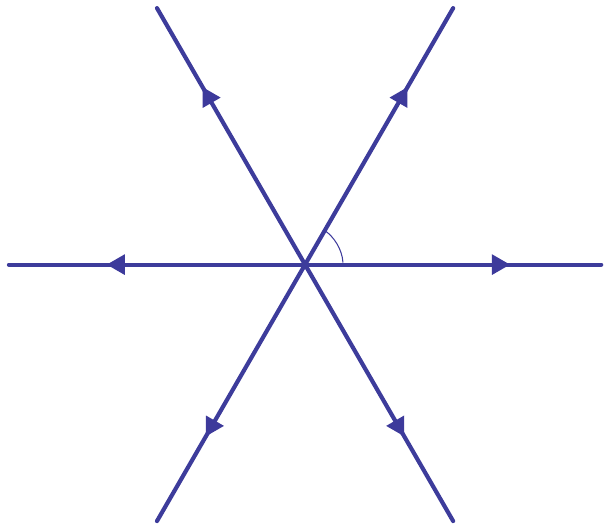}
  \put(101,42.5){\small $\Gamma$}
 \put(56,47){\small $\pi/3$}
 \put(80,60){\small $D_1$}
 \put(48,74){\small $D_2$}
 \put(17,60){\small $D_3$}
 \put(17,25){\small $D_4$}
 \put(48,12){\small $D_5$}
 \put(80,25){\small $D_6$}
  \put(81,38.7){\small $1$}
 \put(67.7,69){\small $2$}
 \put(30,69){\small $3$}
 \put(18,38.7){\small $4$}
 \put(30.5,16){\small $5$}
 \put(67.5,16){\small $6$}
   \end{overpic}
     \begin{figuretext}\label{Gamma.pdf}
       The contour $\Gamma$ and the open sets $D_n$, $n = 1, \dots, 6$, which decompose the complex $k$-plane.
     \end{figuretext}
     \end{center}
\end{figure}

\begin{remark}
The assumption (\ref{u1zeromean}) ensures that the integral $\int_\R u dx$ does not grow linearly but is conserved in time. Indeed, assuming that $u$ has sufficient smoothness and decay and letting $u_0(x) := u(x,0)$, it follows immediately from (\ref{boussinesq}) that 
\begin{align}\label{intuintu1}
\frac{d^2}{dt^2}\int_\R u dx = 0, \quad \text{i.e.} \quad \int_\R u dx = \bigg(\int_\R u_1 dx\bigg)t + \int_\R u_0 dx.
\end{align}
\end{remark}

As usual in the implementation of the inverse scattering transform, the RH solution $M$ appearing in the expressions (\ref{recoveruv}) for $u$ and $v$ is constructed via a spectral analysis of the associated Lax operator $\mathsf{L}$ in (\ref{LAdef}). 
In the case at hand, it turns out that the spectral analysis of $\mathsf{L}$ naturally leads to eigenfunctions which are singular at the origin of the spectral $k$-plane. 
More precisely, we will show that $M$ has a double pole at $k = 0$ for generic initial data (for general initial data, it has {\it at most} a double pole at $k = 0$). 
This has the effect that the RH problem for $M$ has a somewhat unusual form which is singular at the origin. Nevertheless, by prescribing the structure of the behavior both at the origin and at infinity, we can still ensure uniqueness of the solution of the RH problem for $M$. In fact, the handling of the singularity at $k = 0$ is one of the main difficulties in the proof of Theorem \ref{RHth}. 

Apart from the immediate goal of implementing the inverse scattering transform for the good Boussinesq equation (\ref{boussinesq}), one of the main driving forces behind the present work was the larger objective of obtaining detailed asymptotic formulas for the solution of the (good and bad) Boussinesq equation. 
In the 1990s, Deift and Zhou introduced a steepest descent method for RH problems, which is particularly well-suited for the evaluation of asymptotics for integrable PDEs \cite{DZ1993}. Although the Deift-Zhou method by now has been successfully applied to a large number of asymptotic problems for nonlinear integrable PDEs, the question of finding the long-time asymptotics for the different versions of the Boussinesq equation remains an outstanding problem \cite{D2008}. 
There are several reasons why the analysis of the Boussinesq equation is considerably more complicated than the analysis of many other integrable equations, such as the nonlinear Schr\"odinger equation. One major reason is that the spectral problem is third-order; another is that the eigenfunctions have additional singularities (in our case at the origin). 

One of the main advantages of the solution representation featured in Theorem \ref{RHth} is that it is conducive to the evaluation of long-time asymptotics via Deift-Zhou steepest descent arguments. In fact, in \cite{CLWasymptotics}, the representation of Theorem \ref{RHth} together with a steepest descent analysis are used to establish asymptotic formulas for the solution of (\ref{boussinesq}). 


\subsection{Outline of the paper}
The main results (Theorem \ref{r1r2th} and Theorem \ref{RHth}) are stated in Section \ref{mainsec}. 
In Section \ref{specsec}, we begin the spectral analysis. We first transform the third-order spectral problem $\mathsf{L}\varphi = \lambda \varphi$ into a convenient matrix form which makes the underlying symmetries manifest. We then construct eigenfunctions $X$ and $Y$ of this matrix equation which are normalized at $x = +\infty$ and $x = -\infty$, respectively, and we study the scattering matrix $s(k)$ which relates $X$ and $Y$.
Since the spectral problem is third-order, the eigenfunctions $X$ and $Y$ alone are not sufficient for the formulation of a RH problem (the analyticity domains of their columns are not large enough to cover the whole complex plane). We therefore proceed as in \cite{L3x3} and define two further eigenfunctions $X^A$ and $Y^A$ and an associated scattering matrix $s^A$; if the initial data have compact support, $X^A$ is the inverse transpose of $X$, but this is not the case for general initial data. Section \ref{specsec} concludes with the proof of Theorem \ref{r1r2th}.

In Section \ref{Msec}, we define the $3\times 3$ matrix valued function $M$ as the solution of a Fredholm integral equation. In Section \ref{RHsec}, we complete the proof of Theorem \ref{RHth} by relating $M$ to the functions $X$, $Y$, $X^A$, $Y^A$ and showing that it satisfies a $3\times 3$ matrix RH problem. Section \ref{numericalsec} presents some numerical verifications of the results of the earlier sections. 
The proof of uniqueness for the RH problem for $M$ is postponed to an appendix. 

\section{Main results}\label{mainsec}
Our results are formulated in terms of two spectral functions $r_1(k)$  and $r_2(k)$, which can be viewed as the ``reflection coefficients'' for the system (\ref{boussinesqsystem}) determined by the initial data 
$$u_0(x) := u(x,0), \qquad v_0(x) := v(x,0).$$ 
The functions $r_1(k)$  and $r_2(k)$ can also be thought of as nonlinear Fourier transforms of the initial data. They are used to formulate a RH problem from which the solution $\{u(x,t), v(x,t)\}$ can be obtained. We first consider the direct problem, i.e., the construction of $\{r_j(k)\}_1^2$ in terms of the initial data $\{u_0(x), v_0(x)\}$.

\subsection{The direct problem}
Let $\mathcal{S}(\R)$ denote the Schwartz class of rapidly decaying functions on the real line.
Let $u_0, v_0 \in \mathcal{S}(\R)$ be two real-valued functions. The reflection coefficients $r_1(k)$ and $r_2(k)$ associated to $u_0, v_0$ are defined as follows (we refer to Section \ref{specsec} for the origin of the following definitions and for full proofs that the introduced functions are well-defined).

Let $\omega := e^{\frac{2\pi i}{3}}$ and define $\{l_j(k), z_j(k)\}_{j=1}^3$ by
\begin{align}\label{lmexpressions intro}
&l_j(k) = \omega^j k, \quad z_j(k) = \omega^{2j} k^{2}, \qquad k \in \C.
\end{align}
Let the matrix-valued function $\mathsf{U}(x,k)$ be given by
\begin{align}\label{mathsfUdef intro}
\mathsf{U}(x,k) = P(k)^{-1} \begin{pmatrix}
0 & 0 & 0 \\
0 & 0 & 0 \\
-v_0(x)-u_{0x}(x)  & -2u_0(x) & 0
\end{pmatrix} P(k),
\end{align} 
where
\begin{align}\label{Pdef intro}
P(k) = \begin{pmatrix}
\omega & \omega^{2} & 1  \\
\omega^{2} k & \omega k & k \\
k^{2} & k^{2} & k^{2}
\end{pmatrix}.
\end{align}
Define the $3 \times 3$-matrix valued eigenfunctions $X(x,k)$ and $X^A(x,k)$ as the unique solutions of the Volterra integral equations
\begin{subequations}\label{XXAdef intro}
\begin{align}  
 & X(x,k) = I - \int_x^{\infty} e^{(x-x')\widehat{\mathcal{L}(k)}} (\mathsf{U}X)(x',k) dx',
	\\\label{XXAdefb intro}
 & X^A(x,k) = I + \int_x^{\infty} e^{-(x-x')\widehat{\mathcal{L}(k)}} (\mathsf{U}^T X^A)(x',k) dx',	
\end{align}
\end{subequations}
where $\mathcal{L} = \diag(l_1 , l_2 , l_3)$, $\hat{\mathcal{L}}$ denotes the operator which acts on a $3 \times 3$ matrix $A$ by $\hat{\mathcal{L}}A = [\mathcal{L}, A]$ (i.e. $e^{\hat{\mathcal{L}}}A = e^\mathcal{L} A e^{-\mathcal{L}}$), and $\mathsf{U}^T$ denotes the transpose of $\mathsf{U}$. Define $s(k)$ and $s^A(k)$ by 
\begin{align}\label{sdef intro}
& s(k) = I - \int_\R e^{-x\widehat{\mathcal{L}(k)}}(\mathsf{U}X)(x,k)dx,
 	\\ \label{sAdef intro}
& s^A(k) = I + \int_\R e^{x\widehat{\mathcal{L}(k)}}(\mathsf{U}^T X^A)(x,k)dx.
\end{align}
The two spectral functions $\{r_j(k)\}_1^2$ are defined by
\begin{align}\label{r1r2def}
\begin{cases}
r_1(k) = \frac{(s(k))_{12}}{(s(k))_{11}}, & k \in (0,\infty),
	\\ 
r_2(k) = \frac{(s^A(k))_{12}}{(s^A(k))_{11}}, \quad & k \in (-\infty,0).
\end{cases}
\end{align}	

\subsubsection{Assumption of no solitons}
We will show in Proposition \ref{sprop} that the entries $(s(k))_{11}$ and $(s(k))_{12}$ of $s(k)$ that appear in (\ref{r1r2def}) are smooth functions of $k \in (0,\infty)$. Similarly, we will show in Proposition \ref{sAprop} that the entries $(s^A(k))_{11}$ and $(s^A(k))_{12}$ in (\ref{r1r2def}) are smooth functions of $k \in (-\infty,0)$. Thus $r_1(k)$ and $r_2(k)$ are smooth on their respective domains, except possibly at points where $s_{11}$ and $s^A_{11}$ have zeros. The possible zeros of $s_{11}$ and $s^A_{11}$ are related to the presence of solitons. In this paper, we will only consider solitonless (i.e., pure radiation) solutions; it is well-known how to handle the case when solitons are present by considering a RH problem with poles and corresponding residue conditions, see e.g. \cite{L3x3} for a $3 \times 3$ matrix case. 

Propositions \ref{sprop} and \ref{sAprop} imply that $s_{11}$ and $s^A_{11}$ have analytic continuations to $D_1$ and $D_4$, respectively, where $D_n$, $n = 1, \dots, 6$, are the open sectors of the complex $k$-plane displayed in Figure \ref{Gamma.pdf}. Our main results will be stated under the following assumption.

\begin{assumption}[Absence of solitons]\label{solitonlessassumption}\upshape
Assume that $(s(k))_{11}$ and $(s^A(k))_{11}$ are nonzero for $k \in \bar{D}_1\setminus \{0\}$ and $k \in \bar{D}_4\setminus \{0\}$, respectively.
\end{assumption}

\subsubsection{Assumption of generic behavior at $k = 0$}
It turns out that each of the four functions $s_{11}$, $s_{12}$, $s^A_{11}$, and $s^A_{12}$ that appear in (\ref{r1r2def}) has at most a double pole at $k = 0$ (see Propositions \ref{sprop} and \ref{sAprop}). Moreover, $s_{12}$ has a double pole if and only if $s_{11}$ has a double pole, and $s^A_{12}$ has a double pole if and only if $s^A_{11}$ has a double pole, see (\ref{spm2p}) and (\ref{sApm2p}). For simplicity, we will restrict ourselves to the generic case in which all of these four functions have double poles. Our main results will therefore be stated under the following assumption.

\begin{assumption}[Generic behavior at $k = 0$]\label{originassumption}\upshape
Assume that
$$\lim_{k \to 0} k^2 (s(k))_{11} \neq 0, \qquad \lim_{k \to 0} k^2 (s^A(k))_{11} \neq 0.$$
\end{assumption}

\subsubsection{Statement of the first theorem}
We can now state our first theorem, which concerns the direct problem, that is, the map from $\{u_0, v_0\}$ to $\{r_j\}_{j=1}^2$. 

\begin{theorem}[Properties of $r_1(k)$ and $r_2(k)$]\label{r1r2th}
Suppose $u_0,v_0 \in \mathcal{S}(\R)$ are such that Assumptions \ref{solitonlessassumption} and \ref{originassumption} hold.
Then the spectral functions $r_1:(0,\infty) \to \C$ and $r_2:(-\infty,0) \to \C$ are well-defined by (\ref{r1r2def}) and have the following properties:
\begin{enumerate}[$(i)$]
 \item $r_1 \in C^\infty((0,\infty))$ and $r_2 \in C^\infty((-\infty,0))$. 
 
 \item The functions $r_1(k)$, $r_2(k)$, and their derivatives $\partial_{k}^{j}r_{\ell}(k)$ have continuous boundary values at $k=0$ for $\ell = 1,2$ and for all $j = 0,1,2,\ldots$, and there exist expansions
 \begin{subequations}\label{r1r2atzero}
\begin{align}
& r_{1}(k) = r_{1}(0) + r_{1}'(0)k + \tfrac{1}{2}r_{1}''(0)k^{2} + \cdots, & & k \to 0, \ k >0, \\
& r_{2}(k) = r_{2}(0) + r_{2}'(0)k + \tfrac{1}{2}r_{2}''(0)k^{2} + \cdots, & & k \to 0, \ k <0,
\end{align}
\end{subequations}
which can be differentiated termwise any number of times.

\item The leading coefficients are given by
\begin{align}\label{r1r2at0}
r_{1}(0) = \omega, \qquad r_{2}(0) = 1.
\end{align}

\item $r_1(k)$ and $r_2(k)$ are rapidly decreasing as $|k| \to \infty$, i.e.,
 \begin{subequations}\label{r1r2rapiddecay}
\begin{align}
& \max_{j=0,1,\dots,N}\sup_{k \in (0,\infty)} (1+k)^N |\partial_k^jr_1(k)| < \infty,  
	\\
& \max_{j=0,1,\dots,N} \sup_{k \in (-\infty, 0)} (1+|k|)^N|\partial_k^jr_2(k)| < \infty,
\end{align}
\end{subequations}
for each integer $N \geq 0$. 

\item $|r_{1}(k)|<1$ for all $k > 0$ and $|r_{2}(k)| <1$ for all $k < 0$.

\end{enumerate} 
\end{theorem}
\begin{proof}
See Section \ref{r1r2subsec}.
\end{proof}

\subsection{The inverse problem}
We next consider the inverse problem of recovering $\{u_0, v_0\}$ from the scattering data. Since we are assuming that no solitons are present, the scattering data consists only of the two reflection coefficients $r_1(k)$ and $r_2(k)$. We will show that the inverse problem can be solved by means of a RH problem for a $3 \times 3$-matrix valued function $M$ whose jump matrix is expressed in terms of $r_1$ and $r_2$. As usual in the implementation of the inverse scattering transform, the time evolution of the scattering data is very simple. By including this simple time-dependence in the definition of the jump matrix, we obtain the solution $\{u(x,t), v(x,t)\}$ at any later time $t$ from the relations (\ref{recoveruv}). We first give the definition of the RH problem.

\subsubsection{The RH problem for $M$}
Let $\Gamma$ be the contour consisting of the three lines $\R \cup \omega \R \cup \omega^2 \R$ oriented away from the origin as in Figure \ref{Gamma.pdf}.
For $1 \leq i \neq j \leq 3$, define $\theta_{ij} \equiv \theta_{ij}(x,t,k)$ by
$$\theta_{ij}(x,t,k) = (l_i - l_j)x + (z_i - z_j)t.$$
Define the jump matrix $v(x,t,k)$ for $k \in \Gamma$ by
\begin{align}\nonumber
&  v_1 = 
  \begin{pmatrix}  
 1 & - r_1(k)e^{-\theta_{21}} & 0 \\
  r_1^*(k)e^{\theta_{21}} & 1 - |r_1(k)|^2 & 0 \\
  0 & 0 & 1
  \end{pmatrix},
	\\\nonumber
&  v_2 = 
  \begin{pmatrix}   
 1 & 0 & 0 \\
 0 & 1 - r_2(\omega k)r_2^*(\omega k) & -r_2^*(\omega k)e^{-\theta_{32}} \\
 0 & r_2(\omega k)e^{\theta_{32}} & 1 
    \end{pmatrix},
   	\\ \nonumber
  &v_3 = 
  \begin{pmatrix} 
 1 - r_1(\omega^2 k)r_1^*(\omega^2 k) & 0 & r_1^*(\omega^2 k)e^{-\theta_{31}} \\
 0 & 1 & 0 \\
 -r_1(\omega^2 k)e^{\theta_{31}} & 0 & 1  
  \end{pmatrix},
	\\\nonumber
&  v_4 = 
  \begin{pmatrix}  
  1 - |r_2(k)|^2 & -r_2^*(k) e^{-\theta_{21}} & 0 \\
  r_2(k)e^{\theta_{21}} & 1 & 0 \\
  0 & 0 & 1
   \end{pmatrix},
   	\\ \nonumber
&  v_5 = 
  \begin{pmatrix}
  1 & 0 & 0 \\
  0 & 1 & -r_1(\omega k)e^{-\theta_{32}} \\
  0 & r_1^*(\omega k)e^{\theta_{32}} & 1 - r_1(\omega k)r_1^*(\omega k) 
  \end{pmatrix},
	\\\label{vdef}
&  v_6 = 
  \begin{pmatrix} 
  1 & 0 & r_2(\omega^2 k)e^{-\theta_{31}} \\
  0 & 1 & 0 \\
  -r_2^*(\omega^2 k)e^{\theta_{31}} & 0 & 1 - r_2(\omega^2 k)r_2^*(\omega^2 k)
   \end{pmatrix},
\end{align}
where $v_j$ denotes the restriction of $v$  to the subcontour of $\Gamma$ labeled by $j$ in Figure \ref{Gamma.pdf}.
We consider the following RH problem.

\begin{RHproblem}[RH problem for $M$]\label{RH problem for M}
Find a $3 \times 3$-matrix valued function $M(x,t,k)$ with the following properties:
\begin{enumerate}[(a)]
\item $M(x,t,\cdot) : \mathbb{C}\setminus \Gamma \to \mathbb{C}^{3 \times 3}$ is analytic.

\item The limits of $M(x,t,k)$ as $k$ approaches $\Gamma\setminus \{0\}$ from the left and right exist, are continuous on $\Gamma\setminus \{0\}$, and are denoted by $M_+$ and $M_-$, respectively. Furthermore, they are related by
\begin{align}\label{Mjumpcondition}
& M_{+}(x,t,k) = M_{-}(x,t,k)v(x,t,k), \qquad k \in \Gamma.
\end{align}

\item As $k \to \infty$, $k \notin \Gamma$, we have
\begin{align*}
M(x,t,k) = I + \frac{M^{(1)}(x,t)}{k} + \frac{M^{(2)}(x,t)}{k^{2}} + O\bigg(\frac{1}{k^3}\bigg),
\end{align*}
where the matrices $M^{(1)}$ and $M^{(2)}$ depend on $x$ and $t$ but not on $k$, and satisfy
\begin{align}\label{singRHMatinftyb}
M_{12}^{(1)} = M_{13}^{(1)} = M_{12}^{(2)} + M_{13}^{(2)} = 0.
\end{align}

\item There exist matrices $\{\mathcal{M}_1^{(l)}(x,t)\}_{l=-2}^{+\infty}$ depending on $x$ and $t$ but not on $k$ such that, for any $N \geq -2$,
\begin{align}\label{singRHMat0}
M(x,t,k) = \sum_{l=-2}^{N} \mathcal{M}_1^{(l)}(x,t)k^{l} + O(k^{N+1}) \qquad \text{as}\ k \to 0, \ k \in \bar{D}_1.
\end{align}
Furthermore, there exist scalar coefficients $\alpha, \beta, \gamma, \delta, \epsilon$ depending on $x$ and $t$, but not on $k$, such that
\begin{align} \label{explicit Mcalpm2p}
\mathcal{M}_{1}^{(-2)}(x,t) = &\; \alpha(x,t) \begin{pmatrix}
\omega & 0 & 0 \\
\omega & 0 & 0 \\
\omega & 0 & 0
\end{pmatrix}, 
	\\
\mathcal{M}_{1}^{(-1)}(x,t)  = &\; \beta(x,t) \begin{pmatrix}
\omega^{2} & 0 & 0 \\
\omega^{2} & 0 & 0 \\
\omega^{2} & 0 & 0 
\end{pmatrix} + \gamma(x,t) \begin{pmatrix}
\omega^{2} & 0 & 0 \\
1 & 0 & 0 \\
\omega & 0 & 0
\end{pmatrix} \nonumber \\
& + \delta(x,t) \begin{pmatrix}
0 & 1-\omega & 0 \\
0 & 1-\omega & 0 \\
0 & 1-\omega & 0
\end{pmatrix}, \label{explicit Mcalpm1p}
\end{align}
and the third column of $\mathcal{M}_{1}^{(0)}(x,t)$ is given by
\begin{align}
[\mathcal{M}_{1}^{(0)}(x,t)]_{3} = \epsilon(x,t) \begin{pmatrix}
1 \\
1 \\
1 
\end{pmatrix}. \label{explicit Mcalp0p third column}
\end{align}

\item $M$ satisfies the symmetries
\begin{align}\label{singRHsymm}
M(x,t, k) = \mathcal{A} M(x,t,\omega k)\mathcal{A}^{-1} = \mathcal{B} \overline{M(x,t,\overline{k})}\mathcal{B}, \qquad k \in \C \setminus \Gamma,
\end{align}
where
\begin{align}\label{def of Acal and Bcal}
\mathcal{A} := \begin{pmatrix}
0 & 0 & 1 \\
1 & 0 & 0 \\
0 & 1 & 0
\end{pmatrix} \qquad \mbox{ and } \qquad \mathcal{B} := \begin{pmatrix}
0 & 1 & 0 \\
1 & 0 & 0 \\
0 & 0 & 1
\end{pmatrix}.
\end{align}
\end{enumerate}
\end{RHproblem}

The functions $\alpha, \beta, \gamma, \delta, \epsilon$ are not prescribed in the formulation of the RH problem \ref{RH problem for M}. Nevertheless, as we show in Appendix \ref{appA}, the solution $M$ is still uniquely determined thanks to the additional assumptions (\ref{singRHMatinftyb}) on the behavior at $k = \infty$.

\subsubsection{Statement of the second theorem}
Our second theorem states that the solution $\{u(x,t), v(x,t)\}$ of the Boussinesq equation \eqref{boussinesqsystem} can be recovered from the solution $M(x,t,k)$ of the RH problem \ref{RH problem for M} via the relations (\ref{recoveruv}).
Although it is possible to carry out all the arguments under more restricted regularity and decay assumptions, we will only deal with Schwartz class solutions for simplicity. 

\begin{definition}\upshape
We call $\{u(x,t), v(x,t)\}$ a {\it Schwartz class solution of \eqref{boussinesqsystem} with existence time $T \in (0, \infty]$ and initial data $u_0, v_0 \in \mathcal{S}(\R)$} if
\begin{enumerate}[$(i)$] 
  \item $u,v$ are smooth real-valued functions of $(x,t) \in \R \times [0,T)$.

\item $u,v$ satisfy \eqref{boussinesqsystem} for $(x,t) \in \R \times [0,T)$ and 
$$u(x,0) = u_0(x), \quad v(x,0) = v_0(x), \qquad x \in \R.$$ 

  \item $u,v$ have rapid decay as $|x| \to \infty$ in the sense that, for each integer $N \geq 1$,
$$\sup_{\substack{x \in \R \\ t \in [0, T)}} \sum_{i =0}^N (1+|x|)^N(|\partial_x^i u| + |\partial_x^i v| ) < \infty.$$
\end{enumerate} 
\end{definition}

The second theorem can now be stated.

\begin{theorem}[Solution of (\ref{boussinesqsystem}) via inverse scattering]\label{RHth}
Suppose $\{u(x,t), v(x,t)\}$ is a Schwartz class solution of (\ref{boussinesqsystem}) with existence time $T \in (0, \infty]$ and initial data $u_0, v_0 \in \mathcal{S}(\R)$ such that Assumptions \ref{solitonlessassumption} and \ref{originassumption} hold. Define the spectral functions $r_j(k)$, $j = 1,2$, in terms of $u_0, v_0$ by (\ref{r1r2def}).
Then the RH problem \ref{RH problem for M} has a unique solution $M(x,t,k)$ for each $(x,t) \in \R \times [0,T)$ and the formulas (\ref{recoveruv}) expressing $\{u(x,t), v(x,t)\}$ in terms of $M$ are valid for all $(x,t) \in \R \times [0,T)$.
\end{theorem}
\begin{proof}
See Section \ref{RHsec}.
\end{proof}

\subsection{Corollaries}
We have stated Theorem \ref{r1r2th} and Theorem \ref{RHth} for the system (\ref{boussinesqsystem}). We can easily obtain analogous results for equation (\ref{boussinesq}) as corollaries provided that $u_1$ has zero mean. We begin by making the notion of solution of (\ref{boussinesq}) precise. 

\begin{definition}\upshape
We call $u(x,t)$ a {\it Schwartz class solution of \eqref{boussinesq} with existence time $T \in (0, \infty]$ and initial data $u_0, u_1 \in \mathcal{S}(\R)$} if
\begin{enumerate}[$(i)$] 
  \item $u$ is a smooth real-valued function of $(x,t) \in \R \times [0,T)$.

\item $u$ satisfies \eqref{boussinesq} for $(x,t) \in \R \times [0,T)$ and 
$$u(x,0) = u_0(x), \quad u_t(x,0) = u_1(x), \qquad x \in \R.$$ 

  \item $u$ has rapid decay as $|x| \to \infty$ in the sense that, for each integer $N \geq 1$,
\begin{align}\label{rapiddecay}
\sup_{\substack{x \in \R \\ t \in [0, T)}} \sum_{i =0}^N (1+|x|)^N|\partial_x^i u| < \infty.
\end{align}

\end{enumerate} 
\end{definition}

The next lemma shows that equation (\ref{boussinesq}) is equivalent to the system (\ref{boussinesqsystem}) provided that $u_1$ has zero mean.

\begin{lemma}\label{systemlemma}
Let $T \in (0, \infty]$. If $\{u,v\}$ is a Schwartz class solution of \eqref{boussinesqsystem} with existence time $T$ and initial data $u_0, v_0 \in \mathcal{S}(\R)$, then $u$ is Schwartz class solution of \eqref{boussinesq} with existence time $T$ and initial data $u_0, u_1 \in \mathcal{S}(\R)$, where $u_1(x) = u_t(x,0)$; moreover, $u_1$ satisfies (\ref{u1zeromean}).

Conversely, suppose $u$ is Schwartz class solution of \eqref{boussinesq} with existence time $T$ and initial data $u_0, u_1 \in \mathcal{S}(\R)$ such that (\ref{u1zeromean}) holds. Let 
\begin{align}\label{vxtdef}
v(x,t) = \int_{-\infty}^x u_t(x', t) dx', \qquad v_0(x) = \int_{-\infty}^x u_1(x') dx'.
\end{align}
Then $\{u,v\}$ is a Schwartz class solution of \eqref{boussinesqsystem} with existence time $T$ and initial data $u_0, v_0 \in \mathcal{S}(\R)$.
\end{lemma}
\begin{proof}
Suppose $\{u,v\}$ is a Schwartz class solution of \eqref{boussinesqsystem} with existence time $T$ and initial data $u_0, v_0 \in \mathcal{S}(\R)$.
Differentiating the first equation in (\ref{boussinesqsystem}) with respect to $x$ and using the second equation in (\ref{boussinesqsystem}) to replace $v_{tx}$ with $u_{tt}$, we find that $u$ satisfies (\ref{boussinesq}) with initial data $u_0(x) = u(x,0)$ and $u_1(x) = u_t(x,0) = v_{0x}(x)$.

For the converse, suppose $u$ is Schwartz class solution of \eqref{boussinesq} with existence time $T$ and initial data $u_0, u_1 \in \mathcal{S}(\R)$. We see from (\ref{boussinesq}) that (\ref{rapiddecay}) holds with $u$ replaced by $u_{tt}$. Then, writing $u_t(x,t) = \int_{t_0}^t u_{tt}(x,t') dt'$, we infer from straightforward estimates that (\ref{rapiddecay}) holds also with $u$ replaced by $u_{t}$.
Define $v$ by (\ref{vxtdef}). It follows from (\ref{u1zeromean}) and (\ref{intuintu1}) that $\int_\R u dx$ is conserved in time. Thus $\int_\R u_t(x', t) dx' = 0$ and hence (\ref{rapiddecay}) holds also with $u$ replaced by $v$. Integration of (\ref{boussinesq}) from $-\infty$ to $x$ yields
$$\int_{-\infty}^x u_{tt}(x', t) dx' + \frac{1}{3} u_{xxx} + \frac{4}{3} (u^2)_{x}  = 0.$$
Since
$$\int_{-\infty}^x u_{tt}(x', t) dx' = v_t(x,t),$$
we conclude that $u,v$ satisfy (\ref{boussinesqsystem}).
\end{proof}

In view of Lemma \ref{systemlemma}, we immediately obtain the following corollary of Theorems \ref{r1r2th} and \ref{RHth}.

\begin{corollary}[Solution of (\ref{boussinesq}) via inverse scattering]\label{RHcor}
Suppose $u(x,t)$ is a Schwartz class solution of (\ref{boussinesq}) with existence time $T \in (0, \infty]$ and initial data $u_0, u_1 \in \mathcal{S}(\R)$ such that $\int_\R u_1 dx = 0$. Define $v_0 \in \mathcal{S}(\R)$ by (\ref{vxtdef}) and define $r_1:(0,\infty) \to \C$ and $r_2:(-\infty,0) \to \C$ by (\ref{r1r2def}). Suppose Assumptions \ref{solitonlessassumption} and \ref{originassumption} hold. 
Then $r_1$ and $r_2$ have all the properties listed in Theorem \ref{r1r2th}. Moroever, the RH problem \ref{RH problem for M} has a unique solution $M(x,t,k)$ for each $(x,t) \in \R \times [0,T)$ and the formula in (\ref{recoveruv}) expressing $u$ in terms of $M$ is valid for all $(x,t) \in \R \times [0,T)$.
\end{corollary}

Since the transformation (\ref{fromboussinesqtogoodboussinesq}) takes solutions of (\ref{boussinesq}) to solutions of (\ref{goodboussinesq}), we can also express our results in terms of solutions of (\ref{goodboussinesq}).

\begin{definition}\upshape
We call $\tilde{u}(x,t)$ a {\it background density 1/2 solution of \eqref{goodboussinesq} with existence time $T \in (0, \infty]$ and initial data $\tilde{u}_0 \in \frac{1}{2}+\mathcal{S}(\mathbb{R}),\tilde{u}_1 \in \mathcal{S}(\R)$} if
\begin{enumerate}[$(i)$] 
  \item $\tilde{u}$ is a smooth real-valued function of $(x,t) \in \R \times [0,T)$.

\item $\tilde{u}$ satisfies \eqref{goodboussinesq} for $(x,t) \in \R \times [0,T)$ and 
$$\tilde{u}(x,0) = \tilde{u}_0(x), \quad \tilde{u}_t(x,0) = \tilde{u}_1(x), \qquad x \in \R.$$ 

  \item $\tilde{u}$ has rapid decay to $1/2$ as $|x| \to \infty$ in the sense that, for each integer $N \geq 1$,
\begin{align*}
\sup_{\substack{x \in \R \\ t \in [0, T)}} \sum_{i =0}^N (1+|x|)^N\big|\partial_x^i \big(\tilde{u} - \tfrac{1}{2}\big)\big| < \infty.
\end{align*}
\end{enumerate} 
\end{definition}

In view of (\ref{fromboussinesqtogoodboussinesq}), Corollary \ref{RHcor} can be reformulated as follows.

\begin{corollary}[Solution of (\ref{goodboussinesq}) via inverse scattering]\label{RHcor2}
Suppose $\tilde{u}(x,t)$ is a background density 1/2 solution of (\ref{goodboussinesq}) with existence time $T \in (0, \infty]$ and initial data $\tilde{u}_0 \in \frac{1}{2}+\mathcal{S}(\R), \tilde{u}_1 \in \mathcal{S}(\R)$ such that $\int_\R \tilde{u}_1(x)dx = 0$. Define $u_0, v_0 \in \mathcal{S}(\R)$ by 
$$u_0(x) = \frac{\sqrt{3}}{4} \bigg(\tilde{u}_0(3^{1/4}x) - \frac{1}{2}\bigg), \qquad v_0(x) = \frac{\sqrt{3}}{4} \int_{-\infty}^x  \tilde{u}_1(3^{1/4}x') dx',$$
and define $r_1:(0,\infty) \to \C$ and $r_2:(-\infty,0) \to \C$ in terms of $u_0, v_0$ by (\ref{r1r2def}). Suppose Assumptions \ref{solitonlessassumption} and \ref{originassumption} hold. 
Then $r_1$ and $r_2$ have all the properties listed in Theorem \ref{r1r2th}. Moroever, the RH problem \ref{RH problem for M} has a unique solution $M(x,t,k)$ for each $(x,t) \in \R \times [0,T)$ and the following formula for $\tilde{u}$ in terms of $M$ is valid for $(x,t) \in \R \times [0,T)$:
$$\tilde{u}(x,t) = \frac{1}{2} - 2 \cdot 3^{3/4} \frac{\partial}{\partial x}\lim_{k\to \infty}k\bigg(M_{33}\bigg(\frac{x}{3^{1/4}},t,k\bigg) - 1\bigg).$$
\end{corollary}

Our next and last corollary is useful for the evaluation of long-time asymptotics of the solution $u(x,t)$ via steepest descent arguments.
The RH problem \ref{RH problem for M} is singular at the origin in the sense that the solution $M$ is allowed to have a double pole at $k = 0$. 
It is therefore not very convenient to perform a steepest descent analysis of this RH problem. We can obtain a RH problem which is regular at the origin by introducing the row-vector-valued function $n$ by
\begin{align}\label{ndef}
n(x,t,k) = \begin{pmatrix}\omega & \omega^2 & 1 \end{pmatrix} M(x,t,k).
\end{align}
Indeed, since the coefficients of $k^{-2}$ and $k^{-1}$ in the expansion of $M$ at $k = 0$ are such that they vanish when premultiplied by the row-vector $(\omega,\omega^{2},1)$ (see (\ref{explicit Mcalpm2p}) and (\ref{explicit Mcalpm1p})), the function $n$ satisfies the following vector RH problem.

\begin{RHproblem}[RH problem for $n$]\label{RHn}
Find a $1 \times 3$-row-vector valued function $n(x,t,k)$ with the following properties:
\begin{enumerate}[(a)]
\item $n(x,t,\cdot) : \C \setminus \Gamma \to \mathbb{C}^{1 \times 3}$ is analytic.

\item The limits of $n(x,t,k)$ as $k$ approaches $\Gamma \setminus \{0\}$ from the left and right exist, are continuous on $\Gamma \setminus \{0\}$, and are denoted by $n_+$ and $n_-$, respectively. Furthermore, they are related by
\begin{align}\label{njump}
  n_+(x,t,k) = n_-(x, t, k) v(x, t, k), \qquad k \in \Gamma \setminus \{0\}.
\end{align}

\item $n(x,t,k) = (\omega,\omega^{2},1) + O(k^{-1})$ as $k \to \infty$.

\item $n(x,t,k) = O(1)$ as $k \to 0$.
\end{enumerate}
\end{RHproblem}
The RH problem for $n$ is regular at the origin and is clearly simpler than the RH problem for $M$. Moreover, the solution $\{u, v\}$ of (\ref{boussinesqsystem}) can be recovered from $n$ via the relations
\begin{align}\label{recoveruvn}
\begin{cases}
u(x,t) = -\frac{3}{2}\frac{\partial}{\partial x}\lim_{k\to \infty}k(n_{3}(x,t,k) - 1),
	\\
v(x,t) = -\frac{3}{2}\frac{\partial}{\partial t}\lim_{k\to \infty}k(n_{3}(x,t,k) - 1).
\end{cases}
\end{align}
However, we have not been able to establish uniqueness of the solution of the RH problem for $n$ except in special cases (this is the reason why we have chosen to formulate Theorem \ref{RHth} in terms of $M$ rather than $n$). When performing a steepest descent analysis, the existence of a unique solution follows from the analysis of a small-norm RH problem. The following corollary, which follows easily from Theorem \ref{RHth}, is therefore useful for the evaluation of long-time asymptotics, see \cite{CLWasymptotics}. 

\begin{corollary}[Solution of (\ref{boussinesqsystem}) in terms of $n$]\label{ncor}
Suppose the assumptions of Theorem \ref{RHth} hold. Let $U$ be an open subset of $\R \times [0,\infty)$ and suppose that the solution of the RH problem \ref{RHn} for $n$ is unique for each $(x,t) \in U$ whenever it exists. Then the RH problem \ref{RHn} has a unique solution $n(x,t,k)$ for each $(x,t) \in U$ and the formulas (\ref{recoveruvn}) are valid for all $(x,t) \in U$.
\end{corollary}

\begin{remark}[Comparison between the ``good'' and ``bad'' Boussinesq equations]
The above results are all concerned with the ``good'' Boussinesq equation. 
As mentioned in the introduction, an inverse scattering transform formalism for the solution of the ``bad'' Boussinesq equation (more precisely, for equation (\ref{boussinesqnouxx}) with $\sigma = -1$) was developed already in early 1980s \cite{DTT1982}. It may seem counterintuitive that the ``good'' case be dealt with after the ``bad'' case. However, as explained before equation (\ref{goodboussinesq}), the designation ``bad'' refers to the fact that the equation is linearly ill-posed; it does not necessarily imply that the equation is more difficult to handle in other respects. In fact, the methods presented in this paper can be adapted to handle the spectral analysis also for the ``bad'' Boussinesq equation (\ref{badboussinesq}). 
On the other hand, since the resulting RH problems are different and the ``bad'' Boussinesq equation is linearly ill-posed, there are major differences between the ``good'' and ``bad'' cases when it comes to determining formulas for the long-time asymptotics. 
\end{remark}

\section{Spectral analysis}\label{specsec}
\subsection{Preliminaries}

The Lax pair (\ref{operatorLaxpair}) can be expressed as the compatibility condition of the equations
\begin{align}\label{scalarLaxpair}
\mathsf{L}\varphi = k^3 \varphi, \qquad \varphi_t = \mathsf{A}\varphi,
\end{align}
where $\varphi(x,t,k)$ is a scalar-valued eigenfunction and $k \in \C$ is a spectral parameter. We can rewrite the first equation in (\ref{scalarLaxpair}), which is a third-order differential equation, as a first-order system by defining the vector
$$\Phi = \begin{pmatrix} \varphi \\\varphi_x \\ \varphi_{xx} \end{pmatrix}.$$
In terms of $\Phi$, the equations in (\ref{scalarLaxpair}) can be written as
\begin{align}\label{Philax}
\begin{cases}
 \Phi_{x} = \tilde{L}\Phi, \\
 \Phi_{t} = \tilde{Z}\Phi,
\end{cases}
\end{align}
where $\tilde{L} \equiv \tilde{L}(x,t,k)$ and $\tilde{Z} \equiv \tilde{Z}(x,t,k)$ are given by
\begin{equation}
\tilde{L} = \begin{pmatrix}
0 & 1 & 0 \\
0 & 0 & 1 \\
k^{3}-u_{x}-v & -2u & 0
\end{pmatrix}, \qquad \tilde{Z} = \begin{pmatrix}
\frac{4}{3}u & 0 & 1 \\
k^{3} + \frac{u_{x}}{3}-v & -\frac{2}{3}u & 0 \\
\frac{u_{xx}}{3}-v_{x} & k^{3} - \frac{u_{x}}{3}-v & - \frac{2}{3}u
\end{pmatrix}.
\end{equation}
In order to treat the three linearly independent solutions at once, we rewrite (\ref{Philax}) in matrix form as
\begin{equation}\label{Lax pair tilde}
\begin{cases}
 \tilde{X}_{x} = \tilde{L}\tilde{X}, \\
 \tilde{X}_{t} = \tilde{Z}\tilde{X},
\end{cases}
\end{equation}
where $\tilde{X}(x,t,k)$ is a $3 \times 3$-matrix valued function.  
It can easily be verified by direct computation that \eqref{boussinesqsystem} is the compatibility condition of (\ref{Lax pair tilde}).
Note that $\tilde{L}$ and $\tilde{Z}$ are traceless. The next transformation diagonalizes the highest-order terms in $k$ as $k \to \infty$ of the Lax pair \eqref{Lax pair tilde} and ensures that the lower-order terms decay as $x \to \pm \infty$. Let us define $\check{X}$ by
\begin{equation}\label{X hat to X tilde}
\tilde{X}(x,t,k) = P(k) \check{X}(x,t,k),
\end{equation}
where
\begin{equation}
P(k) = \begin{pmatrix}
\omega & \omega^{2} & 1  \\
\omega^{2} k & \omega k & k \\
k^{2} & k^{2} & k^{2}
\end{pmatrix}, \qquad \omega = e^{\frac{2\pi i}{3}}.
\end{equation}
We verify from \eqref{Lax pair tilde} that
\begin{equation}\label{Xhatlax}
\begin{cases}
\check{X}_{x} = L\check{X}, \\
\check{X}_{t} = Z\check{X},
\end{cases}
\end{equation}
where
\begin{equation*}
L = P^{-1}\tilde{L}P, \qquad Z = P^{-1}\tilde{Z}P.
\end{equation*}
Note that the above transformation is valid only for $k \in \mathbb{C}\setminus \{0\}$, because
\begin{equation}
\det P(k) = -3 \omega (1-\omega) k^{3}.
\end{equation}
Let us define
\begin{equation}
J = \begin{pmatrix}
\omega & 0 & 0 \\
0 & \omega^{2} & 0 \\
0 & 0 & 1
\end{pmatrix}.
\end{equation}
The matrices $L$ and $Z$ can be written as
\begin{equation}\label{U and V def}
L = \mathsf{U}+\mathcal{L}, \qquad Z=\mathsf{V}+\mathcal{Z},
\end{equation}
where the diagonal matrices $\mathcal{L} \equiv \mathcal{L}(k)$ and $\mathcal{Z} \equiv \mathcal{Z}(k)$ are given by
\begin{equation}\label{def of mathcal L and mathcal Z}
\mathcal{L} = k J, \qquad \mathcal{Z} = k^{2} J^{2},
\end{equation}
and $\mathsf{U}$ and $\mathsf{V}$ are given by
\begin{subequations}\label{UVexpressions}
\begin{align}
& \mathsf{U}(x,t,k) = \frac{\mathsf{U}^{(2)}(x,t)}{k^{2}} + \frac{\mathsf{U}^{(1)}(x,t)}{k}, \label{expression for U} \\
& \mathsf{V}(x,t,k) = \frac{\mathsf{V}^{(2)}(x,t)}{k^{2}} + \frac{\mathsf{V}^{(1)}(x,t)}{k} + \mathsf{V}^{(0)}(x,t), \label{expression for V}
\end{align}
\end{subequations}
with
\begin{align*}
& \mathsf{U}^{(2)} = - \frac{v+u_{x}}{3} \begin{pmatrix}
\omega & \omega^{2} & 1 \\
\omega & \omega^{2} & 1 \\
\omega & \omega^{2} & 1 
\end{pmatrix}, \qquad \mathsf{U}^{(1)} = - \frac{2u}{3}\begin{pmatrix}
\omega^{2} & \omega & 1 \\
\omega^{2} & \omega & 1 \\
\omega^{2} & \omega & 1 
\end{pmatrix}, \\
& \mathsf{V}^{(2)} = \frac{-3v_{x}+u_{xx}}{9}\begin{pmatrix}
\omega & \omega^{2} & 1 \\
\omega & \omega^{2} & 1 \\
\omega & \omega^{2} & 1
\end{pmatrix}, \qquad \mathsf{V}^{(0)} = \frac{2u}{3} \begin{pmatrix}
0 & \omega & \omega^{2} \\
\omega^{2} & 0 & \omega \\
\omega & \omega^{2} & 0
\end{pmatrix}, \\
& \mathsf{V}^{(1)} = \frac{v}{3}\begin{pmatrix}
-2\omega^{2} & \omega^{2} & \omega^{2} \\
\omega & -2 \omega & \omega \\
1 & 1 & -2 
\end{pmatrix} + \frac{(1-\omega)u_{x}}{9}\begin{pmatrix}
0 & 1 & -1 \\
-\omega^{2} & 0 & \omega^{2} \\
\omega & -\omega & 0
\end{pmatrix}.
\end{align*}
It is directly seen from \eqref{U and V def} and \eqref{UVexpressions} that
\begin{equation}
L(x,t,k) = \mathcal{L} + O(1/k), \qquad Z(x,t,k) = \mathcal{Z} + O(1), \qquad \mbox{as } k \to \infty,
\end{equation}
and since $u,v \in \mathcal{S}(\mathbb{R})$, we have
\begin{equation}
\lim_{x \to \pm \infty} \mathsf{U} = 0, \qquad \lim_{x \to \pm \infty} \mathsf{V} = 0, \qquad \mathcal{L} = \lim_{x \to \pm\infty} L, \qquad \mathcal{Z} = \lim_{x \to \pm\infty} Z.
\end{equation}
We denote the diagonal entries of $\mathcal{L}$ and $\mathcal{Z}$ by $\{l_{j}(k)\}_{j=1}^{3}$ and $\{z_{j}(k)\}_{j=1}^{3}$, respectively:
\begin{equation}
\mathcal{L} = \mbox{diag}(l_{1},l_{2},l_{3}), \qquad \mathcal{Z} = \mbox{diag}(z_{1},z_{2},z_{3}).
\end{equation}
Using \eqref{U and V def}, we can rewrite \eqref{Xhatlax} as
\begin{equation}\label{Lax pair hat}
\begin{cases}
 \check{X}_{x} - \mathcal{L}\check{X} = \mathsf{U}\check{X}, \\
 \check{X}_{t} - \mathcal{Z}\check{X} = \mathsf{V}\check{X}.
\end{cases}
\end{equation}
Finally, the transformation 
\begin{equation}
\check{X} = X e^{\mathcal{L}x + \mathcal{Z}t}
\end{equation}
transforms \eqref{Lax pair hat} into
\begin{equation}\label{Xlax}
\begin{cases}
 X_{x} - [\mathcal{L},X] = \mathsf{U} X, \\
 X_{t} - [\mathcal{Z},X] = \mathsf{V} X.
\end{cases}
\end{equation}

\subsubsection{Symmetries of $L$ and $Z$} The matrices $L$ and $Z$ satisfy the $\mathbb{Z}_{3}$ symmetry
\begin{equation}\label{Acaldef}
L(k) = \mathcal{A}L(\omega k) \mathcal{A}^{-1}, \quad Z(k) = \mathcal{A}Z(\omega k) \mathcal{A}^{-1},
\end{equation}
and, since $u$ and $v$ are real-valued, the $\mathbb{Z}_{2}$ symmetry
\begin{equation}\label{Bcaldef}
L(k) = \mathcal{B}\overline{L(\overline{k})} \mathcal{B}, \quad Z(k) = \mathcal{B}\overline{Z(\overline{k})} \mathcal{B},
\end{equation}
where $\mathcal{A}$ and $\mathcal{B}$ are the matrices defined in (\ref{def of Acal and Bcal}).
%

\subsection{The eigenfunctions $X$ and $Y$} 
From now until the end of Section \ref{specsec}, we fix $t=0$ and abuse notation by writing $\mathsf{U}(x,k)$ for $\mathsf{U}(x,0,k)$. Consider the $x$-part of the Lax pair \eqref{Xlax} evaluated at $t = 0$:
\begin{align}\label{xpart}
X_x - [\mathcal{L}, X] = \mathsf{U} X.
\end{align}
We define two $3 \times 3$-matrix valued solutions $X(x,k)$ and $Y(x,k)$ of (\ref{xpart}) as the solutions of the linear Volterra integral equations
\begin{subequations}\label{XYdef}
\begin{align}  \label{XYdefa}
 & X(x,k) = I - \int_x^{\infty} e^{(x-x')\widehat{\mathcal{L}(k)}} (\mathsf{U}X)(x',k) dx',
  	\\ \label{XYdefb}
&  Y(x,k) = I + \int_{-\infty}^x e^{(x-x')\widehat{\mathcal{L}(k)}} (\mathsf{U}Y)(x',k) dx'.
\end{align}
\end{subequations}
We decompose the complex $k$-plane into the six open subsets $\{D_n\}_1^6$ defined by (see Figure \ref{Gamma.pdf}) 
\begin{align*}
&D_1 = \{k \in \C\,|\, \re l_1 < \re l_2 < \re l_3\},
	\\
&D_2 = \{k \in \C\,|\, \re l_1 < \re l_3 < \re l_2\},
	\\
&D_3 = \{k \in \C\,|\, \re l_3 < \re l_1 < \re l_2\},
	\\
&D_4 = \{k \in \C\,|\, \re l_3 < \re l_2 < \re l_1\},
	\\
&D_5 = \{k \in \C\,|\, \re l_2 < \re l_3 < \re l_1\},
	\\
&D_6 = \{k \in \C\,|\, \re l_2 < \re l_1 < \re l_3\},
\end{align*}
and let $\mathrm{S} = \{k \in \C \, | \, \arg k \in (\frac{2\pi}{3}, \frac{4\pi}{3})\}$ denote the interior of $\bar{D}_3 \cup \bar{D}_4$. 

\begin{proposition}[Basic properties of $X$ and $Y$]\label{XYprop}
Suppose $u_0, v_0 \in \mathcal{S}(\R)$. 
Then the equations (\ref{XYdef}) uniquely define two $3 \times 3$-matrix valued solutions $X$ and $Y$ of (\ref{xpart}) with the following properties:
\begin{enumerate}[$(a)$]
\item The function $X(x, k)$ is defined for $x \in \R$ and $k \in (\omega^2 \bar{\mathrm{S}}, \omega \bar{\mathrm{S}}, \bar{\mathrm{S}}) \setminus \{0\}$. For each $k \in (\omega^2 \bar{\mathrm{S}}, \omega \bar{\mathrm{S}}, \bar{\mathrm{S}}) \setminus \{0\}$, $X(\cdot, k)$ is smooth and satisfies (\ref{xpart}).

\item The function $Y(x, k)$ is defined for $x \in \R$ and $k \in (-\omega^2 \bar{\mathrm{S}}, -\omega \bar{\mathrm{S}}, -\bar{\mathrm{S}}) \setminus \{0\}$. For each $k \in (-\omega^2 \bar{\mathrm{S}}, -\omega \bar{\mathrm{S}}, -\bar{\mathrm{S}}) \setminus \{0\}$, $Y(\cdot, k)$ is smooth and satisfies (\ref{xpart}).

\item For each $x \in \R$, the function $X(x,\cdot)$ is continuous for $k \in (\omega^2 \bar{\mathrm{S}}, \omega \bar{\mathrm{S}}, \bar{\mathrm{S}})\setminus \{0\}$ and analytic for $k \in (\omega^2 \mathrm{S}, \omega \mathrm{S}, \mathrm{S}) \setminus \{0\}$.

\item For each $x \in \R$, the function $Y(x,\cdot)$ is continuous for $k \in (-\omega^2 \bar{\mathrm{S}}, -\omega \bar{\mathrm{S}}, -\bar{\mathrm{S}})\setminus \{0\}$ and analytic for $k \in (-\omega^2 \mathrm{S}, -\omega \mathrm{S}, -\mathrm{S}) \setminus \{0\}$.

\item For each $x \in \R$ and each $j = 1, 2, \dots$, the partial derivative $\frac{\partial^j X}{\partial k^j}(x, \cdot)$ has a continuous extension to $(\omega^2 \bar{\mathrm{S}}, \omega \bar{\mathrm{S}}, \bar{\mathrm{S}})\setminus \{0\}$.

\item For each $x \in \R$ and each $j = 1, 2, \dots$, the partial derivative $\frac{\partial^j Y}{\partial k^j}(x, \cdot)$ has a continuous extension to $(-\omega^2 \bar{\mathrm{S}}, -\omega \bar{\mathrm{S}}, -\bar{\mathrm{S}})\setminus \{0\}$.

\item For each $n \geq 1$ and $\epsilon > 0$, there are bounded smooth positive functions $f_+(x)$ and $f_-(x)$ of $x \in \R$ with rapid decay as $x \to +\infty$ and $x \to -\infty$, respectively, such that the following estimates hold for $x \in \R$ and $ j = 0, 1, \dots, n$:
\begin{subequations}\label{Xest}
\begin{align}\label{Xesta}
& \bigg|\frac{\partial^j}{\partial k^j}\big(X(x,k) - I\big) \bigg| \leq
f_+(x), \qquad k \in (\omega^2 \bar{\mathrm{S}}, \omega \bar{\mathrm{S}}, \bar{\mathrm{S}}), \ |k| > \epsilon,
	\\ \label{Xestb}
& \bigg|\frac{\partial^j}{\partial k^j}\big(Y(x,k) - I\big) \bigg| \leq
f_-(x), \qquad k \in (-\omega^2 \bar{\mathrm{S}}, -\omega \bar{\mathrm{S}}, -\bar{\mathrm{S}}), \ |k| > \epsilon.
\end{align}
\end{subequations}

\item $X$ and $Y$ obey the following symmetries for each $x \in \R$:
\begin{subequations}\label{XYsymm}
\begin{align}
&  X(x, k) = \mathcal{A} X(x,\omega k)\mathcal{A}^{-1} = \mathcal{B} \overline{X(x,\overline{k})}\mathcal{B}, \qquad k \in (\omega^2 \bar{\mathrm{S}}, \omega \bar{\mathrm{S}}, \bar{\mathrm{S}})\setminus \{0\},
	\\
&  Y(x, k) = \mathcal{A} Y(x,\omega k)\mathcal{A}^{-1} = \mathcal{B} \overline{Y(x,\overline{k})}\mathcal{B}, \qquad k \in (-\omega^2 \bar{\mathrm{S}}, -\omega \bar{\mathrm{S}}, -\bar{\mathrm{S}})\setminus \{0\}.
\end{align}
\end{subequations}

\item If $u_0(x), v_0(x)$ have compact support, then, for each  $x \in \R$,  $X(x, k)$ and $Y(x, k)$ are defined and analytic for $k \in \C \setminus \{0\}$ and $\det X = \det Y = 1$.
\end{enumerate}
\end{proposition}
\begin{proof}
The proof follows from a relatively straightforward analysis of the Volterra equations (\ref{XYdef}); see e.g. \cite{DT1979} or Theorem 3.1 in \cite{HLNonlinearFourier} for similar proofs.
The key point of the argument is as follows. The third columns of the matrix equations (\ref{XYdef}) involve the exponentials
$$e^{(x - x')(l_1 - l_3)}, \qquad e^{(x - x')(l_2 - l_3)}.$$
These exponentials are bounded in the following regions of the complex $k$-plane:
\begin{align*}
\{\re  l_3 \leq \re  l_1\} \cap \{\re  l_3 \leq \re  l_2\} \quad \text{if} \quad x \leq x',
	\\
\{\re  l_1 \leq \re  l_3\} \cap \{\re  l_2 \leq \re  l_3\} \quad \text{if} \quad x \geq x'.
\end{align*}
Since the equations in (\ref{XYdef}) are Volterra integral equations, these boundedness properties imply that the third column vectors of $X$ and $Y$ are bounded and analytic for $k$ in $\mathrm{S}$ and $-\mathrm{S}$, respectively, as long as $k$ stays away from the singularity at $0$.
The symmetries (\ref{XYsymm}) hold because if $F$ denotes one of the $3 \times 3$-matrix valued functions $\mathcal{L}$ and $\mathsf{U}$, then $F(k) = \mathcal{A} F(\omega k)\mathcal{A}^{-1}
 = \mathcal{B} \overline{F(\overline{k})}\mathcal{B}$ for $k \in \C$.
 \end{proof}

\subsection{Asymptotics of $X$ and $Y$ as $k \to \infty$}
We next consider the behavior of the eigenfunctions $X$ and $Y$ as $k \to \infty$. Our goal is to prove Proposition \ref{XYprop2} which essentially states that the asymptotics of $X$ and $Y$ as $k \to \infty$ can be obtained by considering formal power series solutions of (\ref{xpart}). Let us first find these formal solutions.

Equation (\ref{xpart}) admits formal power series solutions
\begin{align*}
& X_{formal}(x,k) = I + \frac{X_1(x)}{k} + \frac{X_2(x)}{k^2} + \cdots,
	\\ \nonumber
& Y_{formal}(x,k) =  I + \frac{Y_1(x)}{k} + \frac{Y_2(x)}{k^2} + \cdots,
\end{align*}
normalized at $x = \infty$ and $x = -\infty$, respectively:
\begin{align}\label{Fjnormalization}
\lim_{x\to \infty} X_j(x) = \lim_{x\to -\infty} Y_j(x) = 0, \qquad j \geq 1.
\end{align}
Indeed, the function $\mathsf{U}$ is of the form
$$\mathsf{U}(x,k) = \frac{\mathsf{U}_{1}(x)}{k} + \frac{\mathsf{U}_{2}(x)}{k^2},$$
where
$$ \mathsf{U}_{1}(x) = \mathsf{U}^{(1)}(x,0), \qquad \mathsf{U}_{2}(x) = \mathsf{U}^{(2)}(x,0).
$$
Substituting
$$X = I + \frac{X_1(x)}{k} + \frac{X_2(x)}{k^2} + \cdots$$
into (\ref{xpart}), the off-diagonal terms of $O(k^{-j})$ and the diagonal terms of $O(k^{-j-1})$ yield the relations
\begin{align}\label{xrecursive}
\begin{cases}
[J, X_{j+1}] = \partial_x X_{j}^{(o)} - (\mathsf{U}_{1}X_{j-1})^{(o)} - (\mathsf{U}_{2}X_{j-2})^{(o)},
	\\
\partial_x X_{j+1}^{(d)} = (\mathsf{U}_{1}X_{j})^{(d)} +(\mathsf{U}_{2}X_{j-1})^{(d)},$$
\end{cases}
\end{align}
since $\mathcal{L} = kJ$ (see \eqref{def of mathcal L and mathcal Z}), and where $A^{(d)}$ and $A^{(o)}$ denote the diagonal and off-diagonal parts of a $3 \times 3$ matrix $A$, respectively.
The coefficients $\{X_j(x), Y_j(x)\}$ are uniquely determined recursively from (\ref{Fjnormalization})-(\ref{xrecursive}), the equations obtained by replacing $\{X_j\}$ with $\{Y_j\}$ in (\ref{xrecursive}), and the initial assignments
$$X_{-2} = Y_{-2} = 0, \qquad X_{-1} = Y_{-1} = 0, \qquad X_0 = Y_0 = I.$$
The first few coefficients are given by
\begin{align}\nonumber
X_1(x) = & \; -\frac{2}{3}  \int_{\infty}^{x} u_0(x^{\prime}) dx' \begin{pmatrix} \omega^2 & 0 & 0 \\ 
0 & \omega & 0 \\ 
0 & 0 & 1
\end{pmatrix},
	\\ \nonumber
X_2(x) = &\; \frac{2 u_0(x)}{3(1-\omega)}\begin{pmatrix} 0 & 1 & -1 \\ 
-\omega & 0 & \omega \\ 
\omega^{2} & -\omega^{2} & 0 \end{pmatrix}
	\\ \nonumber
& - \frac{1}{3} \int_{\infty}^x \left(v_0 + u_{0x} + 2u_0 (X_1)_{33}\right)(x')dx'  \begin{pmatrix} \omega & 0 & 0 \\ 0 & \omega^2 & 0 \\ 0 & 0 & 1 \end{pmatrix},
	\\ \nonumber
Y_1(x) = & -\frac{2}{3}  \int_{-\infty}^x u_0(x^{\prime})dx' \begin{pmatrix} \omega^2 & 0 & 0 \\ 
0 & \omega & 0 \\ 
0 & 0 & 1
\end{pmatrix},
	\\ \nonumber
Y_2(x) = &\; \frac{2 u_0(x)}{3(1-\omega)}\begin{pmatrix} 0 & 1 & -1 \\ 
-\omega & 0 & \omega \\ 
\omega^{2} & -\omega^{2} & 0 \end{pmatrix}
	\\ \nonumber
& - \frac{1}{3} \int_{-\infty}^x \left(v_0 + u_{0x} + 2u_0 (Y_1)_{33}\right)(x')dx'  \begin{pmatrix} \omega & 0 & 0 \\ 0 & \omega^2 & 0 \\ 0 & 0 & 1 \end{pmatrix}.
\end{align}

We can now describe the behavior of $X$ and $Y$ for large $k$.

\begin{proposition}[Asymptotics of $X$ and $Y$ as $k \to \infty$]\label{XYprop2}
Suppose $u_0, v_0 \in \mathcal{S}(\R)$. 
As $k \to \infty$, $X$ and $Y$ coincide to all orders with $X_{formal}$ and $Y_{formal}$, respectively. More precisely, let $p \geq 0$ be an integer. Then the functions
\begin{align}\label{Xpdef}
&X_{(p)}(x,k) := I + \frac{X_1(x)}{k} + \cdots + \frac{X_{p}(x)}{k^{p}},
	\\ \nonumber
&Y_{(p)}(x,k) := I + \frac{Y_1(x)}{k} + \cdots + \frac{Y_{p}(x)}{k^{p}},
\end{align}
are well-defined and, for each integer $j \geq 0$,
\begin{subequations}\label{Xasymptotics}
\begin{align}\label{Xasymptoticsa}
& \bigg|\frac{\partial^j}{\partial k^j}\big(X - X_{(p)}\big) \bigg| \leq
\frac{f_+(x)}{|k|^{p+1}}, \qquad x \in \R, \  k \in (\omega^2 \bar{\mathrm{S}}, \omega \bar{\mathrm{S}}, \bar{\mathrm{S}}), \ |k| \geq 2,
	\\ \label{Xasymptoticsb}
& \bigg|\frac{\partial^j}{\partial k^j}\big(Y - Y_{(p)}\big) \bigg| \leq
\frac{f_-(x)}{|k|^{p+1}}, \qquad x \in \R, \  k \in (-\omega^2 \bar{\mathrm{S}}, -\omega \bar{\mathrm{S}}, -\bar{\mathrm{S}}), \ |k| \geq 2,
\end{align}
\end{subequations}
where $f_+(x)$ and $f_-(x)$ are bounded smooth positive functions of $x \in \R$ with rapid decay as $x \to +\infty$ and $x \to -\infty$, respectively.
\end{proposition}
\begin{proof}
The proof follows by considering the equation satisfied by the quotient $X_{(p)}^{-1} X$; see e.g. Theorem 3.3 in \cite{HLNonlinearFourier} for a similar proof in the context of the sine-Gordon equation. See also the proof of Lemma \ref{Matinftylemma}.
\end{proof}

\subsection{Asymptotics of $X$ and $Y$ as $k \to 0$}
If $u_0, v_0$ have compact support, then the next proposition shows that $X$ and $Y$ have at most double poles at $k = 0$ with residues of the form (\ref{Cjminus1}). If $u_0, v_0 \in \mathcal{S}(\R)$ are not compactly supported, a more careful statement is required because the columns of $X$ and $Y$ are then, in general, not defined in neighborhoods of $k = 0$.

\begin{proposition}[Asymptotics of $X$ and $Y$ as $k \to 0$]\label{XYat1prop}
Suppose $u_0, v_0 \in \mathcal{S}(\R)$ and let $p \geq 0$ be an integer. 
Then there are $3 \times 3$-matrix valued functions $C_i^{(l)}(x)$, $i = 1, 2$, $l = -2,-1, \dots, p$, with the following properties:
\begin{itemize}
\item For $x \in \R$ and $k \in (\omega^2 \bar{\mathrm{S}}, \omega \bar{\mathrm{S}}, \bar{\mathrm{S}})$, the function $X$ satisfies
\begin{subequations}\label{XYat1}
\begin{align}\label{Xat0}
& \bigg|\frac{\partial^j}{\partial k^j}\big(X - I - \sum_{l=-2}^p C_1^{(l)}(x)k^l\big) \bigg| \leq
f_+(x)|k|^{p+1-j}, \qquad |k| \leq \frac{1}{2},
\end{align}
while, for $x \in \R$ and $k \in (-\omega^2 \bar{\mathrm{S}}, -\omega \bar{\mathrm{S}}, -\bar{\mathrm{S}})$, the function $Y$ satisfies
\begin{align}\label{Yat0}
& \bigg|\frac{\partial^j}{\partial k^j}\big(Y - I - \sum_{l=-2}^p C_2^{(l)}(x) k^l\big) \bigg| \leq
f_-(x)|k|^{p+1-j}, \qquad |k| \leq \frac{1}{2},
\end{align}
\end{subequations}
where $f_+(x)$ and $f_-(x)$ are smooth positive functions of $x \in \R$ with rapid decay as $x \to +\infty$ and $x \to -\infty$, respectively, and $j \geq 0$ is any integer.

\item For each $l \geq -2$, $C_1^{(l)}(x)$ and $C_2^{(l)}(x)$ are smooth functions of $x \in \R$ which have rapid decay as $x \to +\infty$ and $x \to -\infty$, respectively.

\item The leading coefficients have the form
\begin{align}
 C_{i}^{(-2)}(x)
= &\;
\alpha_{i}(x)\begin{pmatrix}
\omega &  \omega^{2} & 1  \\
\omega &  \omega^{2} & 1  \\
\omega &  \omega^{2} & 1  
\end{pmatrix}, \label{Cjminus2} \\
 C_{i}^{(-1)}(x) = &\; \beta_{i}(x) \begin{pmatrix}
\omega^{2} & \omega & 1 \\
\omega^{2} & \omega & 1 \\
\omega^{2} & \omega & 1 
\end{pmatrix} + \gamma_{i}(x) \begin{pmatrix}
\omega^{2} & 1 & \omega \\
1 & \omega & \omega^{2} \\
\omega & \omega^{2} & 1 
\end{pmatrix} \label{Cjminus1}
	\\
C_{i}^{(0)}(x) = & -I + \delta_{i,1}(x)\begin{pmatrix}
1 & 1 & 1 \\
1 & 1 & 1 \\
1 & 1 & 1
\end{pmatrix} \nonumber \\
& + \delta_{i,2}(x) \begin{pmatrix}
1 & \omega^{2} & \omega \\
\omega & 1 & \omega^{2} \\
\omega^{2} & \omega & 1
\end{pmatrix} + \delta_{i,3}(x) \begin{pmatrix}
1 & \omega & \omega^{2} \\
\omega^{2} & 1 & \omega \\
\omega & \omega^{2} & 1
\end{pmatrix} \label{Cjminus0}
\end{align}
where $\alpha_i(x)$, $\beta_i(x)$, $\gamma_i(x)$, $\delta_{i,j}(x)$, $i = 1,2$, $j=1,2,3$, are  real-valued functions of $x \in \R$. Furthermore, there exists two bounded functions $f_{1}(x)$ and $f_{2}(x)$, with rapid decay at $+\infty$ and $-\infty$ respectively, such that
\begin{align}
& |\alpha_{i}(x)| \leq f_{i}(x), \label{bound alphai}
	\\
& |\beta_{i}(x)| + |\gamma_{i}(x)| \leq (1+|x|)f_{i}(x), \label{bound betai}
	\\
& |\delta_{i,j}(x)-\tfrac{1}{3}| \leq (1+|x|)^{2}f_{i}(x),\label{bound deltai}
\end{align}
for all $x \in \mathbb{R}$, $i=1,2$ and $j=1,2,3$.
\end{itemize}
\end{proposition}
\begin{proof}
The function $\mathcal{X} := PX$ satisfies the integral equation
\begin{align}\nonumber
& \mathcal{X}(x,k) = P(k) - \int_x^\infty P(k)e^{(x - x')\widehat{\mathcal{L}(k)}}(P(k)^{-1} \tilde{\mathsf{U}}(x')\mathcal{X}(x',k)) dx', 
	\\ \label{tildeXeq}
& \hspace{6cm} x \in \R, \ k \in (\omega^2 \bar{\mathrm{S}}, \omega \bar{\mathrm{S}}, \bar{\mathrm{S}}) \setminus \{0\},
\end{align}
where 
\begin{equation}\label{def of mathsfU tilde}
\tilde{\mathsf{U}} := P\mathsf{U}P^{-1} = \begin{pmatrix}
0 & 0 & 0 \\ 
0 & 0 & 0 \\
-u_{0x}-v_{0} & -2u_{0} & 0
\end{pmatrix}
\end{equation}
is independent of $k$.
A computation shows that the kernel 
\begin{align*}
\mathcal{P}(x,x',k) := P(k)e^{(x - x')\mathcal{L}(k)}P(k)^{-1}
\end{align*}
is analytic at $k=0$. Moreover, the function $P(k)$ is analytic for $k \in \C$.
Thus, an analysis of the Volterra equation (\ref{tildeXeq}) as in Proposition \ref{XYprop} shows that $\mathcal{X}(x,k)$ is analytic for $k \in (\omega^2 \mathrm{S}, \omega \mathrm{S}, \mathrm{S})$ and that $\mathcal{X}$ and all its $k$-derivatives have continuous extensions to $(\omega^2 \bar{\mathrm{S}}, \omega \bar{\mathrm{S}}, \bar{\mathrm{S}})$.
In particular, $\mathcal{X}$ admits the Taylor expansion
\begin{align*}
& \mathcal{X}(x,k) = \mathcal{X}(x,0) + \partial_k \mathcal{X}(x,0)k + \frac{1}{2}\partial_k^2\mathcal{X}(x,0)k^2 + \cdots, 
	\\
& \hspace{8cm} k \to 0, \ k \in (\omega^2 \bar{\mathrm{S}}, \omega \bar{\mathrm{S}}, \bar{\mathrm{S}}),
\end{align*}
where the coefficients are smooth functions of $x\in \R$. The derivative $\partial_k^j\mathcal{X}(x,k)$ converges rapidly to $\partial_k^jP(k)$ as $x \to +\infty$ for each $j \geq 0$ and $k \in (\omega^2 \bar{\mathrm{S}}, \omega \bar{\mathrm{S}}, \bar{\mathrm{S}})$. Let us analyze the behavior of $\partial_k^j\mathcal{X}(x,0)$, $j=0,1,2$, as $x \to - \infty$. A simple computation shows that
\begin{align*}
\mathcal{P}(x,x',0) = \begin{pmatrix}
1 & x-x' & \frac{(x-x')^{2}}{2} \\
0 & 1 & x-x' \\
0 & 0 & 1
\end{pmatrix},
\end{align*}
and thus, using \eqref{def of mathsfU tilde}, 
\begin{align*}
\mathcal{P}(x,x',0) \tilde{\mathsf{U}}(x') = \begin{pmatrix}
- \frac{1}{2}(u_{0x}(x')+v_{0}(x'))(x-x')^{2} & -u_{0}(x')(x-x')^{2} & 0 \\
- (u_{0x}(x')+v_{0}(x'))(x-x') & -2u_{0}(x')(x-x') & 0 \\
- (u_{0x}(x')+v_{0}(x')) & -2u_{0}(x') & 0 
\end{pmatrix},
\end{align*}
from which we deduce (from a standard analysis of the associated Volterra equation) that
\begin{align}\label{boundedness of mathsfX at k=0}
\mathcal{X}(x,0) = \begin{pmatrix}
O(x^{2}) & O(x^{2}) & O(x^{2}) \\
O(x) & O(x) & O(x) \\
O(1) & O(1) & O(1)
\end{pmatrix}, \qquad \mbox{as } x \to -\infty.
\end{align}
There are similar estimates for the rows of $\partial_{k}\mathcal{X}(x,0)$ and $\partial_{k}^{2}\mathcal{X}(x,0)$. First, note that
\begin{align*}
\partial_{k}\mathcal{P}(x,x',0) = \partial_{k}^{2}\mathcal{P}(x,x',0) = \begin{pmatrix}
0 & 0 & 0 \\
0 & 0 & 0 \\
0 & 0 & 0
\end{pmatrix},
\end{align*}
which implies (from an analysis of the associated Volterra equations) that
\begin{align}\label{boundedness of derivative of mathsfX at k=0}
\partial_{k}\mathcal{X}(x,0) = \begin{pmatrix}
O(x^{3}) & O(x^{3}) & O(x^{3}) \\
O(x^{2}) & O(x^{2}) & O(x^{2}) \\
O(x) & O(x) & O(x)
\end{pmatrix}, \quad \partial_{k}^{2}\mathcal{X}(x,0) = \begin{pmatrix}
O(x^{4}) & O(x^{4}) & O(x^{4}) \\
O(x^{3}) & O(x^{3}) & O(x^{3}) \\
O(x^{2}) & O(x^{2}) & O(x^{2})
\end{pmatrix},
\end{align}
as $x \to - \infty$.
On the other hand, $P(k)^{-1}$ has a double pole at $k = 0$:
\begin{align}\label{Pinvat1}
& P(k)^{-1} = 
\frac{P^{(-2)}}{k^{2}}+\frac{P^{(-1)}}{k}
+ P^{(0)},
\end{align}
where
$$P^{(-2)} =  \frac{1}{3}\begin{pmatrix}  
0 & 0 & 1 \\
0 & 0 & 1 \\
0 & 0 & 1
\end{pmatrix}, \qquad P^{(-1)} = \frac{1}{3}\begin{pmatrix}  
0 & \omega & 0 \\
0 & \omega^{2} & 0 \\
0 & 1 & 0
\end{pmatrix}, \qquad P^{(0)} = \frac{1}{3}\begin{pmatrix}  
\omega^{2} & 0 & 0 \\
\omega & 0 & 0 \\
1 & 0 & 0
\end{pmatrix}.$$
It follows that $X = P^{-1} \mathcal{X}$ has at most a double pole at $k = 0$ and that $X$ admits an expansion of the form
\begin{align*}
& X(x,k) =  \frac{C_1^{(-2)}(x)}{k^{2}}+\frac{C_1^{(-1)}(x)}{k} + I + C_1^{(0)}(x) + C_1^{(1)}(x)k + \cdots, \quad k \in (\omega^2 \bar{\mathrm{S}}, \omega \bar{\mathrm{S}}, \bar{\mathrm{S}}),
\end{align*}
as $k \to 0$ with 
\begin{align*}
C_{1}^{(-2)}(x)  = P^{(-2)} \mathcal{X}(x,0), \quad 
C_{1}^{(-1)}(x)  = P^{(-2)} \partial_k\mathcal{X}(x,0) + P^{(-1)}\mathcal{X}(x,0), \\
C_{1}^{(0)}(x) = \frac{1}{2}P^{(-2)}\partial_{k}^{2}\mathcal{X}(x,0) + P^{(-1)}\partial_{k} \mathcal{X}(x,0) + P^{(0)}\mathcal{X}(x,0)-I, \quad \text{etc}. 
\end{align*}
Using that $\partial_k^j\mathcal{X}(x,k) \to \partial_k^jP(k)$ rapidly as $x \to +\infty$ and equating powers of $k$ in the identity
$$I = P^{-1} P = \Big(\frac{P^{(-2)}}{k^{2}}+\frac{P^{(-1)}}{k}
+ P^{(0)}\Big)(P(0) + P'(0)k + \frac{1}{2}P''(0)k^{2}),$$
we find that the coefficients $C_1^{(l)}(x)$, $l \geq -2$, vanish rapidly as $x \to +\infty$.
From the exact form of $P^{(-2)}$, it follows that there exist complex-valued functions $a_1(x)$, $a_2(x)$, $a_3(x)$ such that 
$$C_{1}^{(-2)}(x) = \begin{pmatrix}
a_1(x) &  a_2(x) & a_3(x)  \\
a_1(x) &  a_2(x) & a_3(x)  \\
a_1(x) &  a_2(x) & a_3(x)  
\end{pmatrix}.$$
The symmetry $X(x,k) = \mathcal{A}X(x,\omega k)\mathcal{A}^{-1}$ implies $a_{1} = \omega a_{3}$ and $a_{2} = \omega^{2} a_{3}$, and the symmetry $X(x, k) =  \mathcal{B} \overline{X(x,\overline{k})}\mathcal{B}$ implies that $a_{3} = \bar{a}_3$. Since the third row of $\mathcal{X}(x,0)$ is bounded for all $x \in \mathbb{R}$ by \eqref{boundedness of mathsfX at k=0}, so is $C_{1}^{(-2)}(x)$. This completes the proof of \eqref{Cjminus2} with $i=1$ and $\alpha_{1} = a_{3}$. Similarly, the expressions for $P^{(-2)}$ and $P^{(-1)}$ imply that there exist complex-valued functions $\{b_{i}\}_{i=1}^{6}$ such that
\begin{align*}
C_{1}^{(-1)}(x) & = P^{(-2)} \partial_k\mathcal{X}(x,0) + P^{(-1)}\mathcal{X}(x,0) \\
& = \begin{pmatrix}
b_1(x) &  b_2(x) & b_3(x)  \\
b_1(x) &  b_2(x) & b_3(x)  \\
b_1(x) &  b_2(x) & b_3(x)
\end{pmatrix} + 
\begin{pmatrix}
\omega b_{4}(x) & \omega b_{5}(x) & \omega b_{6}(x) \\
\omega^{2} b_{4}(x) & \omega^{2} b_{5}(x) & \omega^{2} b_{6}(x) \\
 b_{4}(x) & b_{5}(x) & b_{6}(x)
\end{pmatrix}.
\end{align*}
Furthermore, since the second row of $\mathcal{X}(x,0)$ and the third row of $\partial_{k}\mathcal{X}(x,0)$ vanish rapidly as $x \to + \infty$, and grow at most linearly as $x \to -\infty$ by \eqref{boundedness of mathsfX at k=0} and \eqref{boundedness of derivative of mathsfX at k=0}, we conclude that there exists a bounded positive function $f_{1}(x)$ with rapid decay as $x \to + \infty$ such that $b_{i}(x) \leq f_{1}(x)(1+|x|)$ for all $x \in \mathbb{R}$. The symmetry $X(x,k) = \mathcal{A}X(x,\omega k)\mathcal{A}^{-1}$ implies
$$
b_{1} = \omega^{2}b_{3}, \quad b_{2} = \omega b_{3}, \quad b_{4} = \omega b_{6}, \quad b_{5} = \omega^{2}b_{6},
$$
and the symmetry $X(x, k) =  \mathcal{B} \overline{X(x,\overline{k})}\mathcal{B}$ implies $b_{3} = \bar{b}_{3}$ and $b_{6} = \bar{b}_{6}$. This completes the proof of \eqref{Cjminus1} with $i=1$, $\beta_{1} = b_{3}$ and $\gamma_{1} = b_{6}$. The proofs for $C_{1}^{(0)}$ and for $Y$ are analogous.
\end{proof}

\begin{remark}
If $u_0$ and $v_0$ have compact support, then all entries of $X$ and  $Y$ are defined and analytic in punctured neighborhoods of $k = 0$ and the content of (\ref{XYat1}) can be expressed more simply as the Laurent series identities
\begin{align}
\begin{cases}
 X(x,k) = I + \sum_{l=-2}^\infty C_1^{(l)}(x)k^l, \qquad \text{$k$ near $0$},
	\\ 
 Y(x,k) = I + \sum_{l=-2}^\infty C_2^{(l)}(x)k^l, \qquad \text{$k$ near $0$}.
 \end{cases}
\end{align}
\end{remark}

\subsection{The spectral function $s(k)$}
The spectral function $s(k)$ is defined by
\begin{equation}\label{sdef}
s(k) = I - \int_{\mathbb{R}}e^{-x \widehat{\mathcal{L}(k)}}(\mathsf{U}X)(x,k)dx.
\end{equation}
Let $\R_+ = (0,+\infty)$ and $\R_- = (-\infty, 0)$.

\begin{proposition}[Properties of $s(k)$]\label{sprop}
Suppose $u_0,v_0 \in \mathcal{S}(\R)$. 
Then the spectral function $s(k)$ defined in (\ref{sdef}) has the following properties:
\begin{enumerate}[$(a)$]
\item The entries of $s(k)$ are defined and continuous for $k$ in
\begin{align}\label{sdomainofdefinition}
 \begin{pmatrix}
 \omega^2 \bar{\mathrm{S}} & \R_+ & \omega \R_+ \\
 \R_+ & \omega \bar{\mathrm{S}} & \omega^2 \R_+ \\
 \omega \R_+ & \omega^2 \R_+ & \bar{\mathrm{S}}
 \end{pmatrix}\setminus \{0\},
\end{align}
that is, the $(11)$ entry of $s(k)$ is defined and continuous for $k \in \omega^2 \bar{\mathrm{S}}\setminus \{0\}$, etc. 
 
\item The diagonal entries of $s(k)$  are analytic in the interior of their domains of definition as given in (\ref{sdomainofdefinition}). 
 
\item For $j = 1, 2, \dots$, the derivative $\partial_k^js(k)$ is well-defined and continuous for $k$ in (\ref{sdomainofdefinition}).

\item $s(k)$ obeys the symmetries
\begin{align}\nonumber
&  s(k) = \mathcal{A} s(\omega k)\mathcal{A}^{-1} = \mathcal{B} \overline{s(\overline{k})}\mathcal{B}.
\end{align}

\item $s(k)$ approaches the identity matrix as $k \to \infty$. More precisely, there are diagonal matrices $\{s_j\}_1^\infty$ such that
\begin{align*}\nonumber
\Big|\partial_k^j \Big(s(k) - I - \sum_{j=1}^N \frac{s_j}{k^j}\Big)\Big| = O(k^{-N-1}), \quad& k \to \infty,
\ k \in \begin{pmatrix}
 \omega^2 \bar{\mathrm{S}} & \R_+ & \omega \R_+ \\
 \R_+ & \omega \bar{\mathrm{S}} & \omega^2 \R_+ \\
 \omega \R_+ & \omega^2 \R_+ & \bar{\mathrm{S}}
 \end{pmatrix},
\end{align*}
for $j = 0, 1, \dots, N$ and each integer $N \geq 1$. In particular, the off-diagonal entries of $s(k)$ have rapid decay as $k \to \infty$.

\item As $k \to 0$, 
\begin{align}\label{s at 0}
s(k) = \frac{s^{(-2)}}{k^{2}} + \frac{s^{(-1)}}{k} + s^{(0)} + s^{(1)}k + \ldots, \quad k \in \begin{pmatrix}
 \omega^2 \bar{\mathrm{S}} & \R_+ & \omega \R_+ \\
 \R_+ & \omega \bar{\mathrm{S}} & \omega^2 \R_+ \\
 \omega \R_+ & \omega^2 \R_+ & \bar{\mathrm{S}}
 \end{pmatrix},
\end{align}
where
\begin{align}\label{spm2p}
s^{(-2)} =  \mathfrak{s}^{(-2)} \begin{pmatrix}
\omega & \omega^{2} & 1 \\
\omega & \omega^{2} & 1 \\
\omega & \omega^{2} & 1 
\end{pmatrix}, \quad \mathfrak{s}^{(-2)} := \int_{\mathbb{R}} \big( 2 u_{0} \gamma_{1} + (u_{0x}+v_{0})\delta_{1,3} \big) dx \in \mathbb{R},
\end{align}
and the expansion can be differentiated termwise any number of times.
\item If $u_0(x), v_0(x)$ have compact support, then $s(k)$ is defined and analytic for $k \in \C \setminus \{0\}$, $\det s = 1$ for $k \in \C \setminus \{0\}$, and
\begin{align}\label{XYs} 
X(x,k) = Y(x,k)e^{x\widehat{\mathcal{L}(k)}} s(k), \qquad k \in \C  \setminus \{0\}.
\end{align}
\end{enumerate}
\end{proposition}
\begin{proof}
The definition of the $(ij)$th entry of $s(k)$ involves the exponential factor $e^{-x(l_i - l_j)}$. Properties $(a)$-$(d)$ follow by using the boundedness properties of this factor together with the properties of $X$ given in Proposition \ref{XYprop} in the definition (\ref{sdef}) of $s(k)$. 
To prove $(e)$, we let $p \geq 1$ be an integer and let $X_{(p)}(x,k) = \sum_{j=0}^{p} X_j(x) k^{-j}$, $X_0 \equiv I$, be the function in (\ref{Xpdef}). Using (\ref{Xasymptoticsa}), we can replace $X$ with $X_{(p)}$ in (\ref{sdef}) with an error of order $O(k^{-p-1})$. Thus, as $k \to \infty$ within the domain of definition (\ref{sdomainofdefinition}), we have
\begin{align}\label{skjXj}
& s(k) = I - \sum_{j=0}^p \frac{1}{k^j} \int_\R e^{-x\widehat{\mathcal{L}(k)}}\big(\mathsf{U}(x,k)X_j(x)\big)dx
+ O(k^{-p-1}).
 \end{align}
Since the functions $X_j(x)$ and their derivatives are bounded and $\mathsf{U}(\cdot,k) \in \mathcal{S}(\R)$, and since $l_{j}(k)^{-1} = O(k^{-1})$ as $k \to \infty$ for each $j=1,2,3$, integration by parts gives $(e)$ for the off-diagonal elements of $s(k)$ in the case $k \to \infty$. For the diagonal elements, the exponential factor is absent from the integral in (\ref{skjXj}), and we obtain (e) for $k \to \infty$ by substituting in \eqref{skjXj} the expression \eqref{expression for U} of $\mathsf{U}$ and expanding.

Note that we can rewrite $s(k)$ in terms of $\tilde{\mathsf{U}}$ and $\mathcal{X}$ as follows:
\begin{equation}\label{s in terms of U tilde}
s(k) = I - \int_{\mathbb{R}} e^{-x \widehat{\mathcal{L}(k)}} \big(P(k)^{-1} \tilde{\mathsf{U}}(x) \mathcal{X}(x,k)\big)dx.
\end{equation} 
Substituting the expansion of $\mathcal{X}$ as $k \to 0$ into the expression \eqref{s in terms of U tilde} for $s(k)$, we have
\begin{align}\label{sIintRe}
& s(k) = I - \int_{\mathbb{R}} e^{-x \widehat{\mathcal{L}(k)}}\bigg[\Big(\frac{P^{(-2)}}{k^{2}}+\frac{P^{(-1)}}{k}
+ P^{(0)}\Big)\tilde{\mathsf{U}}(x) \Big( \mathcal{X}(x,0)+\partial_{k}\mathcal{X}(x,0)k+\cdots\Big)\bigg]dx
\end{align}
which allows us to find the expansion \eqref{s at 0}. In particular, we have
\begin{align}\label{1112entries}
s^{(-2)} = - \int_\R P^{(-2)} \tilde{\mathsf{U}}(x)\mathcal{X}(x,0)dx.
\end{align}
The diagonal entries of (\ref{1112entries}) follow immediately from (\ref{sIintRe}), but the off-diagonal entries deserve some explanation, and we provide the argument for the $(12)$ entry (similar arguments apply to the other off-diagonal entries): Since $l_1 - l_2 \in i\R$ for $k \in \R$ and $l_1 - l_2 = O(k)$, we have, for any $f \in \mathcal{S}(\R)$, 
\begin{align}\label{exl1l2est}
\bigg|\int_\R (e^{-x(l_1-l_2)} -1) f(x) dx\bigg|
\leq \int_{\R} |x(l_2-l_1) f(x)| dx = O(k), \qquad k \to 0, 
\end{align}
and using this estimate $e^{-x(l_1-l_2)}$ can be replaced by $1$ in the derivation of (\ref{1112entries}). 

Next we use \eqref{Cjminus2}-\eqref{Cjminus0} to obtain
\begin{align}\label{mathsfX at 0}
\mathcal{X}(x,0) = 3 \begin{pmatrix}
\omega \delta_{1,3} & \omega^{2} \delta_{1,3} & \delta_{1,3} \\
\omega \gamma_{1} & \omega^{2} \gamma_{1} & \gamma_{1} \\
\omega \alpha_{1} & \omega^{2} \alpha_{1} & \alpha_{1}
\end{pmatrix},
\end{align}
and we deduce the expression \eqref{spm2p} by substituting \eqref{mathsfX at 0}, as well as the expressions for $P^{(-2)}$ and $\tilde{\mathsf{U}}$ given by \eqref{def of mathsfU tilde} and \eqref{Pinvat1}, into \eqref{1112entries}.

Assume now that $u_0(x), v_0(x)$ have compact support. Then the integral in (\ref{sdef}) converges for all $k \in \C \setminus \{0\}$, so that all entries of $X(x,k)$, $Y(x,k)$, and $s(k)$ are well-defined and analytic for $k \in \C \setminus \{0\}$. Since both $X$ and $Y$ solve (\ref{xpart}), it follows that there exists a function $s(k)$ which is independent of $x$ and such that (\ref{XYs}) holds. We determine $s(k)$ by rewriting (\ref{XYs}) as 
$$s(k) = e^{-x\hat{\mathcal{L}}}(Y^{-1}X)
= (e^{-x\hat{\mathcal{L}}}(Y^{-1}))\bigg(I - \int_x^{\infty} e^{-x'\widehat{\mathcal{L}(k)}} (\mathsf{U}X)(x',k) dx'\bigg).$$ 
Taking the limit $x \to -\infty$ and using that $Y(x,k) = I$ and $\mathsf{U}(x,k) = 0$ for all sufficiently large negative $x$, it follows that $s(k)$ is given by (\ref{sdef}). This proves (\ref{XYs}).
Since both  $X$ and $Y$ have unit determinant, we find $\det s = \det(Y^{-1}X) = 1$.
\end{proof}

\subsection{The cofactor matrix}
If a $3 \times 3$ matrix $B$ has unit determinant then $(B^{-1})^T = B^A$, where $B^A$ denotes the cofactor matrix of $B$ defined by
\begin{align}\label{cofactordef}
B^{A} = \begin{pmatrix} m_{11}(B) & -m_{12}(B) & m_{13}(B) \\
-m_{21}(B) & m_{22}(B) & -m_{23}(B) \\
m_{31}(B) & -m_{32}(B) & m_{33}(B)
\end{pmatrix}.
\end{align}
Here $m_{ij}(B)$ denotes the $(ij)$th minor $B$, i.e., $m_{ij}(B)$ equals the determinant of the $2 \times 2$-matrix obtained by deleting the $i$th row and the $j$th column from $B$.

Assume for the moment that $u_0, v_0 \in \mathcal{S}(\R)$ are compactly supported so that all entries of $X(x,k)$, $Y(x,k)$, and $s(k)$ are defined for $k \in \C \setminus \{0\}$. Using the relation $(X^{A})_{x} = -X^{A}(X_{x})^{T}X^{A}$ in (\ref{xpart}), we see that $X^A = (X^{-1})^T$ satisfies
\begin{align}\label{xpartA}
(X^A)_x + [\mathcal{L}, X^A] = -\mathsf{U}^TX^A.
\end{align}
Using that $X^A \to I$ as $x \to +\infty$, we conclude that $X^A$ satisfies the following Volterra integral equation
\begin{equation}\label{XXAdefb}
X^{A}(x,k) = I + \int_{x}^{\infty}e^{-(x-x')\widehat{\mathcal{L}(k)}}(\mathsf{U}^{T}X^{A})(x',k)dx'.
\end{equation}
Moreover, the cofactor matrix $s^A = (s^{-1})^T$ is well-defined for $k \in \C \setminus \{0\}$ and the definition \eqref{r1r2def} of $r_2(k)$ can be written as
\begin{align}\nonumber
& r_2(k) = -\frac{m_{12}(s(k))}{m_{11}(s(k))} = \frac{s_{23}(k)s_{31}(k) - s_{21}(k)s_{33}(k)}{s_{22}(k)s_{33}(k) - s_{23}(k)s_{32}(k)}, 
	\\\label{r2expression}
& \hspace{3cm} k \in \R_-, \ \text{$u_0$ and $v_0$ compactly supported}.
\end{align}
From this equation and Proposition \ref{sprop}, we can deduce various properties of $r_2(k)$ if $u_0, v_0$ are compactly supported. 

If $u_0, v_0$ are not compactly supported, then, in general, the entries $s_{ij}$ of $s$ appearing on the right-hand side of (\ref{r2expression}) are not all defined for $k \in \R_-$. Therefore, the above proof needs to be changed. Our strategy will be to show that the matrix $s^A$ is well-defined in terms of  $X^A$ by the following expression
\begin{equation}\label{sAdef}
s^{A}(k) = I + \int_{\mathbb{R}} e^{x \widehat{\mathcal{L}(k)}}(\mathsf{U}^{T}X^{A})(x,k)dx,
\end{equation}
even if $u_0, v_0$ are not compactly supported. We will keep the notation $s^A$ for the function defined in (\ref{sAdef}) even if the minors of $s$ are not defined in the usual sense.

\subsection{The eigenfunctions $X^A$ and $Y^A$} 
We define two $3 \times 3$-matrix valued solutions $X^A(x,k)$ and $Y^A(x,k)$ of (\ref{xpartA}) as the solutions of the linear Volterra integral equations
\begin{subequations}\label{XAYAdef}
\begin{align}  
 & X^A(x,k) = I + \int_x^{\infty} e^{-(x-x')\widehat{\mathcal{L}(k)}} (\mathsf{U}^T X^A)(x',k) dx',	
  	\\ 
 & Y^A(x,k) = I - \int_{-\infty}^x e^{-(x-x')\widehat{\mathcal{L}(k)}} (\mathsf{U}^T Y^A)(x',k) dx'.
\end{align}
\end{subequations}
The same kind of analysis that led to Propositions \ref{XYprop} and \ref{XYprop2} gives the following results for $X^A$ and $Y^A$. 
\begin{proposition}[Basic properties of $X^A$ and $Y^A$]\label{XAYAprop}
Suppose $u_0, v_0 \in \mathcal{S}(\R)$. 
Then the equations (\ref{XAYAdef}) uniquely define two $3 \times 3$-matrix valued solutions $X^A$ and $Y^A$ of (\ref{xpartA}) with the following properties:
\begin{enumerate}[$(a)$]
\item The function $X^A(x, k)$ is defined for $x \in \R$ and $k \in (-\omega^2 \bar{\mathrm{S}}, -\omega \bar{\mathrm{S}}, -\bar{\mathrm{S}}) \setminus \{0\}$. For each $k \in (-\omega^2 \bar{\mathrm{S}}, -\omega \bar{\mathrm{S}}, -\bar{\mathrm{S}}) \setminus \{0\}$, $X^A(\cdot, k)$ is smooth and satisfies (\ref{xpartA}).

\item The function $Y^A(x, k)$ is defined for $x \in \R$ and $k \in (\omega^2 \bar{\mathrm{S}}, \omega \bar{\mathrm{S}}, \bar{\mathrm{S}}) \setminus \{0\}$. For each $k \in (\omega^2 \bar{\mathrm{S}}, \omega \bar{\mathrm{S}}, \bar{\mathrm{S}}) \setminus \{0\}$, $Y^A(\cdot, k)$ is smooth and satisfies (\ref{xpartA}).

\item For each $x \in \R$, the function $X^A(x,\cdot)$ is continuous for $k \in (-\omega^2 \bar{\mathrm{S}}, -\omega \bar{\mathrm{S}}, -\bar{\mathrm{S}})\setminus \{0\}$ and analytic for $k \in (-\omega^2 \mathrm{S}, -\omega \mathrm{S}, -\mathrm{S}) \setminus \{0\}$.

\item For each $x \in \R$, the function $Y^A(x,\cdot)$ is continuous for $k \in (\omega^2 \bar{\mathrm{S}}, \omega \bar{\mathrm{S}}, \bar{\mathrm{S}})\setminus \{0\}$ and analytic for $k \in (\omega^2 \mathrm{S}, \omega \mathrm{S}, \mathrm{S}) \setminus \{0\}$.

\item For each $x \in \R$ and each $j = 1, 2, \dots$, the partial derivative $\frac{\partial^j X^A}{\partial k^j}(x, \cdot)$ has a continuous extension to $(-\omega^2 \bar{\mathrm{S}}, -\omega \bar{\mathrm{S}}, -\bar{\mathrm{S}})\setminus \{0\}$.

\item For each $x \in \R$ and each $j = 1, 2, \dots$, the partial derivative $\frac{\partial^j Y^A}{\partial k^j}(x, \cdot)$ has a continuous extension to $(\omega^2 \bar{\mathrm{S}}, \omega \bar{\mathrm{S}}, \bar{\mathrm{S}})\setminus \{0\}$.

\item For each $n \geq 1$ and $\epsilon > 0$, there are bounded smooth positive functions $f_+(x)$ and $f_-(x)$ of $x \in \R$ with rapid decay as $x \to +\infty$ and $x \to -\infty$, respectively, such that the following estimates hold for $x \in \R$ and $j = 0, 1, \dots, n$:
\begin{align*}
& \bigg|\frac{\partial^j}{\partial k^j}\big(X^A(x,k) - I\big) \bigg| \leq
f_+(x), \qquad k \in (-\omega^2 \bar{\mathrm{S}}, -\omega \bar{\mathrm{S}}, 
-\bar{\mathrm{S}}), \ |k| > \epsilon,
	\\ 
& \bigg|\frac{\partial^j}{\partial k^j}\big(Y^A(x,k) - I\big) \bigg| \leq
f_-(x), \qquad k \in (\omega^2 \bar{\mathrm{S}}, \omega \bar{\mathrm{S}}, \bar{\mathrm{S}}), \ |k| > \epsilon.
\end{align*}

\item $X^A$ and $Y^A$ obey the following symmetries for each $x \in \R$:
\begin{align}\nonumber
&  X^A(x, k) = \mathcal{A} X^A(x,\omega k)\mathcal{A}^{-1} = \mathcal{B} \overline{X^A(x,\overline{k})}\mathcal{B}, 
	\\ \label{XAsymm}
&\hspace{7cm} k \in (-\omega^2 \bar{\mathrm{S}}, -\omega \bar{\mathrm{S}}, -\bar{\mathrm{S}})\setminus \{0\},
	\\\nonumber
&  Y^A(x, k) = \mathcal{A} Y^A(x,\omega k)\mathcal{A}^{-1} = \mathcal{B} \overline{Y^A(x,\overline{k})}\mathcal{B}, 
	\\
&\hspace{7cm} k \in (\omega^2 \bar{\mathrm{S}}, \omega \bar{\mathrm{S}}, \bar{\mathrm{S}})\setminus \{0\},
\end{align}
\item If $u_0(x), v_0(x)$ have compact support, then, for each  $x \in \R$,  $X^A(x, k)$ and $Y^A(x, k)$ are defined and analytic for $k \in \C \setminus \{0\}$ and $\det X^A = \det Y^A = 1$.
\end{enumerate}
\end{proposition}

Proposition \ref{XAYAprop} shows that the entries of the cofactor matrix $X^A$ have larger domains of definitions than suggested by their definitions as minors of $X$. For example, by Proposition \ref{XYprop}, the second and third columns of $X$ are defined for $k \in \omega \bar{\mathrm{S}}\setminus \{0\}$ and $k \in \bar{\mathrm{S}}\setminus \{0\}$, respectively. The minors in the first column of the cofactor matrix $X^A$ are therefore defined for $k \in (\omega \bar{\mathrm{S}} \cap \bar{\mathrm{S}}) \setminus \{0\} = \omega^{2}\mathbb{R}_{+}$. However, Proposition \ref{XAYAprop} shows that the integral equation (\ref{XXAdefb}) actually defines the first column of $X^A$ for all $k$ in the larger set $-\omega^2 \bar{\mathrm{S}}\setminus \{0\}$. By uniqueness, these two definitions of $X^A$ coincide for $k \in \omega^{2}\mathbb{R}_{+}$. The point is that a combination such as $m_{11}(X) = X_{22}X_{33} - X_{23}X_{32}$ can be analytically extended to all of $-\omega^2\mathrm{S}\setminus \{0\}$ even if the individual factors $\{X_{ij}\}_{i,j=2,3}$ cannot.

\begin{proposition}[Asymptotics of $X^A$ and $Y^A$ as $k \to \infty$]\label{XAYAprop2}
Suppose $u_0, v_0 \in \mathcal{S}(\R)$. 
As $k \to \infty$, $X^A$ and $Y^A$ coincide to all orders with $X^A_{formal}$ and $Y^A_{formal}$, respectively. More precisely, let $p \geq 1$ be an integer and let $X^A_{(p)}(x,k)$ and $Y^A_{(p)}(x,k)$ be the cofactor matrices of the functions in (\ref{Xpdef}). Then, for each integer $j \geq 0$,
\begin{subequations}\label{XAasymptotics}
\begin{align}\label{XAasymptoticsa}
& \bigg|\frac{\partial^j}{\partial k^j}\big(X^A - X^A_{(p)}\big) \bigg| \leq
\frac{f_+(x)}{|k|^{p+1}}, \qquad x \in \R, \  k \in (-\omega^2 \bar{\mathrm{S}}, -\omega \bar{\mathrm{S}}, -\bar{\mathrm{S}}), \ |k| \geq 2,
	\\ \label{XAasymptoticsb}
& \bigg|\frac{\partial^j}{\partial k^j}\big(Y^A - Y^A_{(p)}\big) \bigg| \leq
\frac{f_-(x)}{|k|^{p+1}}, \qquad x \in \R, \  k \in (\omega^2 \bar{\mathrm{S}}, \omega \bar{\mathrm{S}}, \bar{\mathrm{S}}), \ |k| \geq 2,
\end{align}
\end{subequations}
where $f_+(x)$ and $f_-(x)$ are bounded smooth positive functions of $x \in \R$ with rapid decay as $x \to +\infty$ and $x \to -\infty$, respectively.
\end{proposition}

\begin{proposition}[Asymptotics of $X^A$ and $Y^A$ as $k \to 0$]\label{XAYAat1prop}
Suppose $u_0, v_0 \in \mathcal{S}(\R)$ and let $p \geq 0$ be an integer. 
Then there are $3 \times 3$-matrix valued functions $D_i^{(l)}(x)$, $i = 1,2$, $l = -2,-1, \dots, p$, with the following properties:
\begin{itemize}
\item For $x \in \R$ and $k \in (-\omega^2 \bar{\mathrm{S}}, -\omega \bar{\mathrm{S}}, -\bar{\mathrm{S}})$, the function $X^A$ satisfies
\begin{subequations}\label{XAYAat1}
\begin{align}\label{XAYAat1a}
& \bigg|\frac{\partial^j}{\partial k^j}\big(X^A - I - \sum_{l=-2}^p D_1^{(l)}(x)k^l\big) \bigg| \leq
f_+(x)|k|^{p+1-j}, \qquad |k| \leq \frac{1}{2},
\end{align}
while, for $x \in \R$ and $k \in (\omega^2 \bar{\mathrm{S}}, \omega \bar{\mathrm{S}}, \bar{\mathrm{S}})$, the function $Y$ satisfies
\begin{align}
& \bigg|\frac{\partial^j}{\partial k^j}\big(Y^A - I - \sum_{l=-2}^p D_2^{(l)}(x)k^l\big) \bigg| \leq
f_-(x)|k|^{p+1-j}, \qquad |k| \leq \frac{1}{2},
\end{align}
\end{subequations}
where $f_+(x)$ and $f_-(x)$ are smooth positive functions of $x \in \R$ with rapid decay as $x \to +\infty$ and $x \to -\infty$, respectively, and $j \geq 0$ is any integer.

\item For each $l\geq -1$, $D_1^{(l)}(x)$ and $D_2^{(l)}(x)$ are smooth functions of $x \in \R$ which have rapid decay as $x \to +\infty$ and $x \to -\infty$, respectively.

\item The leading coefficients have the form
\begin{align}
D_{i}^{(-2)}(x)
= & \;
\tilde{\alpha}_{i}(x)\begin{pmatrix}
\omega & \omega & \omega \\
\omega^{2} & \omega^{2} & \omega^{2} \\
1 & 1 & 1 
\end{pmatrix}, \label{Djminus2} \\
D_{i}^{(-1)}(x) = & \;  \tilde{\beta}_{i}(x) \begin{pmatrix}
\omega^{2} & 1 & \omega \\
1 & \omega & \omega^{2} \\
\omega & \omega^{2} & 1 
\end{pmatrix} + \tilde{\gamma}_{i}(x) \begin{pmatrix}
\omega^{2} & \omega^{2} & \omega^{2} \\
\omega & \omega & \omega \\
1 & 1 & 1 
\end{pmatrix},  \label{Djminus1} \\
D_{i}^{(0)}(x) = & -I + \tilde{\delta}_{i,1}(x) \begin{pmatrix}
1 & \omega^{2} & \omega \\
\omega & 1 & \omega^{2} \\
\omega^{2} & \omega & 1
\end{pmatrix}  \nonumber \\
&  + \tilde{\delta}_{i,2}(x) \begin{pmatrix}
1 & \omega & \omega^{2} \\
\omega^{2} & 1 & \omega \\
\omega & \omega^{2} & 1
\end{pmatrix} + \tilde{\delta}_{i,3}(x)\begin{pmatrix}
1 & 1 & 1 \\
1 & 1 & 1 \\
1 & 1 & 1
\end{pmatrix}, \label{Djminus0}
\end{align}
where $\tilde{\alpha}_i(x)$, $\tilde{\beta}_i(x)$, $\tilde{\gamma}_i(x)$, $\tilde{\delta}_{i,j}(x)$, $i = 1,2$, $j=1,2,3$ are real-valued functions of $x \in \R$. Furthermore, there exist bounded functions $f_{1}(x)$ and $f_{2}(x)$, with rapid decay at $+\infty$ and $-\infty$ respectively, such that
\begin{align*}
& |\tilde{\alpha}_{i}(x)| \leq f_{i}(x), 
	\\
& |\tilde{\beta}_{i}(x)| + |\tilde{\gamma}_{i}(x)| \leq (1+|x|)f_{i}(x), 
	\\
& |\tilde{\delta}_{i,j}(x)-\tfrac{1}{3}| \leq (1+|x|)^{2}f_{i}(x),
\end{align*}
for all $x \in \mathbb{R}$, $i=1,2$, and $j=1,2,3$.
\end{itemize}
\end{proposition}
\begin{proof}
The function $\mathcal{X}^A := P^AX^A$ satisfies the integral equation
\begin{align}\nonumber
& \mathcal{X}^A (x,k) = P^A(k) + \int_x^\infty P^A(k)e^{-(x - x')\widehat{\mathcal{L}(k)}}(P(k)^T \tilde{\mathsf{U}}(x')^T\mathcal{X}^A (x',k)) dx', 
	\\ \label{tildeXAeq}
& \hspace{6cm} x \in \R, \ k \in (-\omega^2 \bar{\mathrm{S}}, -\omega \bar{\mathrm{S}}, -\bar{\mathrm{S}}) \setminus \{0\},
\end{align}
where $\tilde{\mathsf{U}} = P\mathsf{U}P^{-1}$ is independent of $k$.
A computation shows that the kernel 
\begin{align*}
\tilde{\mathcal{P}}(x,x',k) := P^A(k)e^{-(x - x')\mathcal{L}(k)}P(k)^T
\end{align*}
is analytic at $k=0$, and the statement follows in a similar way as in the proof of Proposition \ref{XYat1prop}.
\end{proof}

\subsection{The spectral function $s^A(k)$}
The following proposition is proved in the same way as Proposition \ref{sprop}.

\begin{proposition}[Properties of $s^A(k)$]\label{sAprop}
Suppose $u_0,v_0 \in \mathcal{S}(\R)$. 
Then the spectral function $s^A(k)$ defined in (\ref{sAdef}) has the following properties:
\begin{enumerate}[$(a)$]
\item $s^A(k)$ is defined and continuous for $k$ in
\begin{align}\label{sAdomainofdefinition}
 \begin{pmatrix}
 -\omega^2 \bar{\mathrm{S}} & \R_- & \omega \R_- \\
 \R_- & -\omega \bar{\mathrm{S}} & \omega^2 \R_- \\
 \omega \R_- & \omega^2 \R_- & -\bar{\mathrm{S}}
 \end{pmatrix}\setminus \{0\},
\end{align}
that is, the $(11)$ entry of $s^A(k)$ is defined and continuous for $k \in -\omega^2 \bar{\mathrm{S}}\setminus \{0\}$, etc. 

\item The diagonal entries of $s^A(k)$  are analytic in the interior of their domains of definition as given in (\ref{sAdomainofdefinition}). 
 
\item For $j = 1, 2, \dots$, the derivative $\partial_k^js^A(k)$ is well-defined and continuous for $k$ in (\ref{sAdomainofdefinition}).

\item $s^A(k)$ obeys the symmetries
\begin{align}\nonumber
&  s^A(k) = \mathcal{A} s^A(\omega k)\mathcal{A}^{-1} = \mathcal{B} \overline{s^A(\overline{k})}\mathcal{B}.
\end{align}

\item $s^{A}(k)$ approaches the identity matrix as $k \to \infty$. More precisely, there are diagonal matrices $\{s^{A}_j\}_1^\infty$ such that
\begin{align*}\nonumber
\Big|\partial_k^j \Big(s^{A}(k) - I - \sum_{j=1}^N \frac{s^{A}_j}{k^j}\Big)\Big| = O(k^{-N-1}), \qquad k \to \infty,
\ k \in \begin{pmatrix}
 -\omega^2 \bar{\mathrm{S}} & \R_- & \omega \R_- \\
 \R_- & -\omega \bar{\mathrm{S}} & \omega^2 \R_- \\
 \omega \R_- & \omega^2 \R_- & -\bar{\mathrm{S}}
 \end{pmatrix},
\end{align*}
for $j = 0, 1, \dots, N$ and each integer $N \geq 1$. In particular, the off-diagonal entries of $s^{A}(k)$ have rapid decay as $k \to \infty$.

\item As $k \to 0$,
\begin{align}\label{sA at 0}
s^{A}(k) = \frac{s^{A(-2)}}{k^{2}} + \frac{s^{A(-1)}}{k} + s^{A(0)} + \ldots, \qquad  k \in \begin{pmatrix}
 -\omega^2 \bar{\mathrm{S}} & \R_- & \omega \R_- \\
 \R_- & -\omega \bar{\mathrm{S}} & \omega^2 \R_- \\
 \omega \R_- & \omega^2 \R_- & -\bar{\mathrm{S}}
 \end{pmatrix},
\end{align}
where
\begin{align}\label{sApm2p}
s^{A(-2)} =  \mathfrak{s}^{A(-2)} \begin{pmatrix}
\omega & \omega & \omega \\
\omega^{2} & \omega^{2} & \omega^{2} \\
1 & 1 & 1 
\end{pmatrix}, \quad \mathfrak{s}^{A(-2)} := -\int_{\mathbb{R}}  (u_{0x}+v_{0})\tilde{\delta}_{1,3} dx \in \mathbb{R},
\end{align}
and the expansion can be differentiated termwise any number of times.
\item If $u_0(x), v_0(x)$ have compact support, then $s^A(k)$ is defined and equals the cofactor matrix of $s(k)$ for all $k \in \C \setminus \{0\}$.
\end{enumerate}

\end{proposition}

\subsection{Proof of Theorem \ref{r1r2th}}\label{r1r2subsec}
The theorem follows from Propositions \ref{sprop} and \ref{sAprop}. Indeed, recall from (\ref{r1r2def}) that $r_1 = s_{12}/s_{11}$ and $r_2 = s_{12}^A/s_{11}^A$. Statements $(a)$ and $(c)$ of Propositions \ref{sprop} and \ref{sAprop} imply that 
$s_{12}(k)$ and $s_{11}(k)$ are smooth on $(0,\infty)$, while $s_{12}^A(k)$ and $s_{11}^A(k)$ are smooth on $(-\infty,0)$. Since $s_{11}$ and $s_{11}^A$ have no zeros by Assumption \ref{solitonlessassumption}, it follows that $r_1 \in C^\infty((0,\infty))$ and $r_2 \in C^\infty((-\infty,0))$. Moreover, statement $(e)$ of the same propositions imply that $r_1(k)$ and $r_2(k)$ satisfy (\ref{r1r2rapiddecay}).
 
Assumption \ref{originassumption} implies that the coefficients $\mathfrak{s}^{(-2)}$ and $\mathfrak{s}^{A(-2)}$ in (\ref{spm2p}) and (\ref{sApm2p}) are both nonzero, ensuring that all four functions $s_{12}$, $s_{11}$, $s_{12}^A$, $s_{11}^A$ are of order $k^{-2}$ as $k \to 0$. 
Hence properties $(ii)$ and $(iii)$ of Theorem \ref{r1r2th} related to the behavior of $r_1$ and $r_2$ as $k \to 0$ follow from statement $(f)$ of Propositions \ref{sprop} and \ref{sAprop}.

It remains to prove that $|r_{1}(k)|<1$ for $k > 0$ and that $|r_{2}(k)| <1$ for $k < 0$.
For $k< 0$, all four entries $\{s_{ij}^{A}(k)\}_{i,j=1}^{2}$ are well-defined, and hence $s_{11}^{A}(k)s_{22}^{A}(k)-s_{12}^{A}(k)s_{21}^{A}(k) = s_{33}(k)$. Using also the symmetries $s^{A}(k) = \mathcal{B}\overline{s^{A}(\overline{k})}\mathcal{B}$ and $s(k) = \mathcal{A} s(\omega k)\mathcal{A}^{-1} $, we conclude that
\begin{align}\label{lol2}
1-|r_{2}(k)|^{2} = 1 - \bigg| \frac{s_{12}^{A}(k)}{s_{11}^{A}(k)} \bigg|^{2} = \frac{s_{11}^{A}(k)s_{22}^{A}(k)-s_{12}^{A}(k)s_{21}^{A}(k)}{|s_{11}^{A}(k)|^{2}} = \frac{s_{33}(k)}{|s_{11}^{A}(k)|^{2}} = \frac{s_{11}(\omega^{2}k)}{|s_{11}^{A}(k)|^{2}}
\end{align}
for $k < 0$. Since the left-hand side of \eqref{lol2} is real and tends to $1$ as $k \to -\infty$, and since the right-hand side is non-zero for all $k < 0$ by Assumption \ref{solitonlessassumption}, we deduce that $1-|r_{2}(k)|^{2}>0$ for all $k < 0$. An analogous argument shows that $|r_{1}(k)| <1$ for $k > 0$.

\section{The function $M$}\label{Msec}
In this section, we construct the sectionally analytic function $M$ which features in the RH problem \ref{RH problem for M}. The restriction of $M$ to the sector $D_n$, $n = 1,\dots, 6$, will be denoted by $M_n$.

\subsection{The eigenfunctions $M_n$}
For each $n = 1, \dots, 6$, we define a $3\times 3$-matrix valued solution $M_n(x,k)$ of (\ref{xpart}) for $k \in D_n\setminus \{0\}$ by the following system of Fredholm integral equations: 
\begin{align}\label{Mndef}
(M_n)_{ij}(x,k) = \delta_{ij} + \int_{\gamma_{ij}^n} \left(e^{(x-x')\widehat{\mathcal{L}(k)}} (\mathsf{U}M_n)(x',k) \right)_{ij} dx', \qquad  i,j = 1, 2,3,
\end{align}
where the contours $\gamma^n_{ij}$, $n = 1, \dots, 6$, $i,j = 1,2,3$, are defined by
 \begin{align} \label{gammaijndef}
 \gamma_{ij}^n =  \begin{cases}
 (-\infty,x),  & \re  l_i(k) < \re  l_j(k), 
	\\
(+\infty,x),  \quad & \re  l_i(k) \geq \re  l_j(k),
\end{cases} \quad \text{for} \quad k \in D_n.
\end{align}
The contours $\gamma_{ij}^n$ are defined in such a way that the exponential $e^{(l_i - l_j)(x-x')}$ appearing in the equation for $(M_n)_{ij}$ in (\ref{Mndef}) is bounded for $k \in D_n$ and $x' \in \gamma_{ij}^n$. The definition (\ref{Mndef}) of $M_n$ can be extended by continuity to the boundary of $D_n$. 
As the next proposition shows, this makes all entries of $M_n$ well-defined for $k \in \bar{D}_n\setminus \mathcal{Q}$, where 
\begin{align}\label{calQdef}
  \mathcal{Q} = \{0\} \cup \mathsf{Z}
\end{align}
and $\mathsf{Z}$ denotes the set of zeros of the Fredholm determinants associated with (\ref{Mndef}) (the intersection of $\mathsf{Z}$ with $\bar{D}_1$ is given by $\cup_{j=1}^3 \{k \in \bar{D}_1 | f_j(k) = 0\}$, where $f_j(k)$ are the Fredholm determinants given explicitly in (\ref{fjdef})).

\begin{proposition}[Basic properties of $M_n$]\label{Mnprop}
Suppose $u_0, v_0 \in \mathcal{S}(\R)$. 
Then the equations (\ref{Mndef}) uniquely define six $3 \times 3$-matrix valued solutions $\{M_n\}_1^6$ of (\ref{xpart}) with the following properties:
\begin{enumerate}[$(a)$]
\item The function $M_n(x, k)$ is defined for $x \in \R$ and $k \in \bar{D}_n \setminus \mathcal{Q}$. For each $k \in \bar{D}_n  \setminus \mathcal{Q}$, $M_n(\cdot, k)$ is smooth and satisfies (\ref{xpart}).

\item For each $x \in \R$, the function $M_n(x,\cdot)$ is continuous for $k \in \bar{D}_n \setminus \mathcal{Q}$ and analytic for $k \in D_n\setminus \mathcal{Q}$.

\item For each $\epsilon > 0$, there exists a $C = C(\epsilon)$ such that
\begin{align}\label{Mnbounded}
|M_n(x,k)| \leq C, \qquad x \in \R, \ k \in \bar{D}_n, \ \dist(k, \mathcal{Q}) \geq \epsilon.
\end{align}

\item For each $x \in \R$ and each $j = 1, 2, \dots$, the partial derivative $\frac{\partial^j M_n}{\partial k^j}(x, \cdot)$ has a continuous extension to $\bar{D}_n \setminus \mathcal{Q}$.

\item $\det M_n(x,k) = 1$ for $x \in \R$ and $k \in \bar{D}_n \setminus \mathcal{Q}$.

\item For each $x \in \R$, the sectionally analytic function $M(x,k)$ defined by $M(x,k) = M_n(x,k)$ for $k \in D_n$ satisfies the symmetries
\begin{align}\label{Msymm}
 M(x, k) = \mathcal{A} M(x,\omega k)\mathcal{A}^{-1} = \mathcal{B} \overline{M(x,\overline{k})}\mathcal{B}, \qquad k \in \C \setminus \mathcal{Q}.
\end{align}

\end{enumerate}
\end{proposition}
\begin{proof}
Let us first consider the third column of $M_1$.
Define $H(x)$ by $H(x) = 1$ for $x > 0$ and $H(x) = 0$ for $x \leq 0$. 
Letting $w_i(x,k) = (M_1)_{i3}(x,k)$, we can write the third column of (\ref{Mndef}) as
\begin{align} \nonumber
&  w_i(x,k) = \delta_{i3} + \int_{-\infty}^{+\infty}  \sum_{l = 1}^3 K(x,x',k)_{il}w_l(x',k) dx', 
  	\\\label{wieq} 
&\hspace{6cm} x \in \R, \ k \in \bar{D}_1\setminus\{0\}, \ i = 1,2,3.
\end{align}  
where the kernel $K$ is defined for $x,x' \in \R$, $k \in \bar{D}_1\setminus\{0\}$, $i,l = 1,2,3$ by
\begin{align}\label{Kjkernel}
K(x,x', k)_{il} = \begin{cases} 
H(x-x')e^{(x - x')(l_i - l_3)} \mathsf{U}(x',k)_{il} & \text{if} \ \gamma_{i3}^1 = (-\infty,x),
	\\
-H(x' - x)e^{(x - x')(l_i - l_3)} \mathsf{U}(x',k)_{il} & \text{if} \ \gamma_{i3}^1 = (\infty,x).
\end{cases}	
\end{align}

Equation (\ref{wieq}) is a Fredholm equation of the second kind. However, the standard Fredholm theory does not immediately apply, because the integral kernel $K$ in (\ref{wieq}) is, in general, not an $L^2$-kernel. 
Indeed, for $k$ such that $ \re (l_i- l_3) =  0$, the exponential factor in (\ref{Kjkernel}) is bounded, but does not decay as $x,x' \to \pm \infty$.
This means that the kernel in (\ref{wieq}) decays as $|x'| \to \infty$, but not necessarily as $|x| \to \infty$. 
Even though the kernel $K_j$ is not of $L^2$-type, equation (\ref{wieq}) can be analyzed by an extension of the standard Fredholm theory, see \cite{C1982}. The remainder of the proof is a minor generalization of the arguments of Appendix A of \cite{C1982}, allowing for a more general $k$-dependence.

Fix $\epsilon >0$ small.
Let $\bar{D}_1^\epsilon$ denote the set $\bar{D}_1$ with open disks of radius $\epsilon$ centered at the origin removed, i.e.,
$$\bar{D}_1^\epsilon = \bar{D}_1 \setminus \{|k|< \epsilon \}.$$
The exponential factor in (\ref{Kjkernel}) is bounded for $k \in \bar{D}_1^\epsilon$. Also, there is a function $b_1 \in \mathcal{S}(\R)$ such that
$$|\mathsf{U}(x,k)| \leq b_1(x), \qquad x \in \R, \ k \in \bar{D}_1^\epsilon.$$
We infer that there exists a function $b \in \mathcal{S}(\R)$ such that
\begin{align}\label{Kmbound}
|K(x,x',k)_{il}| \leq b(x'), \qquad x, x' \in \R, \ k \in \bar{D}_1^\epsilon, \ i,l = 1, 2,3.
\end{align}
Set $K^{(0)} := 1$ and define the complex-valued function $K^{(m)}$ for $m \geq 1$ by
\begin{align}\label{Kmdef}
K^{(m)}\left(\substack{x_1, i_1, x_2, i_2, \dots, x_m, i_m \\ x_1', i_1', x_2', i_2', \cdots, x_m', i_m'}; k\right)
= \det\begin{pmatrix} K(x_1,x_1', k)_{i_1i_1'} & \cdots & K(x_1,x_m', k)_{i_1i_m'}  \\
\vdots &  & \vdots \\
 K(x_m,x_1', k)_{i_mi_1'} & \cdots & K(x_m,x_m', k)_{i_mi_m'}  \\
\end{pmatrix}.
\end{align}
Hadamard's inequality for an $m\times m$ matrix $A$,
$$|\det A|^2 \leq \prod_{i=1}^m \sum_{j=1}^m |A_{ij}|^2,$$
together with the bound (\ref{Kmbound}) gives
\begin{align}\label{Kmestimate}
\left|K^{(m)}\left(\substack{x_1, i_1, x_2, i_2, \dots, x_m, i_m \\ x_1', i_1', x_2', i_2', \dots, x_m', i_m'}; k\right)\right|
\leq m^{m/2}\prod_{j=1}^m b(x_j'), \qquad k \in \bar{D}_1^\epsilon, \ m = 1,2,\dots.
\end{align}
The Fredholm determinant $f(k)$ and the Fredholm minor $F(x,x',k)$ associated with equation (\ref{wieq}) are defined by
\begin{align}\label{Fredholmfdef}
& f(k) = \sum_{m=0}^\infty f^{(m)}(k), \qquad  k \in \bar{D}_1^\epsilon,
	\\\label{FredholmFdef}
& F(x,x',k)_{ii'} = \sum_{m=0}^\infty F^{(m)}(x,x',k)_{ii'}, \qquad x, x' \in \R, \ k \in \bar{D}_1^\epsilon, \ i,i' = 1, 2, 3, 
\end{align}
where $f^{(m)}$ and $F^{(m)}$ are defined for $m \geq 0$ by\footnote{For $m = 0$ these definitions should be interpreted as 
$$f^{(0)}(k) = 1, \qquad F^{(0)}(x,x',k)_{ii'} = K^{(1)}\left(\substack{x, i \\ x', i'}; k\right) 
= K(x,x', k)_{ii'}.$$}
\begin{align}\nonumber
 f^{(m)}(k) = &\; \frac{(-1)^m}{m!}\sum_{i_1, i_2, \dots, i_m=1}^3 \int_{\R^m} K^{(m)}
\left(\substack{x_1, i_1, x_2, i_2, \dots, x_m, i_m \\
x_1, i_1, x_2, i_2, \dots, x_m, i_m}; k\right) dx_1dx_2 \cdots dx_m,
	\\\nonumber
 F^{(m)}(x,x',k)_{ii'} = &\; \frac{(-1)^m}{m!}\sum_{i_1, i_2, \dots, i_m=1}^3 
	\\\label{FredholmFmdef}
& \times \int_{\R^m} 
K^{(m+1)}\left(\substack{x, i, x_1, i_1, x_2, i_2, \dots, x_m, i_m \\
x', i', x_1, i_1, x_2, i_2, \dots, x_m, i_m}; k\right) dx_1dx_2 \cdots dx_m.
\end{align}
In view of (\ref{Kmestimate}), we have
\begin{align}\label{fmkestimate}
|f^{(m)}(k)| \leq \frac{3^mm^{m/2}\|b\|_{L^1(\R)}^m}{m!}, \qquad k \in \bar{D}_1^\epsilon, \ m \geq 0.
\end{align}
Using Stirling's approximation $m! \sim \sqrt{2\pi m}(m/e)^m$, we see that the series in (\ref{Fredholmfdef}) converges absolutely and uniformly for $k \in \bar{D}_1^\epsilon$. For each $m$, $f^{(m)}(k)$ is a continuous function of $k \in \bar{D}_1^\epsilon$ which is analytic in the interior of $\bar{D}_1^\epsilon$. This shows that the Fredholm determinant $f(k)$ is a bounded continuous function of $k \in \bar{D}_1^\epsilon$ which is analytic in the interior of $\bar{D}_1^\epsilon$. Similarly, the estimate
\begin{align}\label{Fmxxkestimate}
|F^{(m)}(x,x',k)_{ii'}| \leq \frac{3^m(m+1)^{(m+1)/2}\|b\|_{L^1(\R)}^mb(x')}{m!}, \qquad x,x' \in \R, \ k \in \bar{D}_1^\epsilon, \ m \geq 0,
\end{align}
shows that the Fredholm minor $F(x,x',k)$ has the following properties: $(i)$ For each $(x,x') \in \R^2$, $F(x,x',k)$ is a bounded continuous function of $k \in \bar{D}_1^\epsilon$ which is analytic in the interior of $\bar{D}_1^\epsilon$; $(ii)$ For each $k \in \bar{D}_1^\epsilon$, $F(x,x',k)$ is a smooth function of  $(x, x') \in \R^2 \setminus \{x = x'\}$ (since $H(x - x')$ has a discontinuity at $x = x'$); $(iii)$ $F$ obeys the estimate
\begin{align}\label{FredholmFbound}
|F(x,x',k)| \leq Cb(x'), \qquad x,x' \in \R, \ k \in \bar{D}_1^\epsilon.
\end{align}

Expanding the determinant in (\ref{Kmdef}) along the first column, we find, for $m \geq 0$,
\begin{align*}
& K^{(m+1)}\left(\substack{x, i, x_1, i_1, \dots, x_m, i_m \\
x', i', x_1, i_1, \dots, x_m, i_m}; k\right)
=  K(x,x',k)_{ii'} K^{(m)}
\left(\substack{x_1, i_1, \dots, x_m, i_m \\
x_1, i_1, \dots, x_m, i_m}; k\right)
	\\
&\hspace{2cm} - \sum_{s=1}^m K(x_s, x',k)_{i_si'}K^{(m)}
\left(\substack{x, i, x_1, i_1, \dots, x_{s-1}, i_{s-1},x_{s+1}, i_{s+1}, \dots, x_m, i_m \\
x_s, i_s, x_1, i_1, \dots, x_{s-1}, i_{s-1},x_{s+1}, i_{s+1}, \dots x_m, i_m}; k\right).
\end{align*}
Substituting this identity into (\ref{FredholmFmdef}) and simplifying, we obtain
\begin{align*}
F^{(m)}(x,x',k)_{ii'} = f^{(m)}(k)K(x,x',k)_{ii'} + \sum_{i_{s}=1}^{3} \int_{\mathbb{R}} F^{(m-1)}(x,x'',k)_{ii_{s}}K(x'',x',k)_{i_{s}i'}dx'',
\end{align*}
or, in matrix-form,
$$F^{(m)}(x,x',k) = f^{(m)}(k)K(x,x',k) + \int_\R F^{(m-1)}(x, x'',k) K(x'', x',k) dx'', \qquad m \geq 0,$$
where $F^{(-1)} := 0$. Summing this equation from $m=0$ to $m = +\infty$, we find, for $x, x' \in \R$ and $k \in \bar{D}_1^\epsilon$,
\begin{subequations}\label{FfK}
\begin{align}
F(x,x',k) = f(k)K(x,x',k) + \int_\R F(x, x'', k) K(x'', x',k) dx''.
\end{align}
If we instead expand the determinant in (\ref{Kmdef}) along the first row, the same type of argument leads to
\begin{align}
F(x,x',k) = f(k) K(x,x',k) + \int_{\R} K(x,x'',k) F(x'', x', k) dx''.
\end{align}
\end{subequations}
The identities in (\ref{FfK}) show that 
$$R(x,x',k) = \frac{F(x,x',k)}{f(k)}$$
satisfies the resolvent equations
\begin{align*}
R(x,x',k) - K(x,x',k) & = \int_{\R} R(x,x'',k) K(x'',x',k) dx''
	\\
& = \int_{\R} K(x,x'',k)R(x'', x',k) dx''
\end{align*}
for $x, x' \in \R$ and $k \in \bar{D}_1^\epsilon\setminus \mathcal{Q}$. 
This shows that (\ref{wieq}) has a unique solution given by
\begin{align}\label{wifinal}
w_i(x,k) = \delta_{i3} + \frac{1}{f(k)} \int_\R F(x,x',k)_{i3} dx', \qquad x \in \R, \ k \in \bar{D}_1^\epsilon \setminus \mathcal{Q}.
\end{align}
Using the properties of $f(k)$ and $F(x,x',k)$  and the fact that $\epsilon > 0$ was arbitrary, it follows from this representation that the third column of $M_1$ satisfies $(a)$-$(d)$. The proofs of $(a)$-$(d)$ for the first and second columns of $M_1$ are similar.

Letting $x \to \infty$ in (\ref{Mndef}) and using (\ref{Mnbounded}), we find that $\lim_{x \to \infty} M_1(x,k)_{ij} = \delta_{ij}$ for $(i,j)$ such that $\gamma_{ij}^1 = (\infty,x)$, and that $M_1(x,k)_{ij}$ remains bounded as $x \to \infty$ for $(i,j)$ such that $\gamma_{ij}^1 = (-\infty,x)$. In other words, the entries of $M_1$ above the diagonal remain bounded as $x \to \infty$, whereas the part of $M_1$ on and below the diagonal approaches the identity matrix:
\begin{align}\label{asymp for M1 at +infty}
M_{1}(x,k) \to \begin{pmatrix}
1 & \star & \star \\
0 & 1 & \star \\
0 & 0 & 1
\end{pmatrix} \qquad \mbox{as } x \to +\infty.
\end{align}
In particular, for each $k \in \bar{D}_{1}\setminus \mathcal{Q}$, $M_{1}(x,k)$ is invertible for all sufficiently large $x$. From standard theory of ODEs, we conclude that $M_{1}(x,k)$ is invertible for all $x \in \mathbb{R}$. Since $M_1(\cdot, k)$ is a smooth solution of (\ref{xpart}), we infer that
\begin{align*}
(\log \det M_{1}(x,k))_{x} = \tr (M_{1}^{-1}M_{x}) = \tr \mathsf{U} = 0,
\end{align*}
from which we conclude that $\det M_1(x,k)$ is independent of $x$. Using \eqref{asymp for M1 at +infty} again, we find $\det M_1(x,k) = 1$ for all $x \in \mathbb{R}$ and $k \in \bar{D}_{1}\setminus \mathcal{Q}$, which proves $(e)$ for $n =1$.

The Fredholm equations (\ref{Mndef}) are consistent with the symmetries in (\ref{Msymm}), because $\mathcal{L}(k)$ and $\mathsf{U}(x,k)$ obey these symmetries. Hence we can construct the unique solutions $M_n$ of (\ref{Mndef}) for $n = 2, \dots, 6$ in the same way as we constructed $M_1$. By uniqueness, these $M_n$ will satisfy the symmetries in (\ref{Msymm}).
\end{proof}

\begin{remark}\label{fredholmremark}
In the proof of Proposition \ref{Mnprop}, we focused on the third column of $M_1$ to avoid an abundance of indices. More generally, the Fredholm determinant associated with the $j$th column of $M_1$ is defined by
\begin{align}\nonumber
f_j(k) = & \; \sum_{m=0}^\infty \frac{(-1)^m}{m!}\sum_{i_1, i_2, \cdots, i_m=1}^3 
	\\ \label{fjdef}
& \times \int_{\R^m} \det\begin{pmatrix} K_j(x_1,x_1, k)_{i_1i_1} & \cdots & K_j(x_1,x_m, k)_{i_1i_m} \\
\vdots &  & \vdots \\
 K_j(x_m,x_1, k)_{i_mi_1} & \cdots & K_j(x_m,x_m, k)_{i_mi_m} \\
\end{pmatrix}dx_1dx_2 \cdots dx_m,
\end{align}
where
\begin{align}\label{Kjkerneldef}
K_j(x,x',k)_{il} = \begin{cases} 
H(x-x')e^{(l_i - l_j)(x - x')} \mathsf{U}(x',k)_{il} & \text{if} \ \gamma_{ij}^1 = (-\infty,x),
	\\
-H(x' - x)e^{(l_i - l_j)(x - x')} \mathsf{U}(x',k)_{il} & \text{if} \ \gamma_{ij}^1 = (\infty,x).
\end{cases}	
\end{align}
Each function $f_j(k)$ is an analytic function of $k \in D_1$ with a continuous extension to $\bar{D}_1 \setminus \{0\}$. Since it does not vanish identically (in fact, $f_j \to 1$ as $k \to \infty$, because $\mathsf{U}(x,k)$ is $O(1/k)$ as $k \to \infty$), it has at most countably many zeros in $D_1$. 
\end{remark}

As in the case of $X$ and $Y$, the asymptotics of $M_n$ as $k \to \infty$ can be obtained by considering formal power series solutions of (\ref{xpart}). The formal solutions take the form
\begin{align*}
& M_{n,formal}(x,k) = I + \frac{M_{n,1}(x)}{k} + \frac{M_{n,2}(x)}{k^2} + \cdots, \qquad n = 1, \dots, 6,
\end{align*}
with the normalization conditions
\begin{align}\label{Mnlnormalization}
\begin{cases}
\displaystyle{\lim_{x\to -\infty}} (M_{n,l}(x))_{ij} = 0 & \text{if} \ \gamma_{ij}^n = (-\infty,x), 
	\\
\displaystyle{\lim_{x\to \infty}} (M_{n,l}(x))_{ij} = 0 & \text{if} \ \gamma_{ij}^n = (\infty,x),
\end{cases} \quad l \geq 1.
\end{align}
The coefficients $M_{n,j}$ are uniquely determined from (\ref{Mnlnormalization}), the recursive relations (\ref{xrecursive}), and the initial assignments $M_{n,-2} = 0$, $M_{n,-1} = 0$, $M_{n,0} = I$.
In fact, since $\gamma_{ii}^n = (\infty,x)$ for $i = 1,2,3$ and $n = 1, \dots,6$, it follows that $M_{n,j}(x) = X_j(x)$ for all $n$ and $j$.

\begin{lemma}[Asymptotics of $M$ as $k \to \infty$]\label{Matinftylemma}
Suppose $u_0, v_0 \in \mathcal{S}(\R)$ and $u_0,v_0 \not\equiv 0$. Given an integer $p \geq 1$, let $X_{(p)}(x,k)$ be the function defined in (\ref{Xpdef}). Then there exists an $R > 0$ such that
\begin{align*}
& \big|M(x,k) - X_{(p)}(x,k) \big| \leq
\frac{C}{|k|^{p+1}}, \qquad x \in \R, \  k \in \C \setminus \Gamma, \ |k| \geq R.
\end{align*}
\end{lemma}
\begin{proof}
Let $M_{(p)} := X_{(p+1)}$. It is enough to show that there exists an $R > 0$ such that
\begin{align}\label{Matinfty}
& \big|M(x,k) - M_{(p)}(x,k) \big| \leq
\frac{C}{|k|^{p+1}}, \qquad x \in \R, \  k \in \C \setminus \Gamma, \ |k| \geq R.
\end{align}
By induction on \eqref{xrecursive}, one shows easily that
\begin{align}\label{recursive prop on Xj}
X_{j}^{(o)} \in \mathcal{S}(\mathbb{R}), \qquad \partial_{x}X_{j}^{(d)} \in \mathcal{S}(\mathbb{R}), \qquad j \geq 1.
\end{align}
For $n = 1, \dots, 6$ and $R> 1$, let $\bar{D}_n^R = \bar{D}_n \cap \{|k| \geq R\}$.
We choose $R > 1$ sufficiently large such that
\begin{align*}
\sup_{|k| \geq R} \sup_{x \in \R} \bigg|\sum_{j=1}^{p+1} \frac{X_j(x)}{k^j}\bigg| < 1.
\end{align*}
Then the inverse $M_{(p)}^{-1}$ exists and $M_{(p)}, M_{(p)}^{-1}$ are uniformly bounded for $x \in \R$ and $|k| \geq R$. 
Define $L(x,k)$ and $L_{(p)}(x,k)$ by 
$$L = \mathcal{L} + \mathsf{U}, \qquad L_{(p)} = \big(\partial_x M_{(p)} + M_{(p)} \mathcal{L}\big)M_{(p)}^{-1}$$
and let $\Delta = L - L_{(p)}$ denote their difference.
The quotient $M_{(p)}^{-1} M$ satisfies the equation
\begin{align*}
  (M_{(p)}^{-1}M)_x & = - M_{(p)}^{-1}(\partial_xM_{(p)}) M_{(p)}^{-1} M + M_{(p)}^{-1} \partial_xM
  	\\
&  = - M_{(p)}^{-1}(L_{(p)}M_{(p)} - M_{(p)}\mathcal{L}) M_{(p)}^{-1} M + M_{(p)}^{-1} (LM -M\mathcal{L})
	\\
&  = M_{(p)}^{-1}\Delta M + [\mathcal{L}, M_{(p)}^{-1}M],
\end{align*}
that is,
$$\Big(e^{-x\hat{\mathcal{L}}}(M_{(p)}^{-1}M)\Big)_x = e^{-x\hat{\mathcal{L}}}(M_{(p)}^{-1} \Delta M).$$
The entries $(M_{(p)})_{ij}$ and $(M_n)_{ij}$ (and hence also the entries $(M_{(p)}^{-1}M)_{ij}$) approach  $\delta_{ij}$ as $x \to \pm \infty$ for $\gamma_{ij}^n = (\pm \infty,x)$. We conclude that $M$ satisfies the Fredholm equation
\begin{align*}
(M_{(p)}^{-1}M_n)(x,k)_{ij} = \delta_{ij} + \int_{\gamma_{ij}^n} \Big(e^{(x-x')\hat{\mathcal{L}}} (M_{(p)}^{-1}\Delta M_n)(x',k) \Big)_{ij} dx',
\end{align*}
that is,
\begin{align}\label{Mndefinfty}
M_n(x,k)_{ij} = M_{(p)}(x,k)_{ij} + \sum_{s=1}^3 M_{(p)}(x,k)_{is}\int_{\gamma_{sj}^n} \Big(e^{(x-x')\hat{\mathcal{L}}} (M_{(p)}^{-1}\Delta M_n)(x',k) \Big)_{sj} dx'
\end{align}
for $x \in \R$ and $k \in \bar{D}_n^R$. 
For definiteness, we focus on the third column of $M_1$. Letting $w_i(x,k) = (M_1)_{i3}(x,k)$, we can write the third column of (\ref{Mndefinfty}) for $n = 1$ as
\begin{align}\label{wieqinfty}  
 & w_i(x,k) = M_{(p)}(x,k)_{i3} + \int_{-\infty}^\infty  \sum_{l = 1}^3 K(x,x',k)_{il}w_l(x',k) dx',  \qquad i = 1,2,3,
\end{align}  
where the kernel $K$ is defined for $x,x' \in \R$, $k \in \bar{D}_1^R$, and $i,l = 1,2,3$ by
\begin{align}\label{def new kernel}
K(x,x', k)_{il} = \sum_{s =1}^3 M_{(p)}(x,k)_{is} H_s(x,x') e^{(l_s - l_3)(x - x')} (M_{(p)}^{-1}\Delta)(x',k)_{sl}
\end{align}
with
$$H_s(x,x') =
\begin{cases} 
H(x-x') & \text{if} \ \gamma_{s3}^1 = (-\infty,x),
	\\
-H(x' - x) & \text{if} \ \gamma_{s3}^1 = (\infty,x).
\end{cases}$$
Let us now rewrite $\Delta$ as follows:
\begin{align*}
\Delta = (\mathsf{U}M_{(p)} + [\mathcal{L},M_{(p)}] - \partial_{x}M_{(p)})M_{(p)}^{-1}.
\end{align*}
For each $k \in \bar{D}_{1}^{R}$, the function $[\mathcal{L},M_{(p)}](\cdot,k)$ only involves the off-diagonal entries of $M_{(p)}$, and thus $[\mathcal{L},M_{(p)}](\cdot,k) \in \mathcal{S}(\mathbb{R})$ by \eqref{recursive prop on Xj}. 
Since $M_{(p)}$ and $M_{(p)}^{-1}$ are uniformly bounded for $x \in \R$ and $|k| \geq R$, and since $\mathsf{U}(\cdot,k)\in \mathcal{S}(\mathbb{R})$, we have $\Delta(\cdot,k) \in \mathcal{S}(\mathbb{R})$ for each $k \in \bar{D}_{1}^{R}$. Furthermore, since $X_{formal}$ is a formal solution of (\ref{xpart}) (i.e., it satisfies \eqref{xrecursive}), there exists a function $b_1 \in \mathcal{S}(\R)$ such that
\begin{align*}
|\mathsf{U}M_{(p)} + [\mathcal{L},M_{(p)}] - \partial_{x}M_{(p)}| < \frac{b_{1}(x)}{|k|^{p+1}}, \qquad x \in \mathbb{R}, \; k \in \bar{D}_{1}^{R}.
\end{align*}
Since $M_{(p)}^{-1}$ is uniformly bounded for $x \in \R$ and $|k| \geq R$, we find
\begin{align}\label{Deltabound}
|\Delta(x,k)| < C\frac{b_1(x)}{|k|^{p+1}}, \qquad x \in \R, \ k \in \bar{D}_1^R.
\end{align}
We infer that there exists a function $b \in \mathcal{S}(\R)$ such that the following analog of (\ref{Kmbound}) holds:
\begin{align}\label{Kmboundinfty}
|K(x,x',k)_{il}| \leq \frac{b(x')}{|k|^{p+1}}, \qquad x, x' \in \R, \ k \in \bar{D}_1^R, \ i,l = 1, 2,3.
\end{align}
Define $f^{(m)}$ and $F^{(m)}$ as in \eqref{FredholmFmdef}, but with $K(x,x',k)$ given by \eqref{def new kernel}. Proceeding as in the proof of Proposition \ref{Mnprop}, we find the following analogs of (\ref{fmkestimate}) and (\ref{Fmxxkestimate}):
\begin{align*}
& |f^{(m)}(k)| \leq \frac{3^mm^{m/2}\|b\|_{L^1(\R)}^m}{m! |k|^{(p+1)m}}, \qquad k \in \bar{D}_1^R, \ m \geq 0.
	\\
& |F^{(m)}(x,x',k)_{ii'}| \leq \frac{3^m(m+1)^{(m+1)/2}\|b\|_{L^1(\R)}^mb(x')}{m! |k|^{(p+1)(m+1)}}, \qquad x,x' \in \R, \ k \in \bar{D}_1^R, \ m \geq 0,
\end{align*}
As a result, the Fredholm determinant $f(k) = \sum_{m=0}^\infty f^{(m)}(k)$ and the Fredholm minor $F(x,x',k) = \sum_{m=0}^\infty F^{(m)}(x,x',k)$ associated with equation (\ref{wieqinfty}) obey the estimates
\begin{subequations}\label{fmFmbounds}
\begin{align}
& |f(k) - 1| \leq \frac{C}{|k|^{p+1}}, \qquad k \in \bar{D}_1^R, \label{fmFmboundsa}
	\\
& |F(x,x',k)| \leq C\frac{b(x')}{|k|^{p+1}}, \qquad x,x' \in \R, \ k \in \bar{D}_1^R.
\end{align}
\end{subequations}
Increasing $R$ if necessary, \eqref{fmFmboundsa} implies $|f(k)| \geq 1/2$ for $k \in \bar{D}_1^R$, and we arrive at the solution representation
$$w_i(x,k) = M_{(p)}(x,k)_{i3} + \frac{1}{f(k)} \int_\R \sum_{l=1}^3 F(x,x',k)_{il} M_{(p)}(x',k)_{l3} dx', \qquad x \in \R, \ k \in \bar{D}_1^R.$$
In view of the boundedness of $M_{(p)}$ and the estimates (\ref{fmFmbounds}), this yields
$$|w_i(x,k) - M_{(p)}(x,k)_{i3}| \leq \frac{C}{|k|^{p+1}}, \qquad x \in \R, \ k \in \bar{D}_1^R.$$
This proves (\ref{Matinfty}) in the case of the third column and $k \in D_1$; the proofs of the other cases are similar.
\end{proof}

We saw in (\ref{XYs}) that the spectral function  $s(k)$ relates the eigenfunctions $X$ and $Y$ if $u_0$ and $v_0$ have compact support. The next lemma introduces spectral functions $S_n(k)$ and $T_n(k)$, $n = 1, \dots, 6$, which relate the eigenfunctions $M_n$ with $X$ and $Y$. 

\begin{lemma}[Relation between $M_n$ and $X,Y$]\label{Snexplicitlemma}
Suppose $u_0,v_0 \in \mathcal{S}(\R)$ have compact support. 
Then
\begin{align*}
   M_n(x,k) & = Y(x, k) e^{x\widehat{\mathcal{L}(k)}} S_n(k)	
  \\
&  = X(x, k) e^{x\widehat{\mathcal{L}(k)}} T_n(k), \qquad x \in \R, \ k \in \bar{D}_n\setminus \mathcal{Q}, \ n = 1, \dots, 6,
\end{align*}
where $S_n(k)$ and $T_n(k)$ are given in terms of the entries of $s(k)$ by
\begin{subequations}\label{SnTnexplicit}
\begin{align}\nonumber
&  S_1(k) = \begin{pmatrix}
 s_{11} & 0 & 0 \\
 s_{21} & \frac{m_{33}(s)}{s_{11}} & 0 \\
 s_{31} & \frac{m_{23}(s)}{s_{11}} & \frac{1}{m_{33}(s)} \\
  \end{pmatrix},
&&
  S_2(k) =  \begin{pmatrix}
 s_{11} & 0 & 0 \\
 s_{21} & \frac{1}{m_{22}(s)} & \frac{m_{32}(s)}{s_{11}} \\
 s_{31} & 0 & \frac{m_{22}(s)}{s_{11}} \\
\end{pmatrix},
	\\ \nonumber
&  S_3(k) = \begin{pmatrix}
 \frac{m_{22}(s)}{s_{33}} & 0 & s_{13} \\
 \frac{m_{12}(s)}{s_{33}} & \frac{1}{m_{22}(s)} & s_{23} \\
 0 & 0 & s_{33} \\
\end{pmatrix},
&&
  S_4(k) =  \begin{pmatrix}
  \frac{1}{m_{11}(s)} & \frac{m_{21}(s)}{s_{33}} & s_{13} \\
 0 & \frac{m_{11}(s)}{s_{33}} & s_{23} \\
 0 & 0 & s_{33} \\
\end{pmatrix},
	\\ \label{Snexplicit}
&  S_5(k) = \begin{pmatrix}
 \frac{1}{m_{11}(s)} & s_{12} & -\frac{m_{31}(s)}{s_{22}} \\
 0 & s_{22} & 0 \\
 0 & s_{32} & \frac{m_{11}(s)}{s_{22}} \\
  \end{pmatrix},
&&
  S_6(k) =  \begin{pmatrix}
 \frac{m_{33}(s)}{s_{22}} & s_{12} & 0 \\
 0 & s_{22} & 0 \\
 -\frac{m_{13}(s)}{s_{22}} & s_{32} & \frac{1}{m_{33}(s)} \\
 \end{pmatrix},	
\end{align}
and
\begin{align}\nonumber
&  T_1(k) = \begin{pmatrix}
  1 & -\frac{s_{12}}{s_{11}} & \frac{m_{31}(s)}{m_{33}(s)} \\
 0 & 1 & -\frac{m_{32}(s)}{m_{33}(s)} \\
 0 & 0 & 1
   \end{pmatrix},
&&
  T_2(k) =  \begin{pmatrix}
 1 & -\frac{m_{21}(s)}{m_{22}(s)} & -\frac{s_{13}}{s_{11}} \\
 0 & 1 & 0 \\
 0 & -\frac{m_{23}(s)}{m_{22}(s)} & 1 
\end{pmatrix},
	\\ \nonumber
&  T_3(k) = \begin{pmatrix}
 1 & -\frac{m_{21}(s)}{m_{22}(s)} & 0 \\
 0 & 1 & 0 \\
 -\frac{s_{31}}{s_{33}} & -\frac{m_{23}(s)}{m_{22}(s)} & 1 
\end{pmatrix},
&&
  T_4(k) =  \begin{pmatrix}
 1 & 0 & 0 \\
 -\frac{m_{12}(s)}{m_{11}(s)} & 1 & 0 \\
 \frac{m_{13}(s)}{m_{11}(s)} & -\frac{s_{32}}{s_{33}} & 1 
\end{pmatrix},
	\\ \label{Tnexplicit}
&  T_5(k) = \begin{pmatrix}
 1 & 0 & 0 \\
 -\frac{m_{12}(s)}{m_{11}(s)} & 1 & -\frac{s_{23}}{s_{22}} \\
 \frac{m_{13}(s)}{m_{11}(s)} & 0 & 1 
  \end{pmatrix},
&&
  T_6(k) =  \begin{pmatrix}
 1 & 0 & \frac{m_{31}(s)}{m_{33}(s)} \\
 -\frac{s_{21}}{s_{22}} & 1 & -\frac{m_{32}(s)}{m_{33}(s)} \\
 0 & 0 & 1 
 \end{pmatrix}.
\end{align}
\end{subequations}
\end{lemma}
\proofbegin
Choose $K > 0$ such that $u_0, v_0$ have support in $[-K,K] \subset \R$. 
Define $S_n(k)$ and $T_n(k)$, $n = 1, \dots, 6$, by
\begin{align}\label{SnTndef}
\begin{cases}
  S_n(k) = \displaystyle{\lim_{x \to -\infty}} e^{-x\widehat{\mathcal{L}(k)}}M_n(x,k), 
 	\\
  T_n(k) =  \displaystyle{\lim_{x \to \infty}} e^{-x\widehat{\mathcal{L}(k)}}M_n(x,k), 
  \end{cases}
  \quad k \in \bar{D}_n\setminus \mathcal{Q}, 
\end{align}
where the limits exist because $\mathsf{U}(x,k) = 0$ for $|x| > K$, which implies by \eqref{Mndef} that $e^{-x\hat{\mathcal{L}}}M_n(x,k)$ is independent of $x$  for $|x| > K$.
Recall that $X(x,k), Y(x,k)$, and $s(k)$ are defined for all $k \in \C \setminus \{0\}$ for compactly supported data.
We find 
\begin{align}
M_n(x,k) = Y(x,k) e^{x\widehat{\mathcal{L}(k)}} S_n(k) = X(x,k) e^{x\widehat{\mathcal{L}(k)}} T_n(k),
\end{align}   
and hence, comparing with (\ref{XYs}), 
\begin{equation}\label{sSSnrelations}  
  s(k) = S_n(k)T_n^{-1}(k), \qquad k \in \bar{D}_n\setminus \mathcal{Q}.
\end{equation}
Given $s(k)$, equation (\ref{sSSnrelations}) constitutes a matrix factorization problem which can be uniquely solved for $S_n(k)$ and $T_n(k)$.
In fact, the integral equations (\ref{Mndef}) imply that
\begin{align} \nonumber
& \left(S_n(k)\right)_{ij} = 0 \quad \text{if} \quad \gamma_{ij}^n = (-\infty,x),
	\\ \nonumber
& \left(T_n(k)\right)_{ij} = \delta_{ij} \quad \text{if} \quad \gamma_{ij}^n = (\infty,x),
\end{align}
so the relation (\ref{sSSnrelations}) yields $9$ scalar equations for $9$ unknowns. 
The explicit solution of this algebraic system gives (\ref{SnTnexplicit}).
\proofend

Let $\eta \in C_c^\infty(\R)$ be a cutoff function which equals one for $|x| \leq 1$ and which vanishes for $|x| \geq 2$. For $j \geq 1$, let $\eta_j(x) = \eta(x/j)$. If $f \in \mathcal{S}(\R)$, then $\eta_jf$ is a sequence of smooth functions with compact support which converges to $f$ in $\mathcal{S}(\R)$ as $j \to \infty$.

\begin{lemma}\label{sequencelemma}
Let $u_0, v_0 \in \mathcal{S}(\R)$.
Let $\{s(k), M_n(x,k)\}$ and $\{s^{(i)}(k), M_n^{(i)}(x,k)\}$ be the spectral functions and eigenfunctions associated with $(u_0, v_0)$ and 
\begin{align}\label{uvsequence}
(u_0^{(i)}(x), v_0^{(i)}(x)) := (\eta_iu_0, \eta_iv_0) \in \mathcal{S}(\R) \times \mathcal{S}(\R),
\end{align} 
respectively. Then 
\begin{align}\label{slimiti}
& \lim_{i\to\infty} s^{(i)}(k) = s(k), \qquad k \in \begin{pmatrix}
 \omega^2 \bar{\mathrm{S}} & \R_+ & \omega \R_+ \\
 \R_+ & \omega \bar{\mathrm{S}} & \omega^2 \R_+ \\
 \omega \R_+ & \omega^2 \R_+ & \bar{\mathrm{S}}
 \end{pmatrix}\setminus \{0\},
 	\\\label{sAlimiti}
& \lim_{i\to\infty} (s^A)^{(i)}(k) = s^A(k), \qquad k \in	 \begin{pmatrix}
 -\omega^2 \bar{\mathrm{S}} & \R_- & \omega \R_- \\
 \R_- & -\omega \bar{\mathrm{S}} & \omega^2 \R_- \\
 \omega \R_- & \omega^2 \R_- & -\bar{\mathrm{S}}
 \end{pmatrix}\setminus \{0\},
	\\ \label{Xlimiti}
& \lim_{i\to \infty} X^{(i)}(x,k) = X(x,k), \qquad x \in \R, \ k \in (\omega^2 \bar{\mathrm{S}}, \omega \bar{\mathrm{S}}, \bar{\mathrm{S}}) \setminus \{0\},
	\\ \label{Ylimiti}
& \lim_{i\to \infty} Y^{(i)}(x,k) = Y(x,k), \qquad x \in \R, \ k \in (-\omega^2 \bar{\mathrm{S}}, -\omega \bar{\mathrm{S}}, -\bar{\mathrm{S}}) \setminus \{0\},
	\\ \label{Mnlimiti}
& \lim_{i\to \infty} M_n^{(i)}(x,k) = M_n(x,k), \qquad x \in \R, \ k \in \bar{D}_n\setminus \mathcal{Q}, \ n = 1, \dots, 6.
\end{align}
\end{lemma}
\begin{proof}
The proof amounts to verifying that the solutions of the Volterra equations (\ref{XYdef}) and of the Fredholm equation (\ref{Mndef}) depend continuously on the potential $(u_0, v_0) \in \mathcal{S}(\R)$. 
Let $\epsilon > 0$ and let the superscript $(i)$ indicate quantities associated with the sequence $(u_0^{(i)}(x), v_0^{(i)}(x))$.
It is easy to see from the Volterra series that $X^{(i)}(x,k)$ converges pointwise to $X(x,k)$ as $i \to \infty$. Moreover, the following bound holds uniformly with respect to $i$ (cf. (\ref{Xesta})):
\begin{align}\label{Xibounded}
|X^{(i)}(x,k) - I| \leq C, \qquad x \in \R, \ k \in (\omega^2 \bar{\mathrm{S}}, \omega \bar{\mathrm{S}}, \bar{\mathrm{S}}), \ |k| > \epsilon.
\end{align}
Indeed, the properties of the Volterra series defining $X^{(i)}(x,k)$ depend on the norm $\|\mathsf{U}^{(i)}(\cdot,k)\|_{L^1([x,\infty))}$. Since these norms are uniformly bounded, we find (\ref{Xibounded}). The limit (\ref{slimiti}) now follows from the definition (\ref{sdef}) of $s(k)$ and dominated convergence. 
The proof of (\ref{sAlimiti}) is similar.

We next prove (\ref{Mnlimiti}).
Consider the kernel $K$ defined in (\ref{Kjkernel}). The expression \eqref{expression for U} for $\mathsf{U}$ implies that $K^{(i)}(x,x',k)$ converges pointwise to $K(x,x',k)$ as $i \to \infty$. Moreover, we can choose a function $b \in \mathcal{S}(\R)$ such that the bound (\ref{Kmbound}) holds uniformly for all $i$:
$$\sup_{i} |K^{(i)}(x,x',k)| \leq b(x'), \qquad x, x' \in \R, \ k \in \bar{D}_1^\epsilon.$$
Using dominated convergence, we can then take the limit $i \to \infty$ in the formulas involving the Fredholm determinant and Fredholm minors in the proof of Proposition \ref{Mnprop}. In particular, the Fredholm determinant $f^{(i)}(k)$ converges pointwise to $f(k)$ and, by uniform convergence of the series in (\ref{FredholmFdef}), the minor $F^{(i)}(x,x',k)$ converges pointwise to $F(x,x',k)$ and the bound (\ref{FredholmFbound}) holds uniformly with respect to  $i$. In the case of $n=1$ and the third column, the limit (\ref{Mnlimiti}) follows by using dominated convergence in the representation (\ref{wifinal}); the other cases are similar. 
\end{proof}

Recall that the sectionally analytic function $M(x,k)$ is defined by $M(x,k) = M_n(x,k)$ for $k \in D_n$.

\begin{lemma}[Jump condition for $M$]\label{Mjumplemma}
Let $u_0, v_0 \in \mathcal{S}(\R)$.
For each $x \in \R$, $M(x,k)$ satisfies the jump condition
\begin{align*}
  M_+(x,k) = M_-(x, k) v(x, 0, k), \qquad k \in \Gamma \setminus \mathcal{Q},
\end{align*}
where $v$ is the jump matrix defined in (\ref{vdef}) and $\mathcal{Q}$ is the set defined in (\ref{calQdef}).
\end{lemma}
\begin{proof}
We will show that 
\begin{align}\label{M1M6v1}
M_1 = M_6 v_1, \qquad k \in (0, \infty)\setminus \mathcal{Q};
\end{align}
the proof that $M_4 = M_3 v_4$ for $k \in (-\infty, 0)\setminus \mathcal{Q}$ is similar, and the jumps on the remaining parts of $\Gamma$ follow from these two jumps by symmetry. 

Suppose first that $u_0, v_0$ have support in some compact subset $[-K,K] \subset \R$,  $K > 0$. 
For each $k$, $M_n(x,k)$ is a smooth function of $x \in \R$ which satisfies (\ref{xpart}). 
Hence there exists a matrix $J_1(k)$ independent of $x$  such that
\begin{align}\label{M1M6J1}
M_1(x,k) = M_6(x,k) e^{x\widehat{\mathcal{L}(k)}}J_1(k), \qquad k \in (0, \infty).
\end{align}
For $x < -K$ we have $M_n(x,k) = e^{x\widehat{\mathcal{L}(k)}}S_n(k)$, where $S_n(k)$ is the matrix defined in (\ref{SnTndef}). Hence, evaluation of (\ref{M1M6J1}) at $x < -K$ gives
$$J_1(k) = S_6(k)^{-1}S_1(k).$$
This completes the proof of (\ref{M1M6v1}) for compactly supported $u_0,v_0$, because, by (\ref{vdef}) and (\ref{Snexplicit}),
$$e^{x\hat{\mathcal{L}}}[S_6(k)^{-1}S_1(k)] = e^{x\hat{\mathcal{L}}}  \begin{pmatrix}  
 1 & - r_1(k) & 0 \\
  r_1^*(k) & 1 - |r_1(k)|^2 & 0 \\
  0 & 0 & 1
  \end{pmatrix} = v_1(x,0,k).$$

If $u_0,v_0 \in \mathcal{S}(\R)$ are not compactly supported, we let $(u_0^{(i)}(x), v_0^{(i)}(x))$ be the sequence of smooth functions with compact support converging to $(u_0, v_0)$ defined in (\ref{uvsequence}). In view of Lemma \ref{sequencelemma}, the relation (\ref{M1M6v1}) for $(u_0,v_0)$ follows immediately by taking the limit $i \to \infty$ in the analogous relation for $(u_0^{(i)}(x), v_0^{(i)}(x))$.
\end{proof}

\begin{lemma}\label{M1XYlemma}
Let $u_0, v_0 \in \mathcal{S}(\R)$.
The functions $M_1$ and $M_1^A \equiv (M_1^{-1})^T$ can be expressed in terms of the entries of $X,Y,X^A, Y^A, s$, and $s^A$ as follows:
\begin{align*}
M_1 = \begin{pmatrix} 
X_{11} & \frac{Y_{31}^AX_{23}^A - Y_{21}^AX_{33}^A}{s_{11}} & \frac{Y_{13}}{s_{33}^A} \\
X_{21} & \frac{Y_{11}^AX_{33}^A - Y_{31}^AX_{13}^A}{s_{11}} & \frac{Y_{23}}{s_{33}^A} \\
X_{31} & \frac{Y_{21}^AX_{13}^A - Y_{11}^AX_{23}^A}{s_{11}} & \frac{Y_{33}}{s_{33}^A} 
\end{pmatrix}, \qquad
M_1^A = \begin{pmatrix} 
\frac{Y_{11}^A}{s_{11}} & \frac{X_{31}Y_{23} - X_{21}Y_{33}}{s_{33}^A} & X_{13}^A \\
\frac{Y_{21}^A}{s_{11}} & \frac{X_{11}Y_{33} - X_{31}Y_{13}}{s_{33}^A} & X_{23}^A \\
\frac{Y_{31}^A}{s_{11}} & \frac{X_{21}Y_{13} - X_{11}Y_{23}}{s_{33}^A} & X_{33}^A 
\end{pmatrix},
\end{align*}
for all $x \in \R$ and $k \in \bar{D}_1 \setminus \mathcal{Q}$.
\end{lemma}
\begin{proof}
Let $(u_0^{(i)}(x), v_0^{(i)}(x))$ be the sequence converging to $(u_0(x), v_0(x))$ in (\ref{uvsequence}). Then Lemma \ref{Snexplicitlemma} implies
\begin{align*}
   M_n^{(i)}(x,k) & = Y^{(i)}(x, k) e^{x\widehat{\mathcal{L}(k)}} S_n^{(i)}(k)	
= X^{(i)}(x, k) e^{x\widehat{\mathcal{L}(k)}} T_n^{(i)}(k), \qquad k \in \bar{D}_n\setminus \mathcal{Q}.
\end{align*}
In particular, the first and third columns of $M_1^{(i)}$ admit the following representations for $x \in \R$, $k \in \bar{D}_1\setminus \mathcal{Q}$, and $i \geq 1$:
\begin{align*}
\begin{cases}
 [M_1^{(i)}(x,k)]_1 = [X^{(i)}(x,k)]_1,
	\\
 [M_1^{(i)}(x,k)]_3 = \frac{[Y^{(i)}(x,k)]_3}{m_{33}(s^{(i)})},
\end{cases}
\end{align*}
where we use the notation $[A]_j$ for the $j$th column of a matrix $A$. 
Using Lemma \ref{sequencelemma}  to let $i \to \infty$, we find (note that all quantities are well-defined also when $u_0, v_0$ are not compactly supported)
\begin{align}\label{M1XY}
\begin{cases}
 [M_1(x,k)]_1 = [X(x,k)]_1,
	\\
 [M_1(x,k)]_3 = \frac{[Y(x,k)]_3}{(s^A(k))_{33}},
 \end{cases} \quad x \in \R, \ k \in \bar{D}_1\setminus \mathcal{Q}.
\end{align} 
Analogous arguments using that
\begin{align*}
S_1(k)^A = \begin{pmatrix}
 \frac{1}{s_{11}} & -\frac{s_{21}}{m_{33}(s)} & m_{13}(s) \\
 0 & \frac{s_{11}}{m_{33}(s)} & -m_{23}(s) \\
 0 & 0 & m_{33}(s)
 \end{pmatrix}, \qquad
 T_1(k)^A = \begin{pmatrix}
  1 & 0 & 0 \\
 \frac{s_{12}}{s_{11}} & 1 & 0 \\
 \frac{s_{13}}{s_{11}} & \frac{m_{32}(s)}{m_{33}(s)} & 1 
 \end{pmatrix},
\end{align*} 
show that
\begin{align}\label{M1AXAYA}
\begin{cases}
 [M_1^A(x,k)]_1 = \frac{[Y^A(x,k)]_1}{s_{11}(k)},
	\\
 [M_1^A(x,k)]_3 = [X^A(x,k)]_3,
 \end{cases} \quad x \in \R, \ k \in \bar{D}_1\setminus \mathcal{Q}.
\end{align} 
On the other hand, the trivial identity $M = (M^A)^A$ shows that $[M]_2$ can be expressed in terms of $[M^A]_1$ and $[M^A]_3$ by
\begin{align}\label{Msecondcolumnd}
M_{12} = -m_{12}(M^A), \quad
M_{22} = m_{22}(M^A), \quad
M_{32} = -m_{32}(M^A).
\end{align}
Similarly $[M^A]_2$ can be expressed in terms of $[M]_1$ and $[M]_3$ by
\begin{align}\label{MAsecondcolumnd}
M_{12}^A = -m_{12}(M), \quad
M_{22}^A = m_{22}(M), \quad
M_{32}^A = -m_{32}(M).
\end{align}
The lemma follows from equations (\ref{M1XY})-(\ref{MAsecondcolumnd}).
\end{proof}

Lemma \ref{M1XYlemma} shows that if $u_0, v_0$ satisfy Assumption \ref{solitonlessassumption}, then $M$ has no singularities apart from $k=0$ (because $s_{11}^{A}(k) \neq 0$ for $k \in \bar{D}_{4}\setminus \{0\}$ implies $s_{33}^{A}(k)=s_{22}^{A}(\omega k) = \overline{s_{11}^{A}(\overline{\omega k})} \neq 0$ for $k \in \bar{D}_{1}\setminus \{0\}$); roughly speaking, this absence of singularities corresponds to the absence of solitons. In this case, we can define the value of $M(x,t,k)$ at any point $k_j \in \mathcal{Q} \cap \bar{D}_n  \setminus \{0\}$ by continuity:
\begin{align}\label{Mnkjdef}
M_n(x,t,k_j) = \lim_{\underset{k \in \bar{D}_n \setminus \mathcal{Q}}{k\to k_j}} M_n(x,t,k).
\end{align}
As a consequence, we can replace $\mathcal{Q}$ with $\{0\}$ in all the above results.

\begin{lemma}\label{QtildeQlemma}
Suppose $u_0,v_0 \in \mathcal{S}(\R)$ are such that Assumption \ref{solitonlessassumption} holds. Then the statements of Proposition \ref{Mnprop} and Lemmas \ref{Snexplicitlemma}-\ref{Mjumplemma} hold with $\mathcal{Q}$ replaced by $\{0\}$.
\end{lemma}
\begin{proof}
Lemma \ref{M1XYlemma} shows that, by analyticity and continuity arguments, the statements of Proposition \ref{Mnprop} and Lemmas \ref{Snexplicitlemma}-\ref{Mjumplemma} can be extended to all $k \notin \{0\}$.
\end{proof}

\begin{lemma}[Asymptotics of $M$ as $k \to 0$]\label{Mat1lemma}
Suppose $u_0,v_0 \in \mathcal{S}(\R)$ are such that Assumptions \ref{solitonlessassumption} and \ref{originassumption} hold.
Let $p \geq 1$ be an integer.
Then there are $3 \times 3$-matrix valued functions $\{\mathcal{M}_n^{(l)}(x), \mathcal{N}_n^{(l)}(x)\}$, $n = 1,\dots,6$, $l = -2,-1,0, \dots, p$, with the following properties:
\begin{enumerate}[$(a)$]
\item The function $M$ satisfies, for $x \in \R$,
\begin{align*}
& \bigg|M_n(x,k) - \sum_{l=-2}^p \mathcal{M}_n^{(l)}(x)k^l\bigg| \leq C
|k|^{p+1}, \qquad |k| \leq \frac{1}{2}, \ k \in \bar{D}_n.
\end{align*}
\item The function $M^{-1}$ satisfies, for $x \in \R$,
\begin{align*}
& \bigg|M_n(x,k)^{-1} - \sum_{l=-2}^p \mathcal{N}_n^{(l)}(x)k^l\bigg| \leq
C|k|^{p+1}, \qquad |k| \leq \frac{1}{2}, \ k \in \bar{D}_n.
\end{align*}
\item For each $n$ and each $l \geq -2$, $\{\mathcal{M}_n^{(l)}(x), \mathcal{N}_n^{(l)}(x)\}$ are smooth functions of $x \in \R$.
\item For n=1, the first coefficients are given by
\begin{align*}
& \mathcal{M}_{1}^{(-2)}(x) = \alpha_{1}(x) \begin{pmatrix}
\omega & 0 & 0 \\
\omega & 0 & 0 \\
\omega & 0 & 0
\end{pmatrix}, \\
& \mathcal{M}_{1}^{(-1)}(x) = \beta_{1}(x) \begin{pmatrix}
\omega^{2} & 0 & 0 \\
\omega^{2} & 0 & 0 \\
\omega^{2} & 0 & 0 
\end{pmatrix} + \gamma_{1}(x) \begin{pmatrix}
\omega^{2} & 0 & 0 \\
1 & 0 & 0 \\
\omega & 0 & 0
\end{pmatrix} + \delta(x) \begin{pmatrix}
0 & 1-\omega & 0 \\
0 & 1-\omega & 0 \\
0 & 1-\omega & 0
\end{pmatrix}, 
\end{align*}
where
\begin{align*}
\delta(x) = \frac{\tilde{\alpha}_{2}(x) \tilde{\gamma}_{1}(x)-\tilde{\alpha}_{1}(x) \tilde{\gamma}_{2}(x)}{\mathfrak{s}^{(-2)}},
\end{align*}
and the third column of $\mathcal{M}_{1}^{(0)}(x)$ is given by
\begin{align}
\mathcal{M}_{1}^{(0)}(x) = \begin{pmatrix}
\star & \star & \epsilon(x) \\
\star & \star & \epsilon(x) \\
\star & \star & \epsilon(x) 
\end{pmatrix}  \qquad \mbox{with } \quad \epsilon(x) = \frac{\alpha_{2}(x)}{\mathfrak{s}^{A(-2)}},
\end{align}
where $\star$ denotes an unspecified entry.
Moreover,
\begin{align}
& \mathcal{N}_{1}^{(-2)}(x) = \tilde{\alpha}_{1}(x) \begin{pmatrix}
0 & 0 & 0 \\
0 & 0 & 0 \\
\omega & \omega^{2} & 1
\end{pmatrix}, \label{explicit Ncalpm2p} \\
& \mathcal{N}_{1}^{(-1)}(x) = \tilde{\beta}_{1}(x) \begin{pmatrix}
0 & 0 & 0 \\
0 & 0 & 0 \\
\omega & \omega^{2} & 1 
\end{pmatrix} + \tilde{\gamma}_{1}(x) \begin{pmatrix}
0 & 0 & 0 \\
0 & 0 & 0 \\
\omega^{2} & \omega & 1
\end{pmatrix} \\
& \hspace{1.7cm} + \tilde{\delta}(x) \begin{pmatrix}
0 & 0 & 0 \\
\omega-1 & \omega^{2}-\omega & 1-\omega^{2} \\
0 & 0 & 0
\end{pmatrix}, \label{explicit Ncalpm1p}
\end{align}
where
\begin{align*}
\tilde{\delta}(x) = \frac{\alpha_{2}(x) \gamma_{1}(x)-\alpha_{1}(x) \gamma_{2}(x)}{\mathfrak{s}^{A(-2)}},
\end{align*}
and the first row of $\mathcal{N}_{1}^{(0)}(x)$ is given by
\begin{align}
\mathcal{N}_{1}^{(0)}(x) = \begin{pmatrix}
\tilde{\epsilon}(x) & \tilde{\epsilon}(x) & \tilde{\epsilon}(x) \\
\star & \star & \star \\
\star & \star & \star 
\end{pmatrix} \qquad \mbox{with } \quad \tilde{\epsilon}(x) = \frac{\tilde{\alpha}_{2}(x)}{\mathfrak{s}^{(-2)}}. \label{explicit Ncalp0p third column}
\end{align}

\item For each $x \in \mathbb{R}$, the function $k \mapsto \begin{pmatrix}
\omega & \omega^{2} & 1
\end{pmatrix}M_{n}(x,k)$ is bounded as $k \to 0$, $k \in \bar{D}_{n}$.
\end{enumerate}
\end{lemma}
\begin{proof}
We only consider the case $k \in \bar{D}_1$. The asymptotics in the other sectors can be obtained from the symmetry $M(x,k) = \mathcal{A}M(x,\omega k) \mathcal{A}^{-1} = \mathcal{B}\overline{M(x,\overline{k})}\mathcal{B}$. Since $u_0,v_0$ satisfy Assumption \ref{originassumption}, $\mathfrak{s}^{(-2)} \neq 0$ and $\mathfrak{s}^{A(-2)} \neq 0$. Lemma \ref{M1XYlemma} provides expressions for $M_1$ and $M_1^A$ for $k \in \bar{D}_1\setminus \{0\}$ in terms of $X, Y, s, X^A, Y^A, s^A$.  
Using the expansions as $k \to 0$ of the latter set of functions (see Propositions \ref{XYat1prop}, \ref{sprop}, \ref{XAYAat1prop}, and \ref{sAprop}) and recalling Assumptions \ref{solitonlessassumption} and \ref{originassumption}, the lemma follows.
\end{proof}

\section{Proof of Theorem \ref{RHth}}\label{RHsec}
Suppose $\{u(x,t), v(x,t)\}$ is a Schwartz class solution of (\ref{boussinesqsystem}) with existence time $T \in (0, \infty]$ and initial data $u_0, v_0 \in \mathcal{S}(\R)$. Suppose Assumptions \ref{solitonlessassumption} and \ref{originassumption} hold and define the spectral functions $r_1(k)$ and $r_2(k)$ in terms of $u_0, v_0$ by (\ref{r1r2def}).
Define time-dependent eigenfunctions $\{M_n(x,t,k)\}_{n=1}^6$ by replacing $\mathsf{U}(x,k)$ with the time-dependent potential $\mathsf{U}(x,t,k)$ in the integral equations (\ref{Mndef}). 
Define the sectionally analytic function $M(x,t,k)$ by setting $M(x,t,k) = M_n(x,t,k)$ for $k \in D_n$. 

By Lemma \ref{Matinftylemma} and the definition (\ref{Xpdef}) of $X_{(p)}$, we have
$$\lim_{k\to \infty}k\big[(M(x,t,k))_{33} - 1\big] = -\frac{2}{3}  \int_{\infty}^{x} u(x^{\prime}, t) dx'.$$
Recalling that $u,v$ have rapid decay as $x \to \infty$ and that $u_t = v_x$, the formulas (\ref{recoveruv}) for $u$ and $v$ follow. On the other hand, uniqueness of the solution of the RH problem \ref{RH problem for M} is proved in Appendix \ref{appA}. Thus it only remains to verify that $M$ satisfies the RH problem \ref{RH problem for M}.

Property $(c)$ of RH problem \ref{RH problem for M} related to the asymptotics of $M$ as $k \to \infty$ follows from Lemma \ref{Matinftylemma} and the definition (\ref{Xpdef}) of $X_{(p)}$.
Property $(d)$ related to the asymptotics of $M$ as $k \to 0$ follows from Lemma \ref{Mat1lemma}, whereas property $(e)$ follows from the symmetries (\ref{Msymm}) of $M$ and Lemma \ref{QtildeQlemma}. To prove the remaining two properties (properties $(a)$ and $(b)$) of RH problem \ref{RH problem for M} we need the following lemma. 

\begin{lemma}\label{Mlaxlemma}
Let $n = 1, \dots, 6$. For each $k \in \bar{D}_n \setminus \{0\}$, $M_n(x,t,k)$ is a smooth function of $(x,t) \in \R \times [0,T)$ satisfying the Lax pair equations \eqref{Xlax}.
\end{lemma}
\begin{proof}
For each fixed $t$, the eigenfunctions $M_n$ possess all the properties derived in Section \ref{specsec}; in particular, they depend smoothly on $x$ and satisfy the $x$-part of \eqref{Xlax}. 
Using the smoothness of $\{u(x,t), v(x,t)\}$ together with the decay assumption (\ref{rapiddecay}), differentiation of the representation (\ref{wifinal}) for $M_n$ shows that the $M_n$ depend smoothly on $t$ whenever $k \in \bar{D}_n\setminus \mathcal{Q}$, where $\mathcal{Q}$ is the set defined in (\ref{calQdef}). 
Furthermore, differentiation of (\ref{Mndef}) with respect to $t$ shows that if $k \in \bar{D}_n\setminus \mathcal{Q}$, then the derivative $\partial_tM_n$ satisfies the Fredholm equation
\begin{align}\nonumber
(\partial_tM_n)_{ij}(x,t,k) = &\; \int_{\gamma_{ij}^n} \left(e^{(x-x')\widehat{\mathcal{L}(k)}} (\mathsf{U}_t M_n)(x',t,k) \right)_{ij} dx' 
	\\\label{Mntdef}
& + \int_{\gamma_{ij}^n} \left(e^{(x-x')\widehat{\mathcal{L}(k)}} (\mathsf{U}\partial_tM_n)(x',t,k) \right)_{ij} dx', \qquad  i,j = 1, 2,3.
\end{align}
Suppose $k \in \bar{D}_n\setminus \mathcal{Q}$. The function $\check{M}_n := M_ne^{\mathcal{L} x + \mathcal{Z} t}$ satisfies the $x$-part in (\ref{Xhatlax}). Furthermore, since $\{u(x,t), v(x,t)\}$ solve (\ref{boussinesqsystem}), $L$ and $Z$ obey the compatibility condition $L_t - Z_x + [L, Z] = 0$. It follows that $\partial_t\check{M}_n - Z\check{M}_n$ also satisfies the $x$-part in (\ref{Xhatlax}), or, in other words, that 
$$\chi_n := (\partial_t\check{M}_n - Z\check{M}_n)e^{-\mathcal{L} x - \mathcal{Z} t}
= \partial_tM_n - [\mathcal{Z}, M_n] - \mathsf{V}M_n$$
satisfies the $x$-part in (\ref{Xlax}). 
For each $t \geq 0$, $M_n(x,t,k)$ is bounded for $x \in \R$ by (\ref{Mnbounded}) and $\mathsf{V}(x,t,k)$ has decay as $|x| \to \infty$. Moreover, we conclude from (\ref{Mndef}) and (\ref{Mntdef}) that the $(ij)$th entries of $M_n$ and $\partial_tM_n$ approach zero as $x \to + \infty$ if $\gamma_{ij}^n = (+\infty, x)$, whereas they approach zero as $x \to - \infty$ if $\gamma_{ij}^n = (-\infty, x)$. 
It follows that 
$$(\chi_n)_{ij} \to 0 \quad \text{as} \quad \begin{cases} \text{$x \to - \infty$ if $\gamma_{ij}^n = (-\infty, x)$}, \\ 
\text{$x \to + \infty$ if $\gamma_{ij}^n = (+\infty, x)$},
\end{cases}
\quad i,j = 1,2,3.$$
Hence $\chi_n$ satisfies the homogeneous version of the Fredholm equations (\ref{Mndef}) (i.e., the equations obtained from (\ref{Mndef}) by replacing $\delta_{ij}$ with zero). These homogeneous equations have no nonzero solution whenever the Fredholm determinant is nonzero.
 Hence $\chi_n$ vanishes identically for each $k \in \bar{D}_n\setminus \mathcal{Q}$. This shows that $M_n$ satisfies the $t$-part of (\ref{Xlax}) for $k \in \bar{D}_n\setminus \mathcal{Q}$. Since the right-hand side of this $t$-part extends continuously to $k \in \bar{D}_n\setminus \{0\}$, so does $\partial_tM_n$ (and by uniform convergence the limit in (\ref{Mnkjdef}) commutes with $\partial_t$, see e.g. \cite[Theorem 7.17]{R1976}). Thus the $t$-part of (\ref{Xlax}) holds for all $k \in \bar{D}_n \setminus \{0\}$. Recursive use of the $t$-part shows that $M_n$ depends smoothly on $t$ also for $k \in \mathcal{Q} \cap \bar{D}_n  \setminus \{0\}$. This completes the proof. 
\end{proof}

The next lemma completes the proof that $M$ satisfies the RH problem \ref{RH problem for M} and therefore also the proof of Theorem \ref{RHth}.

\begin{lemma}\label{Mxtjumplemma}
For each $(x,t) \in \R \times [0,T)$, $M(x,t,k)$ is an analytic function of $k \in \C \setminus \Gamma$ with continuous boundary values on $\Gamma\setminus \{0\}$. Moreover, $M$ satisfies the jump condition (\ref{Mjumpcondition}).
\end{lemma}
\begin{proof}
The analyticity and the existence of continuous boundary values are a consequence of Proposition \ref{Mnprop} and Lemma \ref{QtildeQlemma}.
By Lemma \ref{Mlaxlemma}, $M_n$ satisfies the Lax pair equations \eqref{Xlax}. These equations imply that the functions $M_n$ are related by
$$M_{n+1}(x,t,k) = M_n(x,t,k) e^{\hat{\mathcal{L}}x + \hat{\mathcal{Z}}t}\big(M_n(0,0,k)^{-1}M_{n+1}(0,0,k)\big),$$
for $(x,t) \in \R \times [0,T)$ and $k \in \bar{D}_{n}\cap \bar{D}_{n+1}$, $n=1, \dots,6$ (with $\bar{D}_{7}:=\bar{D}_{1}$ and $M_{7}:=M_{1}$). Equation (\ref{Mjumpcondition}) now follows from Lemma \ref{Mjumplemma}.
\end{proof}

\section{Numerical example}\label{numericalsec}
The functions $X$, $Y$, $s$, $X^{A}$, $Y^{A}$, $s^{A}$, $M$ can all be computed numerically. All properties of these functions listed in Sections \ref{specsec} and \ref{Msec} have been verified numerically for various choices of the initial data. The aim of this section is to illustrate the behavior of some of these functions as $k \to 0$ for a particular choice of $u_{0}$ and $v_{0}$. Our numerics are based on the Julia package ``ApproxFun" developed by Sheehan Olver (see \texttt{https://github.com/JuliaApproximation/ApproxFun.jl} and \cite{OT2013, OTS2020}).
For the illustration, we take the following compactly supported initial data $u_{0},v_{0} \in \mathcal{S}(\mathbb{R})$:
\begin{subequations}\label{u0v0num}
\begin{align}
& u_{0}(x) = \begin{cases}
-8(1+\cos(3 x))\exp\big(\frac{-2}{1-x^{2}}\big), & \mbox{if } x \in (-1,1), \\
0, & \mbox{otherwise},
\end{cases}  \\
& v_{0}(x) = \begin{cases}
8(1+x)\exp\big(\frac{-2}{1-x^{2}}\big), & \mbox{if } x \in (-1,1), \\
0, & \mbox{otherwise}.
\end{cases} 
\end{align}
\end{subequations}
The functions $\alpha_{1}$, $\beta_{1}$, $\gamma_{1}$, $\delta_{1,1}$, $\delta_{1,2}$, and $\delta_{1,3}$ describe the behavior of $X$ as $k \to 0$ (see Proposition \ref{XYat1prop}) and their graphs as functions of $x$ are displayed in Figure \ref{fig: alphatodelta}. We observe that $\alpha_{1}$, $\beta_{1}$, $\gamma_{1}$, and $\delta_{1,j}-1/3$, $j=1,2,3$, have decay as $x \to +\infty$. 
\begin{figure}
\begin{center}
\includegraphics[scale=0.25]{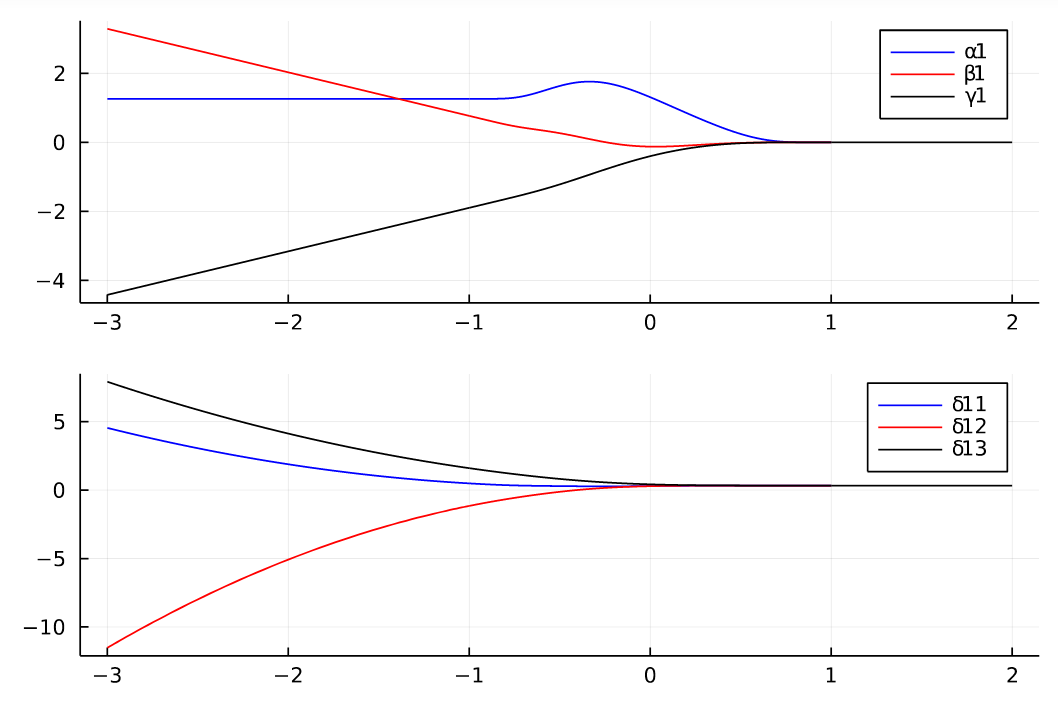}
\end{center}
\begin{figuretext}\label{fig: alphatodelta}The graphs of the functions $\alpha_{1}(x)$, $\beta_{1}(x)$, $\gamma_{1}(x)$, $\delta_{1,1}(x)$, $\delta_{1,2}(x)$, and $\delta_{1,3}(x)$ associated to the initial data \eqref{u0v0num}. 
\end{figuretext}
\end{figure}
\begin{figure}[h!]
\begin{center}
\includegraphics[scale=0.25]{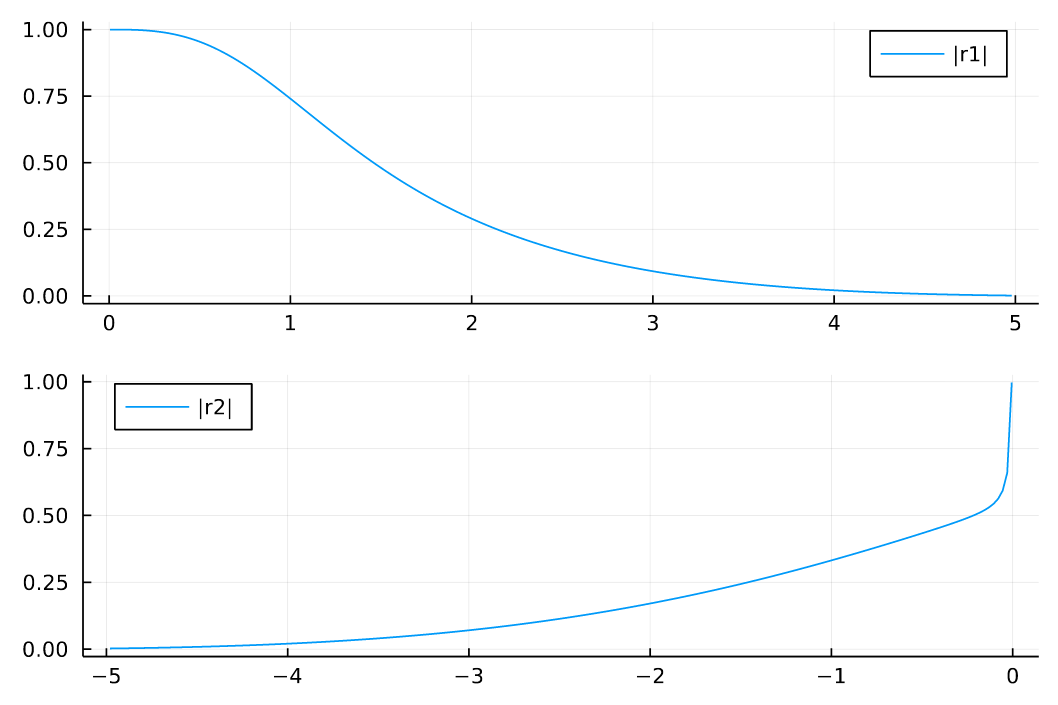}
\end{center}
\begin{figuretext}\label{fig: r1r2 plot}
The graphs of $|r_{1}(k)|$ and $|r_{2}(k)|$ for the initial data \eqref{u0v0num}. 
\end{figuretext}
\end{figure}
We also observe that $\alpha_{1}$ is bounded on $\mathbb{R}$, that $\beta_{1}$ and $\gamma_{1}$ grow linearly as $x \to -\infty$, and that $\delta_{1,j}$, $j=1,2,3$ seem to grow quadratically at $-\infty$. These observations are consistent with the bounds \eqref{bound alphai}-\eqref{bound deltai}. The value of $\mathfrak{s}^{(-2)}$ defined in \eqref{spm2p} is 1.26068... and we obtain excellent numerical agreement for the expansion \eqref{s at 0}. 
Finally, Figure \ref{fig: r1r2 plot} shows the graphs of $|r_1(k)|$ and $|r_2(k)|$ for $k > 0$ and $k < 0$, respectively. In agreement with Theorem \ref{r1r2th}, we observe that $|r_{1}(0)| = |r_{2}(0)| = 1$, that $|r_{1}(k)| \to 0$ as $k \to + \infty$, and that $|r_{2}(k)| \to 0$ as $k \to - \infty$. We also mention, as an interesting aside, that if Assumption \ref{solitonlessassumption} is relaxed, then it is possible to find examples with $|r_1(k)| > 1$ for some small $k >0$ and/or with $|r_2(k)| > 1$ for some small $k <0$. For the KdV equation, the reflection coefficient is bounded above by $1$ for all $k$ (see e.g. \cite{DVZ1994}); this shows that the Boussinesq equation allows for a more general behavior.

\appendix

\section{Uniqueness for RH problem \ref{RH problem for M}}\label{appA}

In this appendix, we show the following uniqueness result.

\begin{proposition}[Uniqueness of solution for RH problem \ref{RH problem for M}]\label{uniquenessprop}
Let $r_1:(0,\infty) \to \C$ and $r_2:(-\infty,0) \to \C$ be two functions which satisfy properties $(i)$-$(iv)$ of Theorem \ref{r1r2th}.
Then the solution of the RH problem \ref{RH problem for M} is unique, if it exists.
\end{proposition}

The proof of Proposition \ref{uniquenessprop} proceeds through a series of lemma. We assume throughout the appendix that $r_1:(0,\infty) \to \C$ and $r_2:(-\infty,0) \to \C$ satisfy $(i)$-$(iv)$ of Theorem \ref{r1r2th}.

\begin{lemma}\label{expansionatinftylemma}
Suppose that $N$ is a $3 \times 3$ matrix which obeys the symmetries (\ref{singRHsymm}) and which admits an asymptotic expansion
\begin{align}\label{asymp for N at inf}
& N(x,t,k) = N^{(0)}(x,t) + \frac{N^{(1)}(x,t)}{k} + \frac{N^{(2)}(x,t)}{k^2} + O(k^{-3}), \qquad k \to \infty.
\end{align}
Then the coefficient matrices $\{N^{(j)}\}_0^2$ have the form
\begin{align*}
& N^{(0)} = \begin{pmatrix} N_{33}^{(0)} & \overline{N_{13}^{(0)}} & N_{13}^{(0)} \\ 
N_{13}^{(0)} & N_{33}^{(0)} & \overline{N_{13}^{(0)}} \\ 
\overline{N_{13}^{(0)}} & N_{13}^{(0)} &  N_{33}^{(0)}
\end{pmatrix}, 
\qquad N^{(1)} = \begin{pmatrix} \omega^2 N_{33}^{(1)} & \omega \overline{N_{13}^{(1)}} & N_{13}^{(1)} \\ 
\omega^2 N_{13}^{(1)} & \omega N_{33}^{(1)} & \overline{N_{13}^{(1)}} \\ 
\omega^2 \overline{N_{13}^{(1)}} & \omega N_{13}^{(1)} &  N_{33}^{(1)}
\end{pmatrix}, 
	\\
& N^{(2)} = \begin{pmatrix} \omega N_{33}^{(2)} & \omega^2 \overline{N_{13}^{(2)}} & N_{13}^{(2)} \\ 
\omega N_{13}^{(2)} & \omega^2 N_{33}^{(2)} & \overline{N_{13}^{(2)}} \\ 
\omega \overline{N_{13}^{(2)}} & \omega^2 N_{13}^{(2)} &  N_{33}^{(2)}
\end{pmatrix}, 
\end{align*}
where $\{N_{33}^{(j)}\}_0^2$ are real-valued functions and $\{N_{13}^{(j)}\}_0^2$ are complex-valued functions.
\end{lemma}
\begin{proof}
It suffices to combine the large $k$ asymptotics \eqref{asymp for N at inf} with the symmetries (\ref{singRHsymm}). 
After identifying the coefficients of different powers of $k$, we obtain
\begin{align*}
& N^{(0)} =  \mathcal{A}N^{(0)}\mathcal{A}^{-1} = \mathcal{B}\overline{N^{(0)}}\mathcal{B}, 
\qquad
N^{(1)} = \omega^2 \mathcal{A}N^{(1)}\mathcal{A}^{-1} = \mathcal{B}\overline{N^{(1)}}\mathcal{B}, 
	\\
& N^{(2)} = \omega \mathcal{A}N^{(2)}\mathcal{A}^{-1} = \mathcal{B}\overline{N^{(2)}}\mathcal{B}.
\end{align*}
Recalling the definitions (\ref{def of Acal and Bcal}) of $\mathcal{A}$ and $\mathcal{B}$, the lemma follows.
\end{proof}

\begin{lemma}[Asymptotics of $M$ as $k \to \infty$]
If $M$ satisfies the RH problem \ref{RH problem for M}, then, as $k \to \infty$,
\begin{align}\nonumber
M(x,t,k) = & \; I + \frac{M_{33}^{(1)}}{k} \begin{pmatrix} \omega^2 & 0 & 0 \\ 
0 & \omega & 0 \\ 
0 & 0 & 1
\end{pmatrix} 
+ \frac{M_{33}^{(2)}}{k^2}\begin{pmatrix} 
\omega  & 0 & 0 \\ 
0 & \omega^2 & 0 \\ 
0 & 0 &  1
\end{pmatrix}  
	\\ \label{singMatinfty}
&  + \frac{\tilde{M}_{13}^{(2)}}{(1-\omega)k^2}\begin{pmatrix} 
0 & 1 & -1 \\ 
-\omega & 0 & \omega \\ 
\omega^2 & -\omega^2  &  0
\end{pmatrix} + O(k^{-3}),  
\end{align}
\color{black}
where $M_{33}^{(1)}(x,t)$, $M_{33}^{(2)}(x,t)$, and $\tilde{M}_{13}^{(2)}(x,t) := -(1-\omega) M_{13}^{(2)}(x,t)$ are real-valued functions of $x$ and $t$.
\end{lemma}
\begin{proof}
From Lemma \ref{expansionatinftylemma} and the conditions in (\ref{singRHMatinftyb}), we see that, as $k \to \infty$,
\begin{align*}
M(x,t,k) = I + \frac{M_{33}^{(1)}}{k} \begin{pmatrix} \omega^2 & 0 & 0 \\ 
0 & \omega & 0 \\ 
0 & 0 & 1
\end{pmatrix} 
+ \frac{1}{k^2}\begin{pmatrix} \omega M_{33}^{(2)} & \omega^2 \overline{M_{13}^{(2)}} & M_{13}^{(2)} \\ 
\omega M_{13}^{(2)} & \omega^2 M_{33}^{(2)} & \overline{M_{13}^{(2)}} \\ 
\omega \overline{M_{13}^{(2)}} & \omega^2 M_{13}^{(2)} &  M_{33}^{(2)}
\end{pmatrix} + O(k^{-3}).
\end{align*}
In particular, the last condition in (\ref{singRHMatinftyb}) implies that $\omega^2 \overline{M_{13}^{(2)}} + M_{13}^{(2)} = 0$, i.e., that $(1-\omega) M_{13}^{(2)}$ is real-valued.
\end{proof}

\begin{lemma}[Unit determinant of $M$]\label{unitdetlemma}
If $M$ is a solution of RH problem \ref{RH problem for M}, then $M$ has unit determinant. 
\end{lemma}
\begin{proof}
The determinant $\det M$ is analytic in $\C \setminus \{0\}$ and approaches $1$ as $k \to \infty$. The assumption (\ref{singRHMat0}) on the behavior as $k \to 0$ implies (by a direct computation) that $\det M$ has at most a simple pole at $k = 0$. Thus,
$$\det M(x,t,k) = 1 + \frac{f(x,t)}{k},$$
for some function $f(x,t)$. On the other hand, (\ref{singMatinfty}) implies that $\det M(x,t,k) = 1+O(k^{-3})$ as $k \to \infty$ and thus $f(x,t) = 0$.
\end{proof}

\begin{proof}[Proof of Proposition \ref{uniquenessprop}]
Suppose $M$ and $N$ are two solutions of the RH problem \ref{RH problem for M}. By Lemma \ref{unitdetlemma}, $\det M$ and $\det N$ are identically equal to one. In particular, the inverse transpose $N^{A}:=(N^{-1})^T$ of the matrix $N$ can be expressed in terms of its minors according to (\ref{cofactordef}). Expanding this expression for $N^A$ as $k \to 0$ in $D_1$ and using (\ref{singRHMat0}), we find 
\begin{align}\label{N1Aat0}
N_1^A(x,t,k) = 
\frac{1}{k^2} 
\begin{pmatrix}
0 & 0 & \star  \\
0 & 0 & \star \\
0 & 0 & \star 
\end{pmatrix} 
+ \frac{1}{k} \begin{pmatrix}
0  & \star & \star \\
0  & \star & \star \\
0  & \star & \star 
\end{pmatrix} 
+  O(1)
\end{align}
as $k \in D_1$ approaches $0$, where $\star$ denotes an entry whose value is irrelevant for the present argument. Similarly, expanding the expression for $N^A$ as $k \to \infty$ and using (\ref{singMatinfty}), we find 
\begin{align}\nonumber
N^A(x,t,k) = &\; I - 
\frac{N_{33}^{(1)}(x,t)}{k} 
\begin{pmatrix} \omega^2 & 0 & 0 \\ 
0 & \omega & 0 \\ 
0 & 0 & 1
\end{pmatrix}
+ \frac{1}{k^2}\Bigg\{(N_{33}^{(1)}(x,t)^2 - N_{33}^{(2)}(x,t))
\begin{pmatrix} \omega & 0 & 0 \\ 
0 & \omega^2 & 0 \\ 
0 & 0 & 1
\end{pmatrix}
	\\ \label{N1Aatinfty}
& + \frac{\tilde{N}^{(2)}_{13}(x,t)}{1-\omega} \begin{pmatrix} 0 & \omega & -\omega^2 \\ 
-1 & 0 & \omega^2  \\ 
1 & -\omega & 0
\end{pmatrix}\Bigg\}
+ O(k^{-3}), \qquad k\to \infty,
\end{align}
where $\tilde{N}^{(2)}_{13} = -(1-\omega)N^{(2)}_{13}$.
By (\ref{singRHMat0}) and (\ref{N1Aat0}), we have, as $k \in \bar{D}_1$ approaches $0$,
\begin{align}\nonumber
MN^{-1} = &\; \Bigg\{\frac{\alpha}{k^2} 
\begin{pmatrix}
\omega & 0 & 0 \\
\omega & 0 & 0 \\
\omega & 0 & 0
\end{pmatrix} 
+ \frac{1}{k} \begin{pmatrix}
\star  & \star & 0 \\
\star  & \star & 0 \\
\star  & \star & 0 
\end{pmatrix} +  \begin{pmatrix}
\star & \star & \epsilon \\
\star & \star & \epsilon \\
\star & \star & \epsilon 
\end{pmatrix} +  O(k)\Bigg\}
	\\ \label{MNinvat0}
& \times \Bigg\{\frac{1}{k^2} 
\begin{pmatrix}
0 & 0 & 0  \\
0 & 0 & 0\\
\star & \star & \star 
\end{pmatrix} 
+ \frac{1}{k} \begin{pmatrix}
0  & 0 & 0 \\
\star  & \star & \star \\
\star  & \star & \star 
\end{pmatrix} 
+  O(1)\Bigg\}
= O(k^{-2}),
\end{align}
showing that $MN^{-1}$ has at most a double pole at $k = 0$.
Since $MN^{-1}$ is analytic for $k \in \C \setminus \{0\}$ and approaches the identity matrix as $k \to \infty$, we conclude that
\begin{align}\label{MNinvBC}
M(x,t,k)N(x,t,k)^{-1} = I + \frac{P(x,t)}{k} + \frac{Q(x,t)}{k^2}
\end{align}
for some matrices $P(x,t)$ and $Q(x,t)$. In fact, keeping track of the terms of order $O(k^{-2})$ in (\ref{MNinvat0}), we see that 
\begin{align}\label{Q1jQ2jQ3j}
Q_{1j} = Q_{2j} = Q_{3j}, \qquad j = 1,2,3.
\end{align}
Furthermore, the symmetries (\ref{singRHsymm}) hold for $M$ and $N$ and hence also for $MN^{-1}$. Thus, by Lemma \ref{expansionatinftylemma}, $P$ and $Q$ have the form
\begin{align*}
&P = \begin{pmatrix} \omega^2 P_{33} & \omega \overline{P_{13}} & P_{13} \\ 
\omega^2 P_{13} & \omega P_{33} & \overline{P_{13}} \\ 
\omega^2 \overline{P_{13}} & \omega P_{13} &  P_{33}
\end{pmatrix}, 
\qquad
 Q = \begin{pmatrix} \omega Q_{33} & \omega^2 \overline{Q_{13}} & Q_{13} \\ 
\omega Q_{13} & \omega^2 Q_{33} & \overline{Q_{13}} \\ 
\omega \overline{Q_{13}} & \omega^2 Q_{13} &  Q_{33}
\end{pmatrix}, 
\end{align*}
where $P_{33}, Q_{33}$ are real-valued and $P_{13}, Q_{13}$ are complex-valued functions. Together with (\ref{Q1jQ2jQ3j}), this implies that
\begin{align}\label{QQ33pmatrix}
Q = Q_{33} \begin{pmatrix} \omega & \omega^2 & 1 \\ \omega & \omega^2 & 1 \\ \omega & \omega^2 & 1 \end{pmatrix}.
\end{align}

On the other hand, substituting the expansion (\ref{singMatinfty}) of $M$ and the transpose of the expansion (\ref{N1Aatinfty}) of $N^A$ into the left-hand side of (\ref{MNinvBC}) and identifying the coefficients of $k^{-1}$ and $k^{-2}$ in the resulting equation, we conclude that
\begin{align}\label{PMexpression}
P = &\; (M^{(1)}_{33} - N^{(1)}_{33})\begin{pmatrix} \omega^2 & 0 & 0 \\ 
0 & \omega & 0 \\ 
0 & 0 & 1
\end{pmatrix},
	\\\nonumber
Q = &\; \Big(M^{(2)}_{33} - M^{(1)}_{33}N^{(1)}_{33} + (N^{(1)}_{33})^2 - N^{(2)}_{33}\Big)\begin{pmatrix} \omega & 0 & 0 \\ 
0 & \omega^2 & 0 \\ 
0 & 0 & 1
\end{pmatrix}
	\\ \label{QMexpression}
& + \frac{\tilde{M}^{(2)}_{13}(x) - \tilde{N}^{(2)}_{13}(x)}{1-\omega}\begin{pmatrix} 0 & 1 & -1 \\ 
-\omega & 0 & \omega \\ 
\omega^2 & -\omega^2 & 0
\end{pmatrix}.
\end{align}
In order to reconcile the equations (\ref{QMexpression}) and (\ref{QQ33pmatrix}), we must have $Q = 0$, and then (\ref{MNinvBC}) becomes
$$M(x,t,k)N(x,t,k)^{-1} = I + \frac{M^{(1)}_{33}(x,t) - N^{(1)}_{33}(x,t)}{k}\begin{pmatrix} \omega^2 & 0 & 0 \\ 
0 & \omega & 0 \\ 
0 & 0 & 1
\end{pmatrix}.$$
Since the determinant of the left-hand side is identically equal to one, we conclude that $M^{(1)}_{33} - N^{(1)}_{33} = 0$. This shows that $M = N$ and completes the proof of the proposition.
\end{proof}

\subsection*{Acknowledgements}
We are grateful to Percy Deift and J\"orgen \"Ostensson for valuable discussions, and to Sheehan Olver for introducing us to the ``ApproxFun" Julia package. Support is acknowledged from the European Research Council, Grant Agreement No. 682537, the Swedish Research Council, Grant No. 2015-05430, the G\"oran Gustafsson Foundation, and the Ruth and Nils-Erik Stenb\"ack Foundation.

\bibliographystyle{plain}
\bibliography{is}

\end{document}